\documentclass[12pt]{amsart}

\usepackage[left=2.5cm,right=2.5cm,top=2.5cm, bottom=2.5cm]{geometry}

\usepackage{amsmath}
\usepackage{amsxtra}
\usepackage{amscd}
\usepackage{amsthm}
\usepackage{amsfonts}
\usepackage{amssymb}
\usepackage {pstricks}
\usepackage{pstricks,pst-node}
\usepackage{tikz}
\usepackage{tikz-cd}
\usepackage {mathrsfs} 
\usepackage[all,cmtip]{xy}
\usepackage{comment}

\usepackage{hyperref}

\newtheorem{thm}{Theorem}[section]

\newtheorem{lem}[thm]{Lemma}

\newtheorem{cor}[thm]{Corollary}
\newtheorem{prop}[thm]{Proposition}

\newtheorem{theorem}{Theorem}[section]
\newtheorem{lemma}[theorem]{Lemma}
\newtheorem{proposition}[theorem]{Proposition}
\newtheorem{corollary}[theorem]{Corollary}
  
\newtheorem{conjecture}[theorem]{Conjecture}

\theoremstyle{definition}
\newtheorem{example}[theorem]{Example}

\newtheorem{definition}[theorem]{Definition}

\theoremstyle{remark}
\newtheorem{remark}[theorem]{Remark}

\numberwithin{equation}{section}

\begin{document}
	\normalfont
	\newcommand{\thmref}[1]{Theorem~\ref{#1}}
	\newcommand{\secref}[1]{Section~\ref{#1}}
	\newcommand{\lemref}[1]{Lemma~\ref{#1}}
	\newcommand{\propref}[1]{Proposition~\ref{#1}}
	\newcommand{\corref}[1]{Corollary~\ref{#1}}
	\newcommand{\remref}[1]{Remark~\ref{#1}}
	\newcommand{\eqnref}[1]{(\ref{#1})}
	\newcommand{\exref}[1]{Example~\ref{#1}}
	\newcommand{\conjref}[1]{Conjecture~\ref{#1}}
	\newcommand{\rmkref}[1]{Remark~\ref{#1}}
	\newcommand{\defref}[1]{Definition~\ref{#1}}

	\newcommand{\nc}{\newcommand}
	
	\nc{\on}{\operatorname}
	
	\nc{\Z}{{\mathbb Z}}
	\nc{\bQ}{{\mathbb Q}}
	\nc{\C}{{\mathbb C}}
	\nc{\R}{{\mathbb R}}
	\nc{\bbP}{{\mathbb P}}
	\nc{\bF}{{\mathbb F}}
	
	\nc{\boldD}{{\mathbb D}}
	\nc{\oo}{{\mf O}}
	\nc{\N}{{\mathbb N}}
	\nc{\bib}{\bibitem}
	\nc{\pa}{\partial}
	\nc{\F}{{\mf F}}
	\nc{\CA}{{\mathcal A}}
	\nc{\cD}{{\mathcal D}}
	\nc{\CE}{{\mathcal E}}
	\nc{\CP}{{\mathcal P}}
	\nc{\CO}{{\mathcal O}}
	\nc{\CU}{{\mathcal U}}
	\nc{\Res}{\text{Res}}
	\nc{\Ind}{\text{Ind}}
	\nc{\Ker}{\text{Ker}}
	\nc{\id}{\text{id}}
	\nc{\ot}{\otimes}
	
	\nc{\be}{\begin{equation}}
		\nc{\ee}{\end{equation}}
	
	\nc{\rarr}{\rightarrow}
	\nc{\larr}{\longrightarrow}
	\nc{\al}{\alpha}
	\nc{\ri}{\rangle}
	\nc{\lef}{\langle}
	
	\nc{\W}{{\mc W}}
	\nc{\bt}{\bf t}
	\nc{\gam}{\ol{\gamma}}
	\nc{\Q}{\ol{Q}}
	\nc{\q}{\widetilde{Q}}
	\nc{\la}{\lambda}
	\nc{\ep}{\epsilon}
	\nc{\ve}{\varepsilon}
	
	\nc{\g}{\mf g}
	\nc{\h}{\mf h}
	\nc{\n}{\mf n}
	\nc{\bb}{\mf b}
	\nc{\G}{{\mf g}}

	\nc{\D}{\mc D}
	\nc{\cE}{\mc E}
	\nc{\CC}{\mc C}
	\nc{\CH}{\mc H}
	\nc{\CK}{\mc K}
	\nc{\CL}{\mc L}
	\nc{\CT}{\mc T}
	\nc{\CI}{\mc I}
	\nc{\CR}{\mc R}
	
	\nc{\UK}{{\mc U}_{\CA_q}}
	
	\nc{\CS}{\mc S}
	
	\nc{\CB}{\mc B}
	
	\nc{\Li}{{\mc L}}
	\nc{\La}{\Lambda}
	\nc{\is}{{\mathbf i}}
	\nc{\V}{\mf V}
	\nc{\bi}{\bibitem}
	\nc{\NS}{\mf N}
	\nc{\dt}{\mathord{\hbox{${\frac{d}{d t}}$}}}
	\nc{\E}{\mc E}
	\nc{\ba}{\tilde{\pa}}
	\nc{\half}{\frac{1}{2}}
	
	\def\smapdown#1{\big\downarrow\rlap{$\vcenter{\hbox{$\scriptstyle#1$}}$}}
	
	\nc{\mc}{\mathcal}
	\nc{\ov}{\overline}
	\nc{\mf}{\mathfrak}
	\nc{\ol}{\fracline}
	\nc{\el}{\ell}
	\nc{\etabf}{{\bf \eta}}
	\nc{\zetabf}{{\bf
			\zeta}}\nc{\x}{{\bf x}}
	\nc{\xibf}{{\bf \xi}} \nc{\y}{{\bf y}}
	\nc{\WW}{\mc W}
	\nc{\SW}{\mc S \mc W}
	\nc{\sd}{\mc S \mc D}
	\nc{\hsd}{\widehat{\mc S\mc D}}
	\nc{\parth}{\partial_{\theta}}
	\nc{\cwo}{\C[w]^{(1)}}
	\nc{\cwe}{\C[w]^{(0)}} \nc{\wt}{\widetilde}
	\nc{\gl}{\mf gl}
	\nc{\K}{\mf k}
	
	\newcommand{\U}{{\rm{U}}}
	\newcommand{\TL}{{\rm{TL}}}
	\newcommand{\Aut}{{\rm{Aut}}}
	\newcommand{\End}{{\rm{End}}}
	\newcommand{\Hom}{{\rm{Hom}}}
	\newcommand{\Rad}{{\rm{Rad}}}
	\newcommand{\Mod}{{\rm{Mod}}}
	\newcommand{\Ext}{{\rm{Ext}}}
	\newcommand{\Tor}{{\rm{Tor}}}
	\newcommand{\Lie}{{\rm{Lie}}}
	\newcommand{\Uq}{{{\mathcal U}_v}}
	\newcommand{\GL}{{\rm{GL}}}
	\newcommand{\Sym}{{\rm{Sym}}}
	\newcommand{\rk}{{\rm{rk}}}
	\newcommand{\tr}{{\rm{tr}}}
	\newcommand{\Rea}{{\rm{Re}}}
	\newcommand{\rank}{{\rm{rank}}}
	\newcommand{\im}{{\rm{Im}}}
	
	\advance\headheight by 2pt
	
	\nc{\fb}{{\mathfrak b}} \nc{\fg}{{\mathfrak g}}
	\nc{\tA}{\widetilde{A}}
	\nc{\fh}{{\mathfrak h}}  \nc{\fk}{{\mathfrak k}}
	
	\nc{\fl}{{\mathfrak l}} \nc{\fn}{{\mathfrak n}}
	
	\nc{\fp}{{\mathfrak p}} \nc{\fu}{u}
	
	
	\nc{\fS}{{\Sym}}
	
	\nc{\fsl}{{\mathfrak {sl}}} \nc{\fsp}{{\mathfrak {sp}}}
	\nc{\fso}{{\mathfrak {so}}} \nc{\fgl}{{\mathfrak {gl}}}

	\nc{\A}{\mc A} \nc{\cF}{{\mathcal F}}
	
	\nc{\cA}{{\mathcal A}} \nc{\cP}{{\mathcal P}} \nc{\cC}{{\mathcal C}}
	\nc{\cU}{{\mathcal U}} \nc{\cB}{{\mathcal B}}

	\def\lr{{\longrightarrow}}
	\def\ol{\overline}
	\def\inv{{^{-1}}}
	
	\def\xl{{\hbox{\lower 2pt\hbox{$\scriptstyle \mathfrak L$}}}}
	
	\nc{\bX}{{\mathbf X}} \nc{\bx}{{\mathbf x}} \nc{\bd}{{\mathbf d}}
	\nc{\bdim}{{\mathbf dim}} \nc{\bm}{{\mathbf m}}
	
	\title[]{Milnor fibre homology complexes}
	
	\author{Gus Lehrer and Yang Zhang}
	\address{School of Mathematics and Statistics,
The University of Sydney, Sydney, NSW 2006,  Australia}
	\email{gustav.lehrer@sydney.edu.au}

	\address{School of Mathematics and Physics,
The University of Queensland, St Lucia,
Brisbane, QLD 4072, Australia}
	\email{yang.zhang@uq.edu.au}
	\date {}
	\dedicatory{Dedicated to Claudio Procesi, good friend, Italian mathematician}
	
	\begin{abstract}
		Let $W$ be a finite Coxeter group.
		We give an algebraic presentation of what we refer to as ``the non-crossing algebra'', which is associated
		to the hyperplane complement of $W$ and to the cohomology of its Milnor fibre. This is used to produce 
		simpler and more general chain (and cochain) complexes which compute the integral homology and cohomology groups of 
		the Milnor fibre $F$ of $W$. In the process we define a new, larger algebra $\wt A$, which seems to be ``dual'' to the 
		Fomin-Kirillov algebra, and in low ranks is linearly isomorphic to it. There is also a mysterious connection between $\wt A$ 
		and the Orlik-Solomon algebra, in analogy with the fact that the Fomin-Kirillov algebra contains the coinvariant algebra
		of $W$. This analysis is applied to compute the multiplicities $\langle \rho, H^k(F,\C)\rangle_W$ and $\langle \rho, H^k(M,\C)\rangle_W$,
		where $M$ and $F$ are respectively the hyperplane complement and Milnor fibre associated to $W$ and $\rho$ is a representation of $W$.
	\end{abstract}
	\maketitle
	
	This work is an outgrowth of \cite{Zha22}, whose notation we follow, by and large. 
	In particular, $W$ is a finite Coxeter group, $M$ is its corresponding complexified hyperplane complement
	and $F$ is the corresponding (non-reduced) Milnor fibre, as defined in \cite[Def. 1]{DL16} or \cite{Zha22}. Our objective is to 
	construct tractable chain complexes which compute the
	homology of $M$ (which is known for all $W$) and that of $F$ (which is poorly understood, even in the case $W=\Sym_n$). 
	Our approach will indicate a connection with the algebra of Fomin-Kirillov \cite{FK99}. 
	
	In the first two sections, we recall two distinct definitions of the central character in our development, the ``non-crossing algebra'' $\CA$, and prove
	their equivalence. This provides us with many properties of the algebra $\CA$. In addition, we discuss a number of preliminaries we shall require later.
	We then introduce a duality theory, by defining a non-degenerate bilinear form on $\CA$, and this is
	applied to study the integral cohomology of $F$.
	
	Our main purpose here is to show how the algebra $\CA$ and its relatives play a crucial role in determining the cohomolgy of the Milnor fibre $F$. For example, we give several results similar to the following (see Corollary \ref{coro: cohomFW} below). Let $\omega=\sum_{t\in T}a_t\in\CA$; then $\omega^2=0$ in $\CA$, and we prove that both left and right multiplication  by $\omega$ on $\CA$ have the same kernel and image. Moreover we have the following isomorphism of graded abelian groups:
	\be\label{eq:homfw}
	H^*(F/W;\Z)[-1]\cong \frac{\CA\omega\cap\omega\CA}{\omega\CA\omega},
	\ee
	where $[-1]$ on the left means that the $\mathbb{Z}$-grading is shifted by $-1$.
	This result could be compared with those in \cite{DPSS99}, whose ultimate purpose is to compute the left side of the equation \eqref{eq:homfw}. 
	In principle, the stated result reduces the question to a mechanical computation in $\CA$, although in practice, this is not an easy computation.
	
	
	\section{Definitions, notation and preliminaries.}
	\subsection{The noncrossing partition lattice}
	Let $(W,S)$ be a finite Coxeter system of rank $n$ with a geometric representation on the Euclidean space $V:=\mathbb{R}^n$, and let $T=\bigcup_{w\in W}wSw^{-1}$ be the set of reflections of $W$. 
	Denote by $\ell_T(w)$ the number of reflections in a shortest expression for $w$ as a product of reflections, and define a partial order $\leq$ on $W$  by stipulating that $u\leq v$ if and only if $\ell_T(v)= \ell_T(u)+ \ell_{T}(u^{-1}v)$ for $u,v\in W$. Then $(W,\leq)$ is a graded poset whose unique minimal element is the identity $e$ and whose maximal elements are those having no fixed points in $\mathbb{R}^n$, sometimes known as elliptic elements.   
	
	Let $\gamma$ be any Coxeter element, i.e., product of all the simple reflections in some order. 
	\begin{definition}\label{def:l}
		Denote by $\mathcal L$ the closed interval $\mathcal{L}:=[e,\gamma]$ of the poset $(W,\leq)$.
	\end{definition}
	
	Brady and Watt proved that the closed interval $\mathcal{L}$ is a lattice, which we call the noncrossing partition (NCP) lattice \cite{BW08}. 
	The isomorphism type of the NCP lattice is independent of $\gamma$, as all Coxeter elements form a conjugacy class in $W$. 
	Although all Coxeter elements of $W$ are conjugate in $W$, we now define a specific Coxeter element in terms of the root system, which
	has properties we shall find useful later. 
	
	Associated with $W$ we have a set $\Phi$ of vectors in $V$, which form a root system (cf. \cite[Ch. VI, \S 1]{Bou02}), and $S$ determines a simple subsystem of $\Phi$, as well as the corresponding set $\Phi^+$  of positive roots. Write $\Pi=\{\alpha_i \mid  i \in [n]\}$ for  the given simple system. Without loss of generality, we may assume that $W$ is irreducible. Then  $\Pi$ can be written as the disjoint union $\Pi=\Pi_1 \cup \Pi_2$, 
	where $\Pi_1=\{\alpha_{i_1},\dots, \alpha_{i_l}\}$ and $\Pi_2=\{\alpha_{i_l+1}, \dots, \alpha_{i_n}\}$ where the $\alpha_{i_k}\in \Pi_1$ are mutually orthogonal as also are the $\alpha_{i_k}\in \Pi_2$ (see \cite{Ste59}).

	The set of positive roots of $\Phi$ is in bijection with the set of reflections of $W$. 
	Recall that $W$ acts faithfully on the Euclidean space $V:=\mathbb{R}^n$ whose inner product we denote by $(-,-)$. For any positive root $\alpha$ relative to $\Pi$, the corresponding reflection  is defined by
	$t_{\alpha}(x):= x - 2 
	\frac{(\alpha,x)}{(\alpha,\alpha)}\alpha$ for any $x\in \mathbb{R}^n$.
	Throughout, we use the following  Coxeter element
	\begin{equation*}\label{eq:Coxelmt}
		\gamma= (\prod_{\alpha\in \Pi_1} t_{\alpha})(\prod_{\alpha\in \Pi_2} t_{\alpha}),
	\end{equation*}
	unless otherwise stated. Note that the simple reflections $t_{\alpha}$ and $t_{\beta}$ commute whenever $\alpha, \beta\in \Pi_1$ or $\alpha, \beta\in \Pi_2$.

	Now we define a total order on the set of positive roots. Let $h$ be the Coxeter number, i.e. the order of $\gamma$. Then the number of positive roots is $nh/2$. 
	It is proved in \cite[Theorem 6.3]{Ste59} that the positive roots $\rho_k$ of $\Phi$ relative to $\Pi$ can be produced successively using the following formulae 
	\begin{equation}\label{stein}
		\rho_k= \begin{cases}
			\alpha_{i_k}, \quad  & 1\leq k\leq l,\\
			-\gamma(\alpha_{i_k}), \quad &l+1\leq k\leq n,\\
			\gamma(\rho_{k-n}), \quad &n+1\leq k\leq  \frac{nh}{2}. 
		\end{cases}
	\end{equation}
	This yields a total order $\preceq$ on the set $T$ of reflections
	\begin{equation}\label{eq: total ordering}
		t_{\rho_1}\prec t_{\rho_2}\prec \dots \prec t_{\rho_{nh/2}}.
	\end{equation}
	
	The  total order $\preceq$ on $T$ gives rise to an EL-labelling (see \cite{ABW07}) of $\mathcal{L}$. Denote by  $\mathcal{E}(\mathcal{L})$ the set of covering relations 
	$u\lessdot v$ of $\mathcal{L}$, that is, relations where there is no third element between $u$ and $v$. Then we have  a natural  edge labelling 
	\begin{equation*}\label{eq: edge-labelling}
		\lambda: \mathcal{E}(\mathcal{L}) \rightarrow T, \quad u\lessdot v \mapsto u^{-1}v. 
	\end{equation*}
	Let  ${\bf c}: x=w_0< w_1< \dots < w_k=y$ be a maximal chain of any closed interval $[x,y]$ of $\mathcal{L}$. We may identify ${\bf c}$ with its  labelling sequence $\lambda({\bf c}):=(w_0^{-1}w_1, \dots, w_{k-1}^{-1}w_k)$, where $w_{i-1}^{-1}w_i\in T$ for $1\leq i\leq k$.
	It has been proved that $\lambda$ is an EL-labelling \cite{ABW07,Bjo}, which means that for every  interval $[x,y]$ of $\mathcal{L}$ 
	\begin{enumerate}
		\item there is a unique increasing  maximal chain in $[x,y]$,  and
		\item this chain is lexicographically smallest among all maximal chains in $[x,y]$.
	\end{enumerate}
	As  $\mathcal{L}$ has an EL-labelling, it is  Cohen-Macaulay \cite[Theorem 2.3]{Bjo}, i.e. for any $u<v$ of $\mathcal{L}$ we have
	\[
	\widetilde{H}_i(u,v)=0, \quad \forall i \neq \ell_T(v)-\ell_T(u)-2,
	\]
	where $\wt H$ denotes reduced homology of the order complex of $(u,v)$. 
	
	If $W$ is a finite Coxeter group, then $\mathcal{L}$ is a direct product of the NCP lattices $\mathcal{L}(W_i)$ over the irreducible components $W_i$ of $W$.
	It is a result of \cite{Bjo} that the EL-labelling is preserved under the direct product. Moreover, any closed interval $[e,w]$ of $\mathcal{L}$ 
	has an EL-labelling given by the natural labelling $\lambda$ restricted to $[e,w]$.
	
	For any $w\in \mathcal{L}$, denote
	\[ {\rm Rex}_{T}(w):=\{(t_1,t_2,\dots, t_k)\,|\, w=t_1t_2\dots t_k \text{\, is $T$-reduced} \}.\]
   With respect to the total order \eqnref{eq: total ordering} on $T$, we define
	\begin{equation}\label{eq: decw}
		\mathcal{D}_{w}:=\{ (t_1,\dots t_{k}) \in {\rm Rex}_{T}(w) \mid   t_1\succ t_2\succ\dots\succ t_k \}.
	\end{equation}
	In words, $\mathcal{D}_{w}$ is the set of decreasing labelling sequences for the maximal chains $e<t_1<t_1t_2<\dots <t_1t_2\dots t_k=w$ of $[e,w]$. Note that any interval $[u,v]$ of $\mathcal{L}$ is isomorphic to $[e, u^{-1}v]$ as posets. The following result can be found in \cite[Proposition 2.2]{Zha22}.
	
	\begin{proposition}\label{prop: dimtop} 
		For any $w\in \mathcal{L}$ with $1\leq \ell_{T}(w)=k\leq n$, we have
		\[{\rm rank}\, \widetilde{H}_{k-2}(e,w)= (-1)^{k}\mu(w)=|\mathcal{D}_w|, \]
		where $\mu$ is the M\"obius function of $\mathcal{L}$.
	\end{proposition}
	
	Any interval $[e,w]$ of $\mathcal{L}$ has the following important interpretation, due to Bessis \cite[Lemma 1.4.3]{Bes03} (see also \cite[Proposition 2.6.11]{Arm09}).
	
	\begin{proposition}\label{prop: intncp}
		Let $\gamma\in W$ be a Coxeter element and ${\rm NC}(W,\gamma)$ the noncrossing partition lattice relative to $\gamma$. 
		For any $w\leq \gamma$, the interval $[e,w]$ of ${\rm NC}(W,\gamma)$  is isomorphic to ${\rm NC}(W^{\prime},w)$ for some parabolic subgroup $W^{\prime}$ of $W$. 
	\end{proposition}

	\subsection{The algebra \texorpdfstring{$\mathcal{A}$}{A}}
	We now give an algebraic definition of the noncrossing algebra in terms of generators and relations. 
	\begin{definition}\label{def: A}
		Let $\mathcal{L}:=[e,\gamma]$ be the noncrossing partition lattice associated  to a finite Coxeter group $W$ and a Coxeter element $\gamma\in W$. We define the   noncrossing algebra $\mathcal{A}=\mathcal{A}(W,\gamma)$  to be the graded algebra over $\mathbb{Z}$ generated by homogeneous elements $a_{t}, t\in T$ of degree $1$, subject to the following quadratic relations:
		\begin{enumerate}
			\item  $a_t^2= a_{t_1}a_{t_2}=0$ for any $t\in T$ 
			and $t_1, t_2 \in T$ with $t_1t_2 \not \leq \gamma$;
			
			\item  $\sum_{(t_1,t_2)\in {\rm Rex}_{T}(w)}a_{t_1} a_{t_2}=0$ for any  $w\in \mathcal{L}$ with $\ell_{T}(w)=2$.
		\end{enumerate}
	\end{definition}
	
	Different choices of the Coxeter element produce isomorphic noncrossing algebras, as the  Coxeter elements are all conjugate to each other and the absolute length is invariant under the conjugation. More precisely, if $\gamma^{\prime}=w\gamma w^{-1}$ for some $w\in W$, the map $a_{t}\mapsto a_{wtw^{-1}}$ extends to an algebra isomorphism between $\mathcal{A}(W,\gamma)$ and $\mathcal{A}(W,\gamma^{\prime})$.  
	
	\begin{example}\label{exam: dih} (Dihedral group)
		Let $I_2(m)$ be the diheral group defined by 
		\[ I_{2}(m):=\langle s_1, s_2 \mid s_1^2=s_2^2=(s_2s_1)^m=1\rangle, \quad m\geq 3.  \]
		Let $\gamma=s_1s_2$ be a Coxeter element and let $t_i=s_1(s_2s_1)^{i-1}$ be reflections for $1\leq i\leq m$. Then $\gamma=t_1t_m=t_2t_1=\dots =t_mt_{m-1}$ has $m$ reduced reflection factorisations. The noncrossing algebra $\mathcal{A}(I_2(m),\gamma)$ is generated by $a_{t_i}, 1\leq i\leq m$ subject to the following relations
		\begin{equation*}
			\begin{aligned}
				a_{t_1}a_{t_i}&=0, \quad 1\leq i\leq m-1,\\
				a_{t_i}a_{t_j}&=0, \quad 2\leq i \leq m, 1\leq j \neq i-1 \leq m, \\
				a_{t_1} a_{t_m}&+a_{t_2} a_{t_1}+\dots +a_{t_m} a_{t_{m-1}}=0. 
			\end{aligned}
		\end{equation*}
		
		Because it is relevant for the proof of Proposition \ref{prop: Abasis}, we set out how the above constructions
		apply to this case. Note first that, using the relation $\gamma t_i\gamma\inv=t_{i+2}$, where
		the index is taken modulo $m$,
		the construction \eqref{eq: total ordering} leads to the following ordering on the reflections $t_i$:
		\begin{equation*}\label{eq: to-dihedral}
			t_{1}(=s_1)\prec t_{2}\prec \dots \prec t_{m}(=s_2).
		\end{equation*}
		It follows, taking into account the reduced reflection factorisations of $\gamma$ given above, that the unique increasing factorisation is 
		$\gamma=t_1t_m=s_1s_2$, and all other factorisations are decreasing.
	\end{example}
	

	
	\begin{example} (Type $A_n$)
		Let $W={\rm Sym}_{n+1}$ be the symmetric group generated by the elementary transpositions $(i,i+1), 1\leq i\leq n$. Let $\gamma=(1,2,\dots, n+1)$ be a Coxeter element. We write $a_{ij}:=a_{(i,j)}$. Then $\mathcal{A}(W,\gamma)$ of type $A$ is the graded $\mathbb{Z}$-algebra  generated by $a_{ij}, 1\leq i<j\leq n+1$ subject to the following relations:
		\begin{equation*}
			\begin{aligned}
				a_{ij}^2&=0, &\quad& 1\leq i<j\leq n+1,\\
				a_{ik}a_{jl}&=0, &\quad& 1\leq i<j<k<l\leq n+1, \\
				a_{ij}a_{ik}=a_{jk}&a_{ij}=a_{ik}a_{jk}=0, &\quad& 1\leq i<j<k\leq n+1,\\
				a_{ij}a_{kl}+a_{kl}a_{ij}&=a_{il}a_{jk}+a_{jk}a_{il}=0,  &\quad& 1\leq i<j<k<l\leq n+1,\\
				a_{ij}a_{jk}+a_{jk}&a_{ik}+a_{ik}a_{ij}=0, &\quad& 1\leq i<j<k\leq n+1.
			\end{aligned} 
		\end{equation*}
	\end{example}

	\subsection{The algebra \texorpdfstring{$\mathcal{B}$}{B}}
	We next give what later will turn out to be a combinatorial definition of the noncrossing algebra above, as introduced in \cite{Zha22}.
	
	For each $k=0, \dots, n$, let  $ \mathcal{L}_k:=\{ w\in \mathcal{L} \mid \ell_{T}(w)=k \}$. Denote by $C_{k-1}(w)$ the abelian group which is freely spanned by all  sequences  of ${\rm Rex}_T(w)$. Let $B_k$ denote the $k$-string braid group,
	with standard generators $\sigma_i, 1\leq i\leq k-1$ subject to the relations $	\sigma_i\sigma_j=\sigma_j\sigma_i$ for $|i-j|\geq 2$, and 
	$\sigma_i\sigma_{i+1}\sigma_i=\sigma_{i+1}\sigma_i\sigma_{i+1}$ for $ 1\leq i\leq k-2$. 
	The Hurwitz action of $B_k$ on $C_{k-1}(w)$  is defined  by 
	\begin{equation*}\label{eq: Braidact}
		\sigma_i.(t_1, \dots, t_{i-1},t_i, t_{i+1}, \dots, t_k):=(t_1,\dots, t_{i-1}, t_{i+1}, t_{i}^{t_{i+1}},\dots, t_k), \quad 1\leq i\leq k-1.
	\end{equation*}
	
	Let ${\rm Sym}_k$ be the symmetric group on $k$ letters with standard generators $s_i=(i,i+1), 1\leq i\leq k-1$. There is  a set-theoretic lift map:
	\begin{equation*}
		\varphi:  {\rm Sym}_k \rightarrow B_k,\quad \pi=s_{i_1} \dots s_{i_p}\mapsto \underline{\pi}:= \sigma_{i_1}\dots \sigma_{i_p},
	\end{equation*}
	where $\pi=s_{i_1} \dots s_{i_p}\in {\rm Sym}_k$ is a reduced expression in the standard generators. This is independent of the choice of reduced expression of $\pi$. 
	For any ${\bf t}=(t_1,t_2,\dots,t_k)\in {\rm Rex}_{T}(w)$, we  define the following $\mathbb{Z}$-linear map using the above Hurwitz action:
	\begin{equation*}\label{eq: beta}
		\beta: C_{k-1}(w) \rightarrow C_{k-1}(w),\quad {\bf t}\mapsto \beta_{{\bf t}}:=\sum_{\pi\in {\rm Sym}_k} {\rm sgn}(\pi)\, \underline{\pi}.(t_1,t_2,\dots,t_k),
	\end{equation*} 
	where ${\rm sgn}$ is the usual sign character  of ${\rm Sym}_k$. The element $\beta_{{\bf t}}$ is viewed as an alternating sum of  maximal chains of the  interval $(e,w]$, with each sequence $(t_1,t_2,\dots, t_k)$ in the sum identified  with the following chain
	\begin{equation*}\label{eq: chain}
	   t_1<t_1t_2<\dots <t_1t_2\dots t_k=w.
	\end{equation*}
	
	\begin{definition}\label{def: AlgB}
		For each $w\in \mathcal{L}_k$, we define $\mathcal{B}_w$ to be the abelian group spanned by the elements $\beta_{{\bf t}}$ of ${\rm {Rex}}_T(w)$, that is, 
		\[ \mathcal{B}_{w}:={\rm Im}\, \beta=\sum_{{\bf t}\in {\rm Rex}_{T}(w)} \mathbb{Z}\beta_{{\bf t}} \subseteq C_{k-1}(w).\]
		In particular $\mathcal{B}_e:=\mathbb{Z}$. Further, we write $\mathcal{B}:= \bigoplus_{w\in \mathcal{L}}\mathcal{B}_w$. 
	\end{definition}
	
	We point out a close connection between $\mathcal{B}$ and the homology of  $\mathcal{L}$. 
	For any $w\in \mathcal{L}_k$, let 
	$\widetilde{H}_{k-2}(e,w)$ be the top reduced homology group of the open interval $(e,w)$. 
	Define  $C_{k-1}$ to be the  abelian group freely spanned by the basis $\bigcup_{w\in \mathcal{L}_k} {\rm Rex}_T(w)$. Then $C_{k-1}=\bigoplus_{w\in \mathcal{L}_k} C_{k-1}(w)$.
	Define the linear truncation  $d_{k-1}$ by
	\begin{equation}\label{eq: theta}
		d_{k-1}: C_{k-1}\rightarrow C_{k-2}, \quad (t_1,t_2,\dots,t_{k-1},t_k)\mapsto (t_1,t_2,\dots, t_{k-1}), \quad k\geq 2,    
	\end{equation}
	and  $d_0(t)=1, \forall t\in T$. We will write $d=d_k$ whenever no confusion arises. As $\mathcal{B}_w\subseteq C_{k-1}(w)$, we may restrict $d$ to  $\mathcal{B}_w$ and  define
	\begin{equation*}\label{eq: defz}
		z_{{\bf t}}:=d(\beta_{\bf t}), \quad \forall {\bf t}\in {\rm Rex}_T(w).
	\end{equation*}
	
	\begin{proposition}\label{prop: cycle} \cite[Proposition 3.5]{Zha22}
		Let $w\in \mathcal{L}_{k}$ with $k\geq 1$. Then  for any ${\bf t}\in {\rm Rex}_{T}(w)$, we have
		\[
		z_{{\bf t}}=\sum_{i=1}^k (-1)^{k-i} \beta_{{\bf t}(\hat{i})}\in \widetilde{H}_{k-2}(e,w),
		\]
		where ${\bf t}(\hat{i}):=(t_1,\dots, \hat{t}_i,\dots,t_k)$ is obtained by removing the $i$-th entry of ${\bf t}$.
	\end{proposition}
	
	\begin{theorem}\label{thm: homint} \cite[Theorem 4.5]{Zha22}
		For any $w\in \mathcal{L}$, let $\mathcal{B}_{w}$ and $\mathcal{D}_w$ be as defined in \defref{def: AlgB} and \eqref{eq: decw}, respectively. 
		\begin{enumerate}
			\item The elements  $\beta_{{\bf t}}, {\bf t}\in \mathcal{D}_w$ constitute a $\mathbb{Z}$-basis for $ \mathcal{B}_w$;
			\item  The elements $z_{{\bf t}}, {\bf t}\in \mathcal{D}_w$ are a $\mathbb{Z}$-basis for $ \widetilde{H}_{k-2}(e,w)$, where $k=\ell_T(w)\geq 1$; 
			\item The $\mathbb{Z}$-linear map 
			\begin{equation*}\label{eq: Bwiso}
				d: \mathcal{B}_w \rightarrow \widetilde{H}_{k-2}(e,w), \quad \beta_{{\bf t}} \mapsto z_{{\bf t}}, \, {\bf t}\in {\rm Rex}_{T}(w)
			\end{equation*}
			is an isomorphism of free abelian groups.
		\end{enumerate}
	\end{theorem}

	We now define a multiplicative structure which makes $\mathcal{B}$ into a finite dimensional algebra. For any $\beta_{{\bf t}} \in \mathcal{B}_w$ and $\beta_{t}, t\in T$, define 
	\[
	\beta_{{\bf t}} \beta_{t}:=
	\begin{cases}
		\beta_{({\bf t},t)}, & \text{if } wt\leq \gamma \text{\, and \,}\ell_{T}(wt)>\ell_{T}(w) ,\\
		0,  &\text{otherwise,}
	\end{cases}
	\]
	where $({\bf t}, t)$ denotes  the concatenation of ${\bf t}$ and $t$. Clearly, we have $\beta_{{\bf t}}=\beta_{t_1}\beta_{t_2}\dots \beta_{t_k}$ for any $T$-reduced expression $w=t_1t_2\dots t_k\in \mathcal{L}$. Therefore, $\mathcal{B}$  is a finite-dimensional $\mathbb{Z}$-graded  algebra   generated by homogeneous elements $\beta_{t}, t\in T$ of degree $1$.

	\begin{proposition}\label{prop: quadrel} \cite[Proposition 5.6]{Zha22}
		We have the following quadratic relations in $\CB$:
		\begin{enumerate}
			\item $\beta_t^2= \beta_{t_1}\beta_{t_2}=0$   for all $t\in T$ and $t_1,t_2\in T$ with  $t_1t_2 \not \leq  \gamma$.
			\item  For any  $w\in \mathcal{L}_2$, we have 
			\begin{equation*}\label{eq: quadrel2}
				\sum_{(t_1,t_2)\in {\rm Rex}_{T}(w)} \beta_{t_1}\beta_{t_2}=0.
			\end{equation*}
		\end{enumerate}
	\end{proposition}

	\section{An isomorphism between \texorpdfstring{$\mathcal{A}$}{A} and \texorpdfstring{$\mathcal{B}$}{B}}
	
	\subsection{The isomorphism}
	The main theorem of this section is the following.
	\begin{theorem}\label{thm: isothm}
		Let $W$ be a finite Coxeter group and let $T$ be the set of reflections of $W$.
		The assignment $a_{t} \mapsto \beta_{t}$ for all $t\in T$ extends to a graded algebra isomorphism
		$\mathcal{A}\cong \mathcal{B}$.
	\end{theorem}
	
	The remainder of this section is devoted to proving \thmref{thm: isothm}.

	We begin with a sketch of the proof. 
	By \propref{prop: quadrel}, the assignment $a_{t}\mapsto \beta_{t}$ preserves the quadratic relations of $\mathcal{A}$. Since $\mathcal{B}$ is a finite-dimensional algebra generated by $\beta_{t}$, we obtain a surjective algebra homomorphism 
	\begin{equation}\label{eq: mapAB}
		\phi: \mathcal{A}\rightarrow \mathcal{B}, \quad a_{t}\mapsto \beta_{t}, \quad t\in T
	\end{equation}
	from $\mathcal{A}$ to $\mathcal{B}$. To see that this is indeed an isomorphism, we will prove that $\mathcal{A}$ is finite-dimensional and then show that ${\rm rank}\, \mathcal{A} = {\rm rank}\, \mathcal{B}$.  It will follow that $\mathcal{A}$ and $\mathcal{B}$ are isomorphic.
	
	To show $\mathcal{A}$ is finite dimensional, it is crucial to find  necessary and sufficient conditions to ensure that $a_{t_1}a_{t_2}\dots a_{t_k}=0$.  We need the following key lemma. 
	
	\begin{lemma}\label{lem: vanishing} (Vanishing property) Let $t_i\in T$  for $1\leq i\leq k$.
		\begin{enumerate}
			\item We have $a_{t_1}a_{t_2}\dots a_{t_k} =0$ if $t_1t_2\dots t_k$  is not $T$-reduced.
			\item Let $t_1t_2\dots t_k$  be a $T$-reduced expression. Then 
			\[ a_{t_1}a_{t_2}\dots a_{t_k} =0, \quad \text{if $t_1t_2\dots t_k \not \leq \gamma$.} \]
		\end{enumerate}
	\end{lemma}
	The proof of this lemma is postponed to \secref{sec: pfvanish}.
	
	\begin{lemma}\label{coro: nonzero}
		The element $a_{t_1}a_{t_2}\dots a_{t_k}\neq 0$  if and only if
		$t_1t_2\dots t_k$ is $T$-reduced and $t_1t_2\dots t_k\leq \gamma $.
	\end{lemma}
	\begin{proof}
		The ``only if" part is evident from \lemref{lem: vanishing}. For the ``if" part, note that $$\phi(a_{t_1}a_{t_2}\dots a_{t_k})= \beta_{t_1}\beta_{t_2}\dots \beta_{t_k}=\beta_{{\bf t}},$$ where ${\bf t}=(t_1,\dots, t_k)$ is the labelling sequence of the chain $e< t_1 < t_1t_2 < \dots < t_1t_2\dots t_k$ of $\mathcal{L}$. By our construction $\beta_{{\bf t}}\neq 0$ and hence $a_{t_1}a_{t_2}\dots a_{t_k}\neq 0$.
	\end{proof}
	
	Recall that the algebra $\mathcal{A}$ has a natural $\mathbb{Z}$-grading with ${\rm deg}(a_{t})=1$ for any $t\in T$. Let $\mathcal{A}_k$ denote the $k$-th graded component of $\mathcal{A}$. For any $w\in \mathcal{L}_k$, 
	let $\CA_w$  be the abelian subgroup of $\mathcal{A}_k$ given by
	\[ \mathcal{A}_{w}:= {\rm Span}_{\mathbb{Z}}\{ a_{{\bf t}}:=a_{t_1}a_{t_2}\dots a_{t_k}\mid {\bf t}\in {\rm Rex}_{T}(w) \}  \]
	By convention we set $\mathcal{A}_e: =\mathbb{Z}$.

	\begin{lemma} \label{lem: property}
		Maintain the notation above. We have:
		\begin{enumerate}
			\item $\mathcal{A}_k=0$ for $k>n =\ell_{T}(\gamma)$.
			\item $\mathcal{A}_k=\bigoplus_{w\in \mathcal{L}_{k}} \mathcal{A}_w $ for $0\leq k\leq n=\ell_{T}(\gamma)$.
		\end{enumerate}
	\end{lemma}
	\begin{proof}
		As the maximal rank of the poset $(W,\leq)$ is $n$, any expression $t_{1}\dots t_k$ with $k>n$ is not $T$-reduced. It follows from \lemref{coro: nonzero} that $\mathcal{A}_k=0$ for $k>n$ and $\mathcal{A}_k= \sum_{w\in \mathcal{L}_k} \mathcal{A}_w$ for $0\leq k\leq n$. The surjective homomorphism  \eqref{eq: mapAB} satisfies $\phi(\mathcal{A}_w)=\mathcal{B}_w$ and hence $\phi(\sum_{w\in \mathcal{L}_k} \mathcal{A}_w)= \bigoplus_{w\in \mathcal{L}_k} \mathcal{B}_w$. Therefore, the sum $\mathcal{A}_k= \sum_{w\in \mathcal{L}_k} \mathcal{A}_w$ is a direct sum. 
	\end{proof}
	
	\begin{proposition}\label{prop: Abasis}
		For any $w\in \mathcal{L}$, let $\mathcal{D}_w$ be as in \eqref{eq: decw}.  Then  the  elements $a_{{\bf t}}, {\bf t}\in \mathcal{D}_w $ form a $\mathbb{Z}$-basis for $\mathcal{A}_w$.
	\end{proposition}
	\begin{proof}
		
		Since $\phi(a_{{\bf t}})=\beta_{{\bf t}}$, 
		the set $\{a_{{\bf t}}\mid {\bf t}\in \mathcal{D}_w \}$ is $\mathbb{Z}$-linearly independent by \thmref{thm: homint}. By \propref{prop: intncp},  without loss of generality, we
		may assume $w=\gamma$. It remains to show that  every element of $\mathcal{A}_{\gamma}$ is a $\mathbb{Z}$-linear combination of the decreasing elements $a_{{\bf t}}, {\bf t}\in \mathcal{D}_\gamma$.
		
		We use  induction on $\ell_{T}(\gamma)=n$. If $n=1$, then there is nothing to prove.  By part (2) of \lemref{lem: property} we have $\mathcal{A}_{n-1}=\sum_{w\in \mathcal{L}_{n-1}}\mathcal{A}_w$. For $n>1$,  we may by induction assume  that  any $a_{t_1} a_{t_2}\dots a_{t_{n-1}} \in \mathcal{A}_{n-1}$ can be expressed as a $\mathbb{Z}$-linear combination of the  elements $a_{{\bf t}}$ for ${\bf t}\in \bigcup_{w\in \mathcal{L}_{n-1}} \mathcal{D}_w$.
		
		Consider the following filtration of  $\mathcal{A}_n=\mathcal{A}_{\gamma}$. Recall that   $T $ is totally ordered as in  \eqref{eq: total ordering}. For each reflection $t_{\rho_i}\in T$,  define to be  $V_{\rho_i}$ be the abelian subgroup of $\mathcal{A}_{\gamma}$ spanned by the elements $a_{t_1}\dots a_{t_{n-1}} a_{t_{\rho_i}}$ for all $(t_1,\dots, t_{n-1})\in {\rm Rex}_{T}(\gamma t_{\rho_i})$. Then we have a filtration
		\begin{equation*}\label{eq: filtration}
			0\subseteq V_{\rho_1} \subseteq\dots \subseteq  \sum_{i=1}^s V_{\rho_i}  \subseteq \dots \subseteq \sum_{i=1}^{hn/2} V_{\rho_i}=\mathcal{A}_{\gamma}.
		\end{equation*}
		
		We use induction on $s$ to show that for each $s$ with $1\leq s\leq \frac{hn}{2}$, $ \sum_{i=1}^s V_{\rho_i}$ is spanned by the elements $a_{{\bf t}}, {\bf t}\in \mathcal{D}_\gamma$. If $s=1$, then  by the minimality of $t_{\rho_1}$ in $T$ and the induction hypothesis on $n$, any element in $V_{\rho_1}$ can be written as a $\mathbb{Z}$-linear combination of decreasing elements $a_{t_1}\dots a_{t_{n-1}} a_{t_{\rho_1}}$ for all $(t_1,\dots, t_{n-1}, t_{\rho_1})\in \mathcal{D}_{\gamma}$.
		
		Now assume $s>1$.  For any $a\in V_{\rho_s}$, by the induction hypothesis on $n$ there exist $\lambda_{{\bf t}}\in \mathbb{Z}$ such that
		\[ a=\sum_{{\bf t}\in \mathcal{D}_{\gamma t_{\rho_s}}} \lambda_{{\bf t}} a_{{\bf t}} a_{t_{\rho_s}}=\sum_{{\bf t}\in \mathcal{D}_{\gamma t_{\rho_s}}^{\succ}} \lambda_{{\bf t}} a_{{\bf t}} a_{t_{\rho_s}} +\sum_{{\bf t}\in \mathcal{D}_{\gamma t_{\rho_s}}^{\prec }} \lambda_{{\bf t}} a_{{\bf t}} a_{t_{\rho_s}},\]
		where
		\[ \begin{aligned}
			\mathcal{D}_{\gamma t_{\rho_s}}^{\succ}&=\{ (t_1,\dots t_{n-1})\in \mathcal{D}_{\gamma t_{\rho_s}} \mid  t_{n-1}\succ t_{\rho_s} \},\\
			\mathcal{D}_{\gamma t_{\rho_s}}^{\prec}&=\{ (t_1,\dots t_{n-1})\in \mathcal{D}_{\gamma t_{\rho_s}} \mid  t_{n-1}\prec t_{\rho_s} \} .
		\end{aligned}\]
		Note that $a_{{\bf t}}a_{t_{\rho_s}}$ is a decreasing element for any ${\bf t}\in \mathcal{D}_{\gamma t_{\rho_s}}^{\succ}$. We claim that
		\begin{equation}\label{eq: cont}
			a_{{\bf t}}a_{t_{\rho_s}} \in V_{\rho_1}+V_{\rho_2}+\dots +V_{\rho_{s-1}}, \quad  \forall {\bf t}\in   \mathcal{D}_{\gamma t_{\rho_s}}^{\prec}.
		\end{equation}
		Given \eqref{eq: cont},  by the induction hypothesis on $s$,  any element $a\in V_{\rho_s}$ is a $\mathbb{Z}$-linear combination of decreasing elements $a_{{\bf t}}, {\bf t}\in \mathcal{D}_{\gamma}$. Therefore,  
		$ \sum_{i=1}^s V_{\rho_i}$ is spanned by the elements $a_{{\bf t}}, {\bf t}\in \mathcal{D}_\gamma$ for any positive integer $s$. In particular,  $\mathcal{A}_{\gamma}=\sum_{i=1}^{hn/2} V_{\rho_i}$ is spanned by the decreasing elements $a_{{\bf t}}, {\bf t}\in \mathcal{D}_\gamma$.
		
		
		It remains to prove \eqref{eq: cont}. Take any ${\bf t}=(t_1, \dots, t_{n-2}, t_{n-1})\in \mathcal{D}_{\gamma t_{\rho_s}}^{\prec}$ with $t_{n-1}\prec t_{\rho_s}$. Let $u= t_{n-1}t_{\rho_s}$. Then  there exists a poset isomorphism between  $[e, u]$ and  $[t_1\dots t_{n-2}, \gamma]$ which sends $x\in [e, u]$ to $t_1\dots t_{n-2}x\in  [t_1\dots t_{n-2}, \gamma]$. In particular, this isomorphism preserves the EL-labelling.

		By \propref{prop: intncp}, the interval $[e,u]$ is the noncrossing partition lattice of a dihedral group $I_{2}(m)$ for some integer $m\geq 2$. Let    
		$t_{\tau_1}\prec t_{\tau_2}\prec \dots \prec t_{\tau_m}$
		be the reflections of $[e,u]$ with the total order inherited from \eqref{eq: total ordering}.  
		By the discussion in Example \ref{exam: dih}  the unique increasing maximal chain of $[e,u]$ is labelled by $(t_{n-1},t_{\rho_s})$ and
		we have  $t_{\tau_1}=t_{n-1}$ and $t_{\tau_m}=t_{\rho_s} $. The other maximal chains in $[e,u]$ are all decreasing. Further, using the defining relation of $\mathcal{A}$, we have
		\begin{equation}\label{eq: arel}
			a_{t_{n-1}}a_{t_{\rho_s}}=a_{t_{\tau_1}}a_{t_{\tau_m}}= -\sum_{t_{\tau_i}\succ t_{\tau_j}} a_{t_{\tau_i}} a_{t_{\tau_j}}, 
		\end{equation}
		where the sum is over all decreasing labelling sequences of maximal chains in $[e,u]$. For any pair $t_{\tau_i}\succ t_{\tau_j}$ we have 
		\begin{equation}\label{eq: uorder}
			t_{\tau_j}\preceq t_{\tau_{m-1}}\preceq t_{\rho_{s-1}} \prec t_{\rho_s}=t_{\tau_m}.
		\end{equation}
		Combining \eqref{eq: arel} and \eqref{eq: uorder}, we obtain
		\[ a_{{\bf t}}a_{t_{\rho_s}}=a_{t_1}\dots a_{t_{n-2}}(a_{t_{n-1}}a_{t_{\rho_s}})=-\sum_{t_{\tau_i}\succ t_{\tau_j}} a_{t_1}\dots a_{t_{n-2}} a_{t_{\tau_i}} a_{t_{\tau_j}} \in \sum_{i=1}^{s-1}V_{\rho_i}. \]
		The statement \eqref{eq: cont} follows, and the proof of Proposition \ref{prop: Abasis} is complete.
	\end{proof}
	
	The following is an immediate consequence of \lemref{lem: property} and \propref{prop: Abasis}.
	\begin{corollary}\label{coro: basis}
		The set $\{a_{{\bf t}}\mid {\bf t}\in \bigcup_{w\in \mathcal{L}} \mathcal{D}_w,\; \bf t\text{ decreasing} \}$  is a $(\Z)$-basis of $\mathcal{A}$. 
	\end{corollary}
	
	We are now in a position  to prove \thmref{thm: isothm}.
	\begin{proof}[Proof of Theorem~\ref{thm: isothm}]
		The map $\phi: \mathcal{A}\rightarrow \mathcal{B}$ defined by $\phi(a_{\bf t})=\beta_{\bf t}$ extends to a surjective algebra homomorphism. By \corref{coro: basis} the algebra $\mathcal{A}$ has a basis consisting of decreasing elements $a_{{\bf t}}$ with ${\bf t}\in \bigcup_{w\in \mathcal{L}} \mathcal{D}_w$. It follows from \thmref{thm: homint} that  $\phi$ is injective and hence  $\mathcal{A}$ is isomorphic to $\mathcal{B}$. 
	\end{proof}
	
	\begin{proposition}\label{prop: Apro}
		The algebra $\mathcal{A}$ enjoys the following properties.
		\begin{enumerate}
			\item   Let  $\mathcal{L}$ and $\mathcal{L}^{\prime}$ be two noncrossing partition lattices. Then as free abelian groups,   
			$$  \mathcal{A}(\mathcal{L}\times \mathcal{L}^{\prime}) \cong \mathcal{A}(\mathcal{L})\otimes \mathcal{A}(\mathcal{L}^{\prime}).$$
			\item    Let $\mathcal{L}_w=[e,w]$ be a closed interval in $\mathcal{L}$. Then the inclusion $i: \mathcal{L}_w \hookrightarrow \mathcal{L}$ of posets induces an injective homomorphism $i_{\mathcal{A}}: \mathcal{A}(\mathcal{L}_w)\rightarrow \mathcal{A}(\mathcal{L}) $ of algebras. In particular,
			\[ \mathcal{A}(\mathcal{L}_w)_u= \mathcal{A}(\mathcal{L})_u, \quad \forall u\leq w.   \]
		\end{enumerate} 
	\end{proposition}
	\begin{proof}
		For part (a), let $T$ and $T^{\prime}$ be the set of reflections of $\mathcal{L}$ and $\mathcal{L}^{\prime}$ respectively.   The algebra $\mathcal{A}(\mathcal{L}\times \mathcal{L}^{\prime})$ is generated by  $a_{t}, t\in T\cup T^{\prime}$. By the defining relation  we have  $a_{t} a_{t^{\prime}}=-a_{t^{\prime}} a_{t}$ for any $t \in T$ and $t^{\prime}\in T^{\prime}$. We have two natural embeddings $i_1: \mathcal{A}(\mathcal{L})\rightarrow \mathcal{A}(\mathcal{L}\times \mathcal{L}^{\prime})$ and $i_2: \mathcal{A}(\mathcal{L}^{\prime})\rightarrow \mathcal{A}(\mathcal{L}\times \mathcal{L}^{\prime})$, inducing the algebra isomorphism
		\[ f:\; \mathcal{A}(\mathcal{L}) \otimes \mathcal{A}(\mathcal{L}^{\prime}) \rightarrow \mathcal{A}(\mathcal{L}\times \mathcal{L}^{\prime}), \]
		such that $f(a_{t}\otimes a_{t^{\prime}})= i_1(a_{t})i_2(a_{t^{\prime}})$ and $i_1(a_{t})i_2(a_{t^{\prime}})=-i_2(a_{t^{\prime}})i_1(a_{t})$ for any $t\in T$ and $t^{\prime}\in T$.

		We turn to the proof of part (b). The set $T_w$ of reflections in $\mathcal{L}_w$ inherits the total order from the totally ordered set $T$ of reflections in $\mathcal{L}$.  Since $\mathcal{L}_w \subset \mathcal{L}$, the induced map $i_{\mathcal{A}}: \mathcal{A}(\mathcal{L}_w) \rightarrow \mathcal{A}(\mathcal{L})$ defined by $a_{t}\rightarrow a_{t}, t\in T_{w}$ preserves the defining relations and hence is an algebraic homomorphism. Moreover,  the induced map $i_{\mathcal{A}}$ preserves the decreasing basis elements and thus $i_{\mathcal{A}}$ is injective. 
	\end{proof}

	\subsection{The vanishing lemma}\label{sec: pfvanish}
	In this subsection we prove the vanishing property stated in \lemref{lem: vanishing}.
	\subsubsection{$T$-reduced expressions and root systems} 
	
	Consider the following two sets:
	\begin{equation}\label{eq: T}
		\begin{aligned}
			\mathcal{T}_1:&=\{ (t_1,t_2,\dots,t_k) \mid \text{$k\in \mathbb{N}$ and $t_1t_2\dots t_k$ is not $T$-reduced} \},\\
			\mathcal{T}_2:&=\{ (t_1,t_2,\dots,t_k) \mid \text{$k\in \mathbb{N}$ and $t_1t_2\dots t_k$ is $T$-reduced and $t_1t_2\dots t_k\not \leq \gamma$} \}.
		\end{aligned} 
	\end{equation}
	Then \lemref{lem: vanishing} can be restated as:
	\begin{equation}\label{eq: T1andT2}
		a_{{\bf t}}=0, \quad \text{for any ${\bf t}\in  \mathcal{T}_1\cup \mathcal{T}_2$.}
	\end{equation}
	To prove this,  we need a geometric characterisation for  $\mathcal{T}_1\cup \mathcal{T}_2$ in terms of the root system. 
 
   Let $\varrho: W\rightarrow  {\rm GL}(V)$ be the geometric  representation of $W$ with $V=\mathbb{R}^n$. Denote by $\Phi^+$ the set of positive roots of $W$,
	as determined by $S$ (see the remarks preceding \eqref{stein}).  
	Define
	\[ {\rm Fix}(w):={\rm Ker}\, (\varrho(w)-{\rm Id})\subseteq V \]
	to be the  vector subspace fixed by $w\in W$.  By \cite[Lemma 2]{Car72}, we have 
	\begin{equation}\label{eq: lencodim}
		\ell_{T}(w)={\rm codim}\,  {\rm Fix}(w)=n- {\rm dim}\,{\rm Fix}(w), \quad \forall w\in W. 
	\end{equation}
	

	Since the Coxeter element $\gamma$ fixes no vector in $V$, the linear map $\gamma-1$ is an automorphism of $V$. We  define the linear map 
	\begin{equation*}\label{eq: varthe}
		\vartheta:=(\gamma-1)^{-1}: V \rightarrow V.
	\end{equation*}
	This map satisfies the following properties which come from the proof of \cite[Lemma 2]{Car72}; see also  \cite[Corollary 4.2]{BW08}.
	
	\begin{lemma}\label{lem: thetaprop}
		Let $\rho\in \Phi^+$ be a positive root and let $\vartheta$ be as defined above.
		\begin{enumerate}
			\item $(\vartheta(\rho), \rho)= -\frac{1}{2}(\rho, \rho)$;
			\item $ \vartheta(\rho)\in {\rm Fix}(t_{\rho}\gamma)$. 
		\end{enumerate}
	\end{lemma}
	\begin{proof}
		We have $(\gamma-1)(\vartheta(\rho))=\rho$, which implies that $\gamma(\vartheta(\rho))=\vartheta(\rho)+\rho$. Since Coxeter group action preserves the inner product of $V$, we have  $(\gamma(\vartheta(\rho)), \gamma(\vartheta(\rho)))=(\vartheta(\rho), \vartheta(\rho))$ and hence
		$ (\vartheta(\rho), \rho)= -\frac{1}{2}(\rho, \rho)$. 
		It follows that $t_{\rho}(\vartheta(\rho))=\vartheta(\rho)+\rho$ and hence $\gamma(\vartheta(\rho))=t_{\rho}(\vartheta(\rho))$. This leads to
		$\vartheta(\rho)\in {\rm Fix}(t_{\rho}\gamma)$. 
	\end{proof}

	
	
	
	The following lemma characterises the reduced $T$-expressions.
	
	\begin{lemma}\label{lem: Carter} \cite[Lemma 3]{Car72}
		Let $\rho_1,\rho_2,\dots, \rho_k\in \Phi^{+}$. Then the expression  $t_{\rho_1}t_{\rho_2}\dots t_{\rho_k}$ is $T$-reduced if and only if $\rho_1,\rho_2,\dots,\rho_k$ are linearly independent.  
	\end{lemma}
	
	This follows directly from the fact that for $w\in W$, $\ell_T(w)=\dim(\im(w-1))$ (cf. \cite[(1.2)]{HL99}).
	The following lemma characterises the $T$-reduced expressions of elements occurring in the lattice $\mathcal{L}$.
	
	\begin{lemma}\label{lem: BW} \cite[Lemma 4.8]{BW08}
		Let $t_{\rho_1}t_{\rho_2}\dots t_{\rho_k}$ be a $T$-reduced expression. Then the following are equivalent:
		\begin{enumerate}
			\item $t_{\rho_1}t_{\rho_2}\dots t_{\rho_k}\leq \gamma$;
			\item $(\vartheta(\rho_i),\rho_j)=0$ whenever $1\leq i<j\leq k$.
		\end{enumerate}
	\end{lemma}

	\subsubsection{Proof of the vanishing lemma}
	
	To prove the equivalent statement \eqref{eq: T1andT2} of \lemref{lem: vanishing}, we need a description of the set $\mathcal{T}_1\cup \mathcal{T}_2$ in terms of positive roots.
	
	\begin{proposition}\label{prop: T}
		Let $\rho_1,\rho_2,\dots, \rho_k\in \Phi^{+}$, and let $\mathcal{T}_1$ and $\mathcal{T}_2$  be  as  in \eqref{eq: T}. Then we have $(t_{\rho_1}, t_{\rho_2},\dots ,t_{\rho_k})\in \mathcal{T}_1\cup \mathcal{T}_2$ if and only if there exists a pair $i<j$ such that $(\vartheta(\rho_i), \rho_j)\neq 0$. 
	\end{proposition}
	\begin{proof}
		Let  ${\bf t}=(t_{\rho_1}, t_{\rho_2},\dots, t_{\rho_k})\in \mathcal{T}_1\cup \mathcal{T}_2$ and assume for contradiction that we have $(\vartheta(\rho_i), \rho_j)= 0$ for any $1\leq i<j\leq k$.   Since the  matrix $((\vartheta(\rho_i), \rho_j))_{k\times k}$ is   non-singular by part (1) of \lemref{lem: thetaprop}, the positive roots  $\rho_i, 1\leq i\leq k$ are linearly independent. It follows from \lemref{lem: Carter} that  $t_{\rho_1}t_{\rho_2}\dots t_{\rho_k}$ is $T$-reduced and hence ${\bf t}\notin \mathcal{T}_1$, which implies that ${\bf t}\in \mathcal{T}_2$.  However, if $t_{\rho_1}t_{\rho_2}\dots t_{\rho_k}$ is $T$-reduced and   $(\vartheta(\rho_i), \rho_j)= 0$ for any $1\leq i<j\leq k$, then by \lemref{lem: BW}  we have $t_{\rho_1}t_{\rho_2}\dots t_{\rho_k}\leq \gamma$, which implies that ${\bf t}\notin \mathcal{T}_2$. This contradicts ${\bf t}\in \mathcal{T}_1\cup \mathcal{T}_2$. This proves the "only if" part.
		
	 For the converse, assume that there exist $i<j$ such that  $(\vartheta(\rho_i), \rho_j)\neq 0$. If $t_{\rho_1}t_{\rho_2} \dots t_{\rho_k}$ is not $T$-reduced, then $(t_{\rho_1}, t_{\rho_2},\dots t_{\rho_k})\in \mathcal{T}_1$. Otherwise, by \lemref{lem: BW} $t_{\rho_1}t_{\rho_2}\dots t_{\rho_k}\not \preceq \gamma$, and hence $(t_{\rho_1}, t_{\rho_2},\dots t_{\rho_k})\in  \mathcal{T}_2$. This completes the proof.
	\end{proof}

	We  now make  the following definition.
	
	\begin{definition}\label{def: betavan}
		A sequence of positive roots $(\rho_1, \rho_2,\dots, \rho_k)$ is called a \emph{vanishing sequence} if the inner product $(\vartheta(\rho_i), \rho_j)=0$ for all $1\leq i<j\leq k$ except that $(\vartheta(\rho_1), \rho_k)\neq 0$. 
	\end{definition}
	
	With this definition, we can refine \propref{prop: T} as follows.
	
	\begin{proposition}\label{prop: vanish}
		Let $\rho_1,\rho_2,\dots, \rho_k\in \Phi^{+}$. Then ${\bf t}=(t_{\rho_1}, t_{\rho_2},\dots t_{\rho_k})\in \mathcal{T}_1\cup \mathcal{T}_2$ if and only if there exists a pair $i<j$ such that $(\rho_i, \rho_{i+1} \dots ,\rho_j)$ is a vanishing sequence.
	\end{proposition}
	\begin{proof}
		By \propref{prop: T}, we may choose a pair $i<j$ for which $j-i$ is minimal such that $(\vartheta(\rho_i), \rho_j)\neq 0$. It follows from the minimality of $j-i$ that $(\vartheta(\rho_s), \rho_t)=0$ for all $i\leq s<t\leq j$ except for $s=i, t=j$. Therefore,  $(\rho_i,\rho_{i+1}, \dots, \rho_j)$ is a  vanishing sequence.  
	\end{proof}
	
	Note that for any $\rho \in \Phi^+$, the sequence $(\rho, \rho)$ is a vanishing sequence as the inner product $(\vartheta(\rho),\rho)\neq 0$ by part (1) of \lemref{lem: thetaprop}. If $t_{\rho_1}t_{\rho_2}$ is $T$-reduced and $t_{\rho_1}t_{\rho_2}\not \leq\gamma$, then by \lemref{lem: BW} the inner product  $(\vartheta(\rho_1),\rho_2)\neq 0$ and therefore the sequence $(\rho_1,\rho_2)$ is a vanishing sequence. It follows from the defining relations of $\mathcal{A}$ that $a_{t_{\rho_1}}a_{t_{\rho_2}}=a_{t_{\rho}}^2=0$.

	In general, we will prove  that $a_{t_{\rho_1}}\dots a_{t_{\rho_k}}=0$ for any vanishing sequence $(\rho_1,\rho_2,\dots , \rho_k)$. Before proving this, we need the following observations.
	
	\begin{lemma}\label{lem: vanprop}
		Let $(\rho_1, \rho_2,\dots, \rho_k)$ be a vanishing sequence with $k>2$. Then
		\begin{enumerate}
			\item  $t_{\rho_1}t_{\rho_2}\cdots t_{\rho_{k-1}}\leq \gamma$  and $t_{\rho_2}t_{\rho_3}\cdots t_{\rho_{k}}\leq \gamma$ are both $T$-reduced;
			\item   $t_{\rho_i}t_{\rho_j}\leq \gamma$ is $T$-reduced for all $i<j$ except that $i=1, j=k$;
			\item Let $w=t_{\rho_{1}}t_{\rho_2}\leq \gamma$ and let $\Phi^+(w)=\{\tau_1, \tau_2,\dots,\tau_m\}$ be the set of all positive roots for which $t_{\tau_i}\leq w$. Then $(\tau_i, \rho_3,\dots, \rho_k)$ is a vanishing sequence for any $\tau_i\neq \rho_2$. 
		\end{enumerate} 
	\end{lemma}
	\begin{proof}
		The  matrix $((\vartheta(\rho_i), \rho_j))_{(k-1)\times (k-1)}$ with $1\leq i,j \leq k-1$ is lower triangular with nonzero diagonal entries, thereby $\rho_i, 1\leq i\leq k-1$ are linearly independent. It follows from \lemref{lem: Carter} and  \lemref{lem: BW} that $t_{\rho_1}t_{\rho_2}\cdots t_{\rho_{k-1}}$ is $T$-reduced and  $t_{\rho_1}t_{\rho_2}\cdots t_{\rho_{k-1}}\leq \gamma$. Similarly, one can prove that $t_{\rho_2}t_{\rho_3}\cdots t_{\rho_{k}}\leq \gamma$. This completes the proof of part (1), from which part (2) follows.

		For part (3), let $V_1={\rm Fix}(w)\subseteq V$ and let $V_1^{\perp}$ be the orthogonal subspace such that $V=V_1\oplus V_1^{\perp}$. Then by \eqref{eq: lencodim} ${\rm dim} V_1=n-2$ and ${\rm dim} V_1^{\perp}=2$. Since $w$ fixes every vector in $V_1$, so is any expression $t_{\tau_i}t_{\tau_j}$ of $w$. Therefore, those positive roots $\tau_i\in \Phi^{+}(w)$ are in $V_1^{\perp}$.
		Since $t_{\rho_1}t_{\rho_2}$ is $T$-reduced, it follows from  \lemref{lem: Carter} that $\rho_1, \rho_2\in V_{1}^{\perp}$ are linearly independent and hence make up a basis for $V_1^{\perp}$. 
		
		For any $\tau_i\in \Phi^+(w)\subseteq V_{1}^{\perp}$ we have $\tau_i=\lambda_1 \rho_1+ \lambda_2 \rho_2$ for some $\lambda_1, \lambda_2\in \mathbb{R}$.  Note that $\lambda_1=0$ if and only if $\tau_i=\rho_2$. Then we have 
		\[
		\begin{aligned}
			(\vartheta(\tau_i), \rho_j)&= \lambda_1(\vartheta(\rho_1), \rho_j)+ \lambda_2(\vartheta(\rho_2), \rho_j)=0, \quad 1\leq i\leq n, 1\leq j\leq k-1,\\
			((\vartheta(\tau_i), \rho_k))&=\lambda_1(\vartheta(\rho_1), \rho_k)+ \lambda_2(\vartheta(\rho_2), \rho_k)=\lambda_1 (\vartheta(\rho_1), \rho_k).
		\end{aligned}
		\]
		Since $(\vartheta(\rho_1), \rho_k)\neq 0$,  we have $((\vartheta(\tau_i), \rho_k))\neq 0$ if and only if $\lambda_1\neq 0$ if and only if $\tau_i\neq \rho_2$. Therefore the sequence $(\tau_i, \rho_3,\dots, \rho_k)$ is a vanishing sequence for any $\tau_i\neq \rho_2$. 
	\end{proof}
	

	\begin{lemma}\label{lem: vanish}
		Let  $(\rho_1, \rho_2,\dots, \rho_k)$ be a vanishing sequence with $k\geq 2$. Then
		\[a_{t_{\rho_1}}a_{t_{\rho_2}}\cdots a_{t_{\rho_k}}=0. \] 
	\end{lemma}
	\begin{proof}
		We use induction on $k$. For the base case $k=2$, since $(\rho_1,\rho_2)$ is a vanishing sequence we have $(\vartheta(\rho_1), \rho_2)\neq 0$.  Then by \propref{prop: T}, $(t_{\rho_1},t_{\rho_2})\in \mathcal{T}_1\cup \mathcal{T}_2$. If $(t_{\rho_1},t_{\rho_2})\in \mathcal{T}_1$, then $t_{\rho_1}t_{\rho_2}$ is not $T$-reduced. This implies that $\rho_1=\rho_2$ and hence $a_{t_{\rho_1}}^2=0$. If $(t_{\rho_1},t_{\rho_2})\in \mathcal{T}_2$, then we have $t_{\rho_1}t_{\rho_2}\not\leq  \gamma$ and hence $a_{t_{\rho_1}}a_{t_{\rho_2}}=0$.
		
		Now let $(\rho_1, \rho_2,\dots, \rho_k)$ be vanishing sequence with $k>2$. Using part (2) of \lemref{lem: vanprop} we have  $w=t_{\rho_1}t_{\rho_2}\leq \gamma$. Then by \propref{prop: intncp} the interval $[e, w]$ is isomorphic to the noncrossing partition lattice of a dihedral group $I_{2}(m)$ for some $m\geq 2$. Suppose that $\Phi^{+}(w)=\{ \tau_1, \tau_2,\dots ,\tau_m \}$ is set of all positive roots for which $t_{\tau_i}\leq w$. 
		By the defining relation, we have 
		\[ a_{t_{\rho_1}} a_{t_{\rho_2}}=-\sum_{w=t_{\tau_i}t_{\tau_j}, \; \tau_j\neq \rho_2 }  a_{t_{\tau_i}}a_{t_{\tau_j}}, \]
		where the sum is over all $T$-reduced expressions $t_{\tau_i}t_{\tau_j}$ of $w$ with $\tau_j\neq \rho_2$. Using the above relation, we obtain
		\[ a_{t_{\rho_1}}a_{t_{\rho_2}}a_{t_{\rho_3}}\cdots a_{t_{\rho_n}}=-\sum_{w=t_{\tau_i}t_{\tau_j}, \; \tau_j\neq \rho_2 }  a_{t_{\tau_i}}a_{t_{\tau_j}} a_{t_{\rho_3}}\cdots a_{t_{\rho_n}}. \]
		It follows from part (3) of \lemref{lem: vanprop} that $(\tau_j, \rho_3,\dots, \rho_n)$ is a vanishing sequence for any $\tau_j \neq \rho_2$, and hence by induction hypothesis $a_{t_{\tau_j}}a_{t_{\rho_3}}\cdots a_{t_{\rho_n}}=0$. Therefore,  we have $a_{t_{\rho_1}}a_{t_{\rho_2}}\cdots a_{t_{\rho_n}}=0$
	\end{proof}
	
	We are now in a position to prove \lemref{lem: vanishing}.

	\begin{proof}[Proof of \lemref{lem: vanishing}]
		By \propref{prop: vanish}, ${\bf t}=(t_{\rho_1}, t_{\rho_2}, \dots, t_{\rho_k})\in \mathcal{T}_1\cup \mathcal{T}_2$ if and only if  there exists $i<j$ such that $(\rho_i, \rho_{i+1} \dots ,\rho_j)$ is a vanishing sequence. Using \lemref{lem: vanish}, in this case we have 
		$a_{{\bf t}}=a_{t_{\rho_1}} \dots (a_{t_{\rho_i}}a_{t_{\rho_{i+1}}}\dots a_{t_{\rho_j}})\dots a_{t_{\rho_k}}=0$.
	\end{proof}

	
	\section{Chain complexes for Milnor fibres and hyperplane arrangements}
	We now begin our discussion of the chain complexes whose homology realise that of the Milnor fibre 
	and of the hyperplane complement associated with $W$. For any $u,v\in W$, write $u^w:=w\inv uw$.
	
	\subsection{Some acyclic chain complexes}
	Recall that  $\mathcal{B}=\bigoplus_{k=0}^{n}\mathcal{B}_k$ is a $\mathbb{Z}$-graded algebra generated by the $\beta_{t}$ for $t\in T$.  For any ${\bf t}=(t_1, t_2, \dots, t_k)\in {\rm Rex}_{T}(w)$ with $w\in \mathcal{L}$, i.e.,  such that $w=t_1t_2\dots, t_k\in \mathcal{L}$ is a $T$-reduced expression, we have $\beta_{{\bf t}}= \beta_{t_1}\beta_{t_2}\dots \beta_{t_k}$. Recalling  the $\mathbb{Z}$-linear map $d$  from \eqref{eq: theta} and \propref{prop: cycle}, we have
	\begin{equation}
		d_k: \mathcal{B}_{k}\rightarrow \mathcal{B}_{k-1}, \quad \beta_{{\bf t}} \mapsto z_{{\bf t}}=\sum_{i=1}^{k}(-1)^{k-i}\beta_{{\bf t}(\hat{i})} , \quad
	\end{equation}
	In particular, $d_1(\beta_t)=1$ for any $t\in T$. The following results can be found in \cite[Lemma 5.9, Proposition 5.11]{Zha22}\label{prop: dprop}.

	\begin{proposition}\label{prop: propB}
		The following properties hold for $(\CB,d)$. 
		\begin{enumerate}
			\item  We have $d^2=0$, whence we have the following chain complex $(\mathcal{B}, d)$:
			\begin{equation*}\label{eq: chaincomp}
				\xymatrix{
					0  \ar[r] &  \mathcal{B}_{n} \ar[r]^-{d_{n}} &  \mathcal{B}_{n-1}\ar[r]^-{d_{n-1}} &  \cdots  \ar[r]^-{d_1} &  \mathcal{B}_{0}\ar[r] & 0.
				}
			\end{equation*}

			\item The chain complex $(\mathcal{B}, d)$ is acyclic. More precisely, let  $\mathcal{L}_{[k]}:=\{ w\in \mathcal{L} | 1\leq \ell_{T}(w)\leq k\}$ be a rank-selected subposet of $\mathcal{L}$. If $1\leq k\leq n-1$, then
			\begin{equation*}
				{\rm Im}\,d_k=\widetilde{H}_{k-2}(\mathcal{L}_{[k-1]}),\quad   {\rm Ker}\, d_k = \widetilde{H}_{k-1}(\mathcal{L}_{[k]}).
			\end{equation*}
			Otherwise, ${\rm Im}\,d_n=\widetilde{H}_{n-2}(\mathcal{L}_{[n-1]})$ and  ${\rm Ker}\, d_n =0$. 
			
			\item  (Leibniz rule)
			For each $i=1,2,\dots, k-1$, we have
			\[ d(\beta_{\mathbf{t}})=(-1)^{k-i}(d\beta_{(t_1,\dots,t_i)})\beta_{(t_{i+1},\dots,t_k)}+\beta_{(t_1\dots, t_i)} (d\beta_{(t_{i+1},\dots,t_k)}) \]
			for any ${\bf t}=(t_1,t_2,\dots, t_k)\in {\rm Rex}_{T}(w)$, where $w\in \mathcal{L}_k$  with $2\leq k\leq n$.
		\end{enumerate}
	\end{proposition}
	
	In view of \thmref{thm: isothm}, the map $\phi:\CA\lr\CB$ given by $\phi(a_{t_1}a_{t_2}\dots a_{t_k})=\beta_{t_1}\beta_{t_2}\dots\beta_{t_k}$ 
	for any  $T$-reduced expression   $w=t_1t_2\dots, t_k\in \mathcal{L}$, is a graded isomorphism from $\CA$  to $\mathcal{B}$. 
	By abuse of notation, for each $k=1,\dots n$ we define the $\mathbb{Z}$-linear map
	\begin{equation}\label{eq: d}
		d: \mathcal{A}_{k}\rightarrow \mathcal{A}_{k-1}, \quad a_{t_1}a_{t_2}\dots a_{t_k} \mapsto \sum_{i=1}^{k}(-1)^{k-i} a_{t_1}\dots \hat{a}_{t_i}\dots a_{t_k}.
	\end{equation}
	Then the properties described in \propref{prop: dprop} for $(\CB,d)$ hold {\it mutatis mutandem} for $(\mathcal{A},d)$. In particular, $(\mathcal{A},d)$ is an acyclic complex.

	We proceed next to give another acyclic complex induced by the ``Kreweras complement'' of the NCP lattice. Now the noncrossing partition lattice $\mathcal{L}$ is self-dual, i.e., $\mathcal{L}\cong \mathcal{L}^{op}$, where  $\mathcal{L}^{{\rm op}}$ is the set $\mathcal{L}$ with the reverse partial order $<_{{\rm op}}$. This isomorphism may be realised explicitly by the Kreweras complement $K$, defined by
		\begin{equation*}\label{eq: Kre}
		K:  \mathcal{L}\longrightarrow  \mathcal{L}^{{\rm op}}, \quad w\mapsto \gamma w^{-1}.
	\end{equation*}
	
	The Kreweras map induces an automorphism of  $\mathcal{A}=\mathcal{A}(\mathcal{L})$ as a graded algebra. Recall from \lemref{coro: nonzero} that an element $a_{t_1}a_{t_2}\dots a_{t_k}$ of $\mathcal{A}$ is nonzero if and only if $e< t_1< t_1t_2< \dots < t_1t_2\dots t_k$ is a chain of $\mathcal{L}$. The latter is mapped by $K$ to the  chain $\gamma <_{{\rm op}} \gamma t_1<_{{\rm op}} \gamma t_2t_1 <_{{\rm op}}
	\dots <_{{\rm op}}   \gamma t_k \dots t_1 $
	of $\mathcal{L}^{{\rm op}}$, and this corresponds to a nonzero element $a_{t_1} a_{t_2^{t_1}}\dots  a_{t_k^{t_{k-1}\dots t_{1}}}$ of $\mathcal{A}(\mathcal{L}^{{\rm op}})$. Note further that $\mathcal{A}(\mathcal{L}^{{\rm op}})$ is isomorphic to the opposite algebra  $ \mathcal{A}(\mathcal{L})^{{\rm op}}$. Therefore, we have the following composite of linear maps:

	\[
	\kappa: \\ \mathcal{A}(\mathcal{L}) \overset{\bar{\kappa}} \rightarrow \mathcal{A}(\mathcal{L}^{{\rm op}})\cong \mathcal{A}(\mathcal{L})^{{\rm op}} \overset{\theta} \rightarrow \mathcal{A}(\mathcal{L}),
	\]
	where $\bar{\kappa}$ is induced as above by $K$, so that  $\bar{\kappa}(a_{t_1}a_{t_2}\dots a_{t_k})=a_{t_1} a_{t_2^{t_1}}\dots  a_{t_k^{t_{k-1}\dots t_{1}}}$ and $\theta$ is the anti-isomorphism given by $\theta(a_{t_1} a_{t_2}\dots a_{t_k})= a_{t_k} \dots a_{t_2}a_{t_1}$. Hence the linear automorphism $\kappa$ of $\mathcal{A}$ is defined explicitly by 
	\begin{equation}\label{eq: kap}
		\kappa (a_{t_1}a_{t_2}\dots a_{t_k})=  a_{t_k^{t_{k-1}\dots t_{1}}} \dots a_{t_2^{t_1}} a_{t_1}.
	\end{equation}
	It is clear that $\kappa$ preserves the defining relations of $\mathcal{A}$. 
	
	In addition to the differential \eqref{eq: d}, for each $k=1,\dots, n$ we define the following $\mathbb{Z}$-linear map
	\begin{equation}\label{eq: del}
		\delta: \mathcal{A}_{k}\rightarrow \mathcal{A}_{k-1}, \quad a_{t_1}a_{t_2}\dots a_{t_k} \mapsto \sum_{i=1}^{k}(-1)^{i-1} a_{t_1^{t_i}}\dots a_{t_{i-1}^{t_i}} \hat{a}_{t_i}\dots a_{t_k}.
	\end{equation}
	
	The properties of $\delta$ are summarised in the next result.
	
	\begin{proposition}\label{prop: delta} Let $\delta$ and $\kappa$ be as defined above. 
		\begin{enumerate}
			\item (Leibniz rule)
			For each $k=2,\dots, n$, we have
			\[ \delta(a_{t_1}a_{t_2}\dots a_{t_k})=\delta(a_{t_1}a_{t_2}\dots a_{t_{k-1}})a_{t_k}+ (-1)^{k-1}a_{t_1^{t_k}}a_{t_2^{t_k}}\dots a_{t_{k-1}^{t_k}}.  \]
			for any ${\bf t}=(t_1,t_2,\dots, t_k)\in {\rm Rex}_{T}(w)$ with $w\in \mathcal{L}_k$. 
			
			\item For each integer $k=1,\dots, n$, the following diagram commutes:
			\begin{center}
				\begin{tikzcd}
					\mathcal{A}_k \arrow[r, "d"] \arrow[d, "\kappa"]
					&  \mathcal{A}_{k-1} \arrow[d, "\kappa"] \\
					\mathcal{A}_k  \arrow[r,  "\delta" ]
					&  \mathcal{A}_{k-1} .\end{tikzcd}
			\end{center}
			Therefore, the complex $(\mathcal{A}, \delta)$ is acyclic.
			
			\item We have $d\delta=\delta d$.
		\end{enumerate}
	\end{proposition}
	\begin{proof}
		Part (1) follows directly from the definition \eqnref{eq: del}.  The  diagram commutes by  straightforward calculation. As $\kappa$ is an automorphism and $(\mathcal{A},d)$ is acyclic, the acyclicity of $(\mathcal{A},\delta)$ follows. It is easily verified directly that $d$ and $\delta$ commute with each other. 
	\end{proof}

	\begin{remark}
		Note that $d$ and $\delta$ do not preserve the defining relations of the algebra $\mathcal{A}$. For instance, we have $a_{t_1}a_{t_2}=0$ for any pair of reflections $t_1, t_2$ satisfying $t_1t_2\not \leq \gamma$. However, $d(a_{t_1}a_{t_2})=-a_{t_2}+a_{t_1}\neq 0$ and  $\delta (a_{t_1}a_{t_2})=-a_{t_2}+a_{t_2t_1t_2}\neq 0$. 
	\end{remark}
	
	\begin{remark}\label{rmk: welldef}
		It follows from \lemref{coro: nonzero} that $\mathcal{A}_{k}$ as a free abelian group is spanned by nonzero elements $a_{t_1}a_{t_2}\dots a_{t_k}$, where $t_1t_2\dots t_k=w\in \mathcal{L}$ is a $T$-reduced expression.  By the defining relations of $\mathcal{A}$, all linear relations among nonzero elements $a_{t_1}\dots a_{t_k}$ are generated by the quadratic relation $   \sum_{(t_1,t_2)\in {\rm Rex}_T(w)} a_{t_1} a_{t_2}=0$ for any $w\in \mathcal{L}_2$. It is easily verified that $d(\sum_{(t_1,t_2)\in {\rm Rex}_T(w)} a_{t_1} a_{t_2})=0$ for any $t\in T$, and similarly this holds for $\delta$. Therefore, the linear maps $d$ and $\delta$ are well-defined. 
	\end{remark}
	
	\subsection{Complexes for the Milnor fibre and hyperplane complement}\label{sec: comp}
	
	Let $\mathcal{H}$ be the set of (complexified) reflecting hyperplanes of $W$ and write $M=M_W:=V_\C\setminus \cup_{H\in\mathcal{H}}H$ for the corresponding hyperplane complement.
	For $H\in\mathcal{H}$, let $\ell_H\in V_\C^*$ be a corresponding linear form, so that $H=\ker(\ell_H)$. Let $Q=\prod_{H\in\mathcal{H}}\ell_H^2$; it is well known that the $\ell_H$
	may be chosen so that $Q$ is $W$-invariant. The Milnor fibre $F:=Q\inv(1)=\{v\in V_\C\mid Q(v)=1\}$.  Evidently $W$ acts on both $M$ and $F$, so that we may speak of the orbit spaces $M/W$ and $F/W$.

\subsubsection{Chain complexes for $M$ and $F$}	Recall from \cite{Zha22} the following chain complexes which compute the integral homology of the hyperplane complement and of the Milnor fibre. We use the algebra $\mathcal{A}$ instead of $\mathcal{B}$ in the original complexes. 
	
	\begin{theorem}\cite[Theorem 7.2]{Zha22}\label{thm: homoM}
		The integral homology of the hyperplane complement $M$ is isomorphic to the homology of the following chain  complex of abelian groups:
		\be\label{eq:ccm}
		\begin{tikzcd}
			0 \arrow[r]&  \mathbb{Z}W\otimes  \mathcal{A}_n \arrow[r,"\partial_{n}"]  & \cdots \arrow[r] &   \mathbb{Z}W\otimes \mathcal{A}_1  \arrow[r,"\partial_1"]&  \mathbb{Z}W\otimes \mathcal{A}_0 \arrow[r] & 0,
		\end{tikzcd}
		\ee
		where the boundary maps are given by 
		\[ 
		\partial_k(w\otimes a_{t_1}a_{t_2}\dots a_{t_k} )= \sum_{i=1}^k(-1)^{i-1} wt_i\otimes a_{t_1^{t_i}}\dots a_{t_{i-1}^{t_i}} \hat{a}_{t_i}\dots a_{t_k}
		- \sum_{i=1}^k (-1)^{i-1} w\otimes a_{t_1}\dots \hat{a}_{t_i}\dots a_{t_k} 
		\]
		for any $w\in W$ and $(t_1, t_2, \dots, t_k)\in \bigcup_{u\in \mathcal{L}_k} {\rm Rex}_T(u)$.
	\end{theorem}
	
	\begin{theorem}\cite[Theorem 6.3]{Zha22}\label{thm: fullcom}
		The integral homology of the  Milnor fibre $F_Q=Q^{-1}(1)$ is isomorphic to the homology of the following chain complex
		\be\label{eq:ccf}
		\begin{tikzcd}
			0 \arrow[r]&  \mathbb{Z}W\otimes d(\mathcal{A}_n) \arrow[r,"\partial_{n-1}"]  & \cdots \arrow[r] & \mathbb{Z}W\otimes  d(\mathcal{A}_2)  \arrow[r,"\partial_1"]& \mathbb{Z}W\otimes d(\mathcal{A}_1) \arrow[r] & 0,
		\end{tikzcd}
		\ee
		where the boundary maps  are given by 
		\begin{equation*}
			\partial_{k-1}(w\otimes d(a_{t_1}a_{t_2}\dots a_{t_k}) )= \sum_{i=1}^{k} (-1)^{i-1} wt_i \otimes d(a_{t_1^{t_i}}\dots a_{t_{i-1}^{t_i}} \hat{a}_{t_i}\dots a_{t_k})
		\end{equation*}
		for any $w\in W$,  $(t_1,\dots ,t_k)\in \bigcup_{u\in \mathcal{L}_k} {\rm Rex}_{T}(u)$ with $2\leq k\leq n$. 
	\end{theorem}
	
	\subsubsection{$W$-action on these complexes}\label{sub:wact}
	Note that $W$ acts on the above chain complexes as follows. 
	For $x\in W$, $x. w\ot a_{t_1}a_{t_2}\dots a_{t_k}=(xw)\ot a_{t_1}a_{t_2}\dots a_{t_k}$ in the case of \eqref{eq:ccm}, while in the case of \eqref{eq:ccf},
	$x. w\ot d(a_{t_1}a_{t_2}\dots a_{t_k})=(xw)\ot d(a_{t_1}a_{t_2}\dots a_{t_k})$. It is evident that this $W$-action respects the boundary homomorphisms.
	Thus these chain complexes are both left $W$-modules.
	
	Now it is well known that $W$ acts freely on $M$ and $F$. The quotient spaces $M/W$ and $F/W$ are both $K(\pi,1)$-spaces. In particular, the fundamental group $\pi_1(M/W)=A(W)$, the Artin group of $W$. Therefore, we have $H_k(A(W); \mathbb{Z})= H_{k}(M/W; \mathbb{Z})$, and this may be computed using $\mathcal{A}$ as follows. 
	
	\begin{theorem}\cite[Theorem 7.5]{Zha22}\label{thm: homoMquo}
		The integral homology of  $M/W$ or the Artin group $A(W)$ is isomorphic to the homology of the following chain  complex:
		\[ \begin{tikzcd}
			0 \arrow[r]&    \mathcal{A}_n \arrow[r,"\partial_{n}"]  & \mathcal{A}_{n-1}\arrow[r]& \cdots \arrow[r] &    \mathcal{A}_1  \arrow[r,"\partial_1"]&   \mathcal{A}_0 \arrow[r] & 0,
		\end{tikzcd} \]
		where the boundary maps are given by $\partial_k= \delta_k +(-1)^kd_k$, i.e., 
		\begin{equation}\label{eq: HomArtin}
			\partial_k(a_{t_1}a_{t_2}\dots a_{t_k} )= \sum_{i=1}^k(-1)^{i-1} a_{t_1^{t_i}}\dots a_{t_{i-1}^{t_i}} \hat{a}_{t_i}\dots a_{t_k}
			- \sum_{i=1}^k (-1)^{i-1}  a_{t_1}\dots \hat{a}_{t_i}\dots a_{t_k}
		\end{equation}
		for $(t_1, t_2, \dots, t_k)\in \bigcup_{u\in \mathcal{L}_k} {\rm Rex}_T(u)$.
	\end{theorem}

	\begin{theorem}\cite[Theorem 6.7]{Zha22} \label{thm: disfibre}  
		The integral homology of $F/W$ is isomorphic to the homology of the following chain complex :
		\begin{equation*}
			\xymatrix{
				0  \ar[r] & d(\mathcal{A}_n)  \ar[r]^-{\partial _{n-1}}  &  d(\mathcal{A}_{n-1}) \ar[r]& \cdots  \ar[r] &  d(\mathcal{A}_2) \ar[r]^-{\partial _1} & d(\mathcal{A}_1) \ar[r]&  0,
			}
		\end{equation*}
		where the boundary maps are given by $\partial_k=\delta_k$, i.e.,
		\begin{equation}\label{eq: bounddis}
			\partial _{k-1}(d(a_{t_1}a_{t_2}\dots a_{t_k}))= \sum_{i=1}^{k} (-1)^{i-1} d(a_{t_1^{t_i}}\dots a_{t_{i-1}^{t_i}} \hat{a}_{t_i}\dots a_{t_k})
		\end{equation}
		for any $(t_1,\dots ,t_k)\in \bigcup_{w\in \mathcal{L}_k} {\rm Rex}_{T}(w)$ and $2\leq k\leq n$. In particular, $\partial_1=0$.
	\end{theorem}

	\subsection{A new pair of complexes}
	Let us start with two new complexes defined over $\mathbb{C}$.
	The first, denoted $\CC$, is
	\be\label{eq:cxc}
	\CC:= 0\lr \C W\ot\CA_n\overset{\partial_{n}}{\lr}\dots\lr \C W\ot\CA_1\overset{\partial_1}{\lr}\C W\ot \CA_0\lr 0,
	\ee
	with the boundary homomorphisms $\partial_i$ defined as in \thmref{thm: homoM}.  The second, denoted $\CK$, is
	\be\label{eq:cxm}
	\CK:=0\lr \C W\ot d(\CA_n)\overset{\partial_{{n-1}}}{\lr}\dots \lr \C W\ot d(\CA_2)\overset{\partial_1}{\lr}\C W\ot d(\CA_1)\lr 0,
	\ee
	where the boundary homomorphisms are as defined in \thmref{thm: fullcom}. Then by \thmref{thm: homoM} and  \thmref{thm: fullcom}, the homology of the complex $\CC$ (resp. $\CK$) is the homology of the corresponding hyperplane complement $M=M_W$ (resp. Milnor fibre $F$) with complex coefficients.
	
	\subsubsection{Left, right and bi-$W$-modules} In the development below, we shall need to distinguish between left and 
	right $\C W$-modules. Let $\Mod_R(\C W)$ (resp. $\Mod_L(\C W)$) be the category of finite dimensional right (resp. left) $\C W$-modules.
	We shall use the categorical isomorphisms $\lambda_{RL}:U\mapsto U_L$ from $\Mod_R(\C W)\lr \Mod_L(\C W)$, and $\lambda_{LR}: X\mapsto X_R$ from 
	$\Mod_L(\C W)\lr \Mod_R(\C W)$ ($U\in\Mod_R(\C W), X\in\Mod_L(\C W)$) where, for $u\in U_L$ and $w\in W$,  $w.u:=u.w\inv$, and for $x\in X_R$,
	$x.w:=w\inv.x$. 
	
	In addition to the above categories, we shall need the category $\Mod_{LR}(\C W)$ of finite dimensional $\C W$-bimodules. If 
	$Y\in\Mod_{LR}(\C W)$, then for $w_1,w_2\in W$ and $y\in Y$, we have $(w_1y)w_2=w_1(yw_2)$. Evidently if $U\in\Mod_R(\C W)$, $X\in\Mod_L(\C W)$,
	then $X\ot_\C U$ is naturally a $W$-bimodule. Using this, we have the following formulation of a standard decomposition of a finite group algebra.
	
	\begin{lemma}\label{lem:cwdec}
		Maintaining the above notation, we have the following isomorphism in the category of $W$-bimodules.
		\[
		\C W\cong \oplus_{U\in Irr(\Mod_R(\C W))}(U_L^*\ot_\C U).
		\]
		Here the sum is over the simple modules $U\in\Mod_R(\C W)$ and $U_L^*$ denotes the contragredient of $U_L$ (defined above).
	\end{lemma}
	\begin{proof}
		It is easily verified that for any right $\C W$-module $U$, the left module $U_L^*$ has the same character as $U$.
		This is because $U_L\simeq U$ as vector spaces, and $w\in W$ acts on $U_L$ as $w\inv$ does on $U$. Thus they have 
		complex conjugate traces. But the character of $w$ on $U_L^*$ is the conjugate of its character on $U_L$, and hence coincides with 
		its charcater on $U$.
		The stated decomposition is now standard.
		
		Note that each of the tensor factors in each summand is referred to as the 
		multiplicity module of the other factor in that summand.
	\end{proof}

	\begin{corollary}\label{cor.cw}
		In the above notation, we have $\Hom_{\C W}(U_L^*,\C W)\cong U$ as right $W$-module. Here the left side indicates
		homomorphisms of left $\C W$-modules.
	\end{corollary}
	\begin{proof}
		By Lemma \ref{lem:cwdec}, 
		$\Hom_{\C W} (U_L^*, \C W)\cong \oplus_{U\in Irr(\Mod_R(\C W))}\Hom_{\C W}(U_L^*, U_L^*\ot_\C U)$
		But for fixed $U$, $U_L^*\ot_\C U$ is isomorphic to a sum of simple left modules $U_L^*$. The result follows.
	\end{proof}
	
\subsubsection{A new pair of complexes}	We next define ``relative'' versions of the complexes \eqref{eq:cxc} and \eqref{eq:cxm}.

	Let $U\in\Mod_R(\C W)$. Define complexes $\CC(U)$ and $\CK(U)$ as follows.
	\[
	\CC(U):= 0\lr U\ot\CA_n\overset{\partial_{n}}{\lr}\dots\lr U\ot\CA_1\overset{\partial_1}{\lr}U\ot \CA_0\lr 0,
	\]
	where the boundary maps are given (for $u\in U$) by
	\be\label{eq:eq:bdycu}
	\begin{aligned}
		\partial(u\ot a_{t_1}a_{t_2}\dots a_{t_k})=&\sum_{i=1}^k(-1)^{i-1}ut_i\ot a_{t_1^{t_i}}\dots a_{t_{i-1}^{t_i}}a_{t_{i+1}}\dots a_{t_k}\\
		&-\sum_{i=1}^k(-1)^{i-1}u\ot a_{t_1}\cdots \widehat{a_{t_i}}\dots a_{t_k}.\\
	\end{aligned}
	\ee
	\[
	\CK(U):=0\lr U\ot d(\CA_n)\overset{\partial_{n-1}}{\lr}\cdots \lr U\ot d(\CA_2)\overset{\partial_1}{\lr}U\ot d(\CA_1)\lr 0,
	\]
	where the boundary homomorphisms are defined for $u\in U$ by
	\be\label{eq:eq:bdyku}
	\partial(u\ot d(a_{t_1}a_{t_2}\dots a_{t_k}))=\sum_{i=1}^k(-1)^{i-1}ut_i\ot d(a_{t_1^{t_i}}\dots a_{t_{i-1}^{t_i}}a_{t_{i+1}}\dots a_{t_k}).
	\ee
	
	Now as observed in \ref{sub:wact}, the complexes $\CC$ and $\CK$ admit a (left) $W$-action. This is not generally the case for $\CC(U)$ or $\CK(U)$.  
	Hence for each integer $k$, $H_k(\CC)$ and $H_k(\CK)$
	are left $\C W$-modules. For left $\C W$-modules $U_1,U_2$ we write $\langle U_1,U_2\rangle_W$ for the multiplicity $\dim \Hom_{\C W}(U_1,U_2)$ as usual.
	
	Recall that for any simple module $U\in\Mod_R(\C W)$, we have a ``corresponding'' module $U_L^*\in\Mod_L(\C W)$, whose character coincides with that of $U$.
	
	\begin{theorem}\label{thm: mult} 
		For any right $W$-module $U$ and for each integer $k\geq 0$, we have
		\be\label{eq:conj1}
		\dim H_k(\CC(U))=\langle U_L^*, H_k(M)\rangle
		\ee
		and 
		\be\label{eq:conj2}
		\dim H_k(\CK(U))=\langle U_L^*, H_k(F)\rangle.
		\ee
	\end{theorem}
	\begin{proof}
		We prove  the first equation \eqref{eq:conj1}; the second equation \eqref{eq:conj2} has a similar proof. Since $H_k(M)=H_k(\mathcal{C})$, we have 
		\[
		\begin{aligned}
			\langle U_L^*, H_k(M) \rangle &=  \langle U, H_k(\mathcal{C})  \rangle={\rm dim}\,  {\rm Hom}_W(U_L^*, H_k(\mathcal{C})) 
		\end{aligned}    
		\]
		
		Now by semisimplicity, $U_L^*$ is a flat module, whence the functor $\Hom_\C(U_L^*,-)$ is exact. Moreover the $k^{th}$ homology functor commutes with 
		any exact functor. It follows that we have the following isomorphism of left $W$-modules.
		\be\label{eq:fl}
		H_k(\Hom_\C(U_L^*,\CC))\cong \Hom_\C(U_L^*,H_k(\CC)).
		\ee
		
		Further, the fixed point functor $(-)^W: M \mapsto M^W$ from the left $W$-module $M$ to the $\mathbb{C}$-module $M^W$ is representable. It is represented by the trivial 
		$W$-module $\mathbb{C}$, so that $(-)^W$ is naturally isomorphic to the functor ${\rm Hom}_W(\mathbb{C}, -)$. Since $\mathbb{C}W$ is a semisimple algebra, the trivial $W$-module $\mathbb{C}$ is projective. 
		Therefore, ${\rm Hom}_W(\mathbb{C}, -)\cong (-)^W$ is an exact functor. It follows that, upon taking $W$-fixed points in \eqref{eq:fl}, we obtain 
		\[
		\begin{aligned}
			H_k(\Hom_\C(U_L^*,\CC))^W\cong & H_k((\Hom_\C(U_L^*,\CC))^W)\\
			\cong & H_k(\Hom_{\C W}(U_L^*,\CC) )\cong \Hom_{\C W}(U_L^*,H_k(\CC)).
		\end{aligned}
		\]
		
		It now remains only to relate the complex $\Hom_{\C W}(U_L^*,\CC)$ to $\CC(U)$. for this, observe that 
		\[
		\Hom_{\C W}(U_L^*,\CC)_k\cong \Hom_{\C W}(U_L^*,\CC_k)\cong \Hom_{\C W}(U_L^*,\C W\ot_\C \CA_k).
		\]
		
		Moreover since $W$ acts trivially on $\CA$, we have 
		\[
		\Hom_{\C W}(U_L^*,\CC)_k\cong \Hom_{\C W}(U_L^*,\C W)\ot_\C \CA_k.
		\]
		But by Corollary \ref{cor.cw}, the right side of the above equation is equal to $U\ot_\C\CA_k=\CC(U)_k$, and the proof is complete.
	\end{proof}
	
	\subsection{Applications}\label{applmult} We give some special cases and applications of Theorem \ref{thm: mult}.
	
	First, consider the case $U=1_W$, the trivial $\C W$-module. Then $U_L^*=1_W$ and we have 
	\[
	\langle U_L^*,H_k(M)\rangle=\langle 1_W,H_k(M)\rangle=\dim H_k(M)^W.
	\]
	
	Moreover by the transfer theorem for homology with coefficients in $\C$ (cf. \cite[Theorem III.2.4]{GB72}), $H_k(M)^W\cong H_k(M/W)$.
	It follows from the first statement in Theorem \ref{thm: mult} that 
	\[	H_k(M/W)\cong H_k(\CC(1_W)).
	\]
	
	It is readily checked that the complex $\CC(1_W)$ coincides with the (complexification of the) complex in Theorem \ref{thm: homoMquo}, and in this way we recover that theorem for complex homology. Note that Theorem 3.7 is stronger, in that it computes the integral homology. 
	
	Next, using precisely the same arguments, we deduce that if $F$ is the Milnor fibre as defined above, then
	\[
	H_k(F/W)\cong H_k(\CK(1_W)).
	\]
	
	In this case, one again checks readily that $\CK(1_W)$ may be identified with the complexification of the complex in Theorem \ref{thm: disfibre}, whence in this case
	we recover Theorem \ref{thm: disfibre}, again with coefficients in $\C$,  by applying Theorem \ref{thm: mult} in the case of $\CK(U)$ with $U=1_W$.
	
	Consider next the case $U=\ve$, the alternating representation of $W$. Then $U_L^*\simeq \ve$ and it follows from
	\cite[(1.2)]{Leh96} that
	\[
	\langle H_k(M),\ve\rangle=0\text{ for all }k.
	\]
	
	It follows immediately from the first part of  Theorem \ref{thm: mult} that
	\[
	\text{The complex }\CC(\ve)\text{ is  acyclic}.
	\]

	Now we may think of $\CC(U)=\CC(\ve)$
	as having chain groups with bases $\{b\ot a_{\bf t}\mid {\bf t}\in \bigcup_{w\in \mathcal{L}} \mathcal{D}_w\}$, where $b$ is the basis element of $\ve$.
	Using the fact that $bt=-b$ for all reflections $t$, it follows from \eqref{eq:eq:bdycu}    that $\CC(\ve)$ may be identified with the chain complex
	\be\label{eq:cve}
	0\lr \CA_n\overset{D_n}{\lr}\dots\CA_1\overset{D_1}{\lr} \CA_0\lr 0,
	\ee
	where the boundary homomorphism $D_k:\CA_k\lr\CA_{k-1}$ is given by
	\[
	\begin{aligned}
		D_k(b\ot a_{t_1}\dots a_{t_k})=&\sum_{i=1}^k(-1)^{i-1}(-b)\ot a_{t_1^{t_i}}\dots a_{t_{i-1}^{t_i}}a_{t_{i+1}}\dots a_{t_{k}}\\
		-& \sum_{i=1}^k(-1)^{i-1}b\ot a_{t_1}\dots \widehat{a_{t_i}}\dots a_{t_k}.\\
	\end{aligned}
	\]
	
	It follows that we may identify $\CC(\ve)$ with the complex \eqref{eq:cve}, where 
	\[
	D_k=-{\delta_k}+\tilde{d_k},
	\]
	where $\delta_k$ is the restriction to $\CA_k$ of $\delta$, which is defined in \eqref{eq: del} and $\tilde{d_k}=(-1)^{k}d_k$, where $d=\oplus_{k=1}^nd_k$
	is defined in \eqref{eq: d}.
	
	It follows from Proposition \ref{prop: delta} (3) that $\tilde d \delta=-\delta\tilde d$, and hence that $D$ is a differential. The acyclicity of $(\CA,D)$ is not 
	evident from these arguments.

 Finally, observe that for any right $W$-module $U$ we have
	\begin{equation}\label{eq: homK}
		H_k(\mathcal{K}(U))=H_k(\mathcal{K}(U\otimes \epsilon)), \quad 0\leq k\leq n-1,
	\end{equation}
	 as the boundary maps \eqnref{eq:eq:bdyku} of  $\mathcal{K}(U)$ and $\mathcal{K}(U\otimes \epsilon)$ differ by $-1$.

	\section{Dual complexes} \label{sec: dualcomp}
	
	The principal purpose of this section is to obtain sharper results on the integral homology and cohomology of both $M$, the hyperplane complement,
	and $F$, the non-reduced Milnor fibre. With this in mind we begin by defining a $\Z$-bilinear form on $\CA$, which will later become a tool for moving
	between homology and cohomology.
	
	\subsection{A bilinear form on $\mathcal{A}$}
	We will make frequent use of the following  $\mathbb{Z}$-linear maps in later sections. For each $t\in T$, we define
	\[
	\begin{aligned}
		&d_t: \mathcal{A}_{k}\rightarrow \mathcal{A}_{k-1}, &\quad &a_{t_1}a_{t_2}\dots a_{t_k} \mapsto \sum_{i=1}^{k}(-1)^{k-i}\delta_{t,t_i^{t_{i+1}\dots t_{k}}} a_{t_1}\dots \hat{a}_{t_i}\dots a_{t_k},\\
		&\delta_t: \mathcal{A}_{k}\rightarrow \mathcal{A}_{k-1}, &\quad &a_{t_1}a_{t_2}\dots a_{t_k} \mapsto \sum_{i=1}^{k}(-1)^{i-1}\delta_{t,t_i} a_{t_1^{t_i}}\dots a_{t_{i-1}^{t_i}} \hat{a}_{t_i}\dots a_{t_k},
	\end{aligned}
	\] 
	where  $\delta_{t,t_i}= 1$ if $t=t_i$ and 0 otherwise. It is clear that in the notation of  \eqref{eq: d} and \eqref{eq: del},
	$d=\sum_{t\in T}d_t$ and $\delta=\sum_{t\in T} \delta_t$. Note that $\delta_{t}$ and $d_{t}$ also satisfy the Leibniz rule, i.e.
	\[
	\begin{aligned}
		d_t (a_{t_1}a_{t_2}\dots a_{t_k})&=-d_{t^{t_k}}(a_{t_1}a_{t_2}\dots a_{t_{k-1}})a_{t_k}+ \delta_{t,t_k} a_{t_1}a_{t_2}\dots a_{t_{k-1}},\\
		\delta_t (a_{t_1}a_{t_2}\dots a_{t_k})&=\delta_t(a_{t_1}a_{t_2}\dots a_{t_{k-1}})a_{t_k}+ (-1)^{k-1}\delta_{t,t_k}a_{t_1^{t_k}}a_{t_2^{t_k}}\dots a_{t_{k-1}^{t_k}}. 
	\end{aligned} 
	\]
	for any $(t_1,t_2,\dots, t_k)\in {\rm Rex}_{T}(w)$ with $w\in \mathcal{L}_k$. We call $d_{t}$ and $\delta_t$ \emph{skew derivations}.

	The linear maps $d_t$ and $\delta_t$ are well-defined, for the same reason as in \rmkref{rmk: welldef}.

	\begin{lemma}\label{lem: deld}
		For any $t,t'\in T$, we have $d_{t}\delta_{t'} =\delta_{t'}d_{t}$. 
	\end{lemma}
	\begin{proof}
		We evaluate both sides on $a_{t_1}\dots a_{t_k}\in\CA$ and use induction on $k$. For any nonzero element $a_{t_1}\dots a_{t_k}\in \mathcal{A}$, we  have 
		\[
		\begin{aligned}
			d_{t} \delta_{t'}(a_{t_1}\dots a_{t_k}) =& d_t( \delta_{t'}(a_{t_1} \dots a_{t_{k-1}} )a_{t_k} +(-1)^{k-1} \delta_{t', t_k} a_{t_1^{t_k}}a_{t_2^{t_k}}\dots a_{t_{k-1}^{t_k}}  ) \\
			=& -d_{t^{t_k}} ( \delta_{t'}(a_{t_1}\dots a_{t_{k-1}} ) a_{t_k} + \delta_{t,t_{k}}  \delta_{t'}(a_{t_1}\dots a_{t_{k-1}} ) \\
			&+ (-1)^{k-1} \delta_{t', t_k} d_{t}(a_{t_1^{t_k}}a_{t_2^{t_k}}\dots a_{t_{k-1}^{t_k}}),
		\end{aligned}
		\]
		while on the other hand, 
		\[
		\begin{aligned}
			\delta_{t'}	d_{t} (a_{t_1}\dots a_{t_k}) =& \delta_{t'}(-d_{t^{t_k}}(a_{t_1}a_{t_2}\dots a_{t_{k-1}})a_{t_k}+ \delta_{t,t_k} a_{t_1}a_{t_2}\dots a_{t_{k-1}} )\\
			=& -\delta_{t'}( d_{t^{t_k}} (a_{t_1} \dots a_{t_{k-1}} ) )a_{t_k} +(-1)^{k-1} \delta_{t', t_k} d_{t}(a_{t_1^{t_k}}a_{t_2^{t_k}}\dots a_{t_{k-1}^{t_k}})\\
			&+\delta_{t,t_{k}}  \delta_{t'}(a_{t_1} \dots a_{t_{k-1}} ). 
		\end{aligned}
		\]
		The result is trivial  if $k=1$. For $k>1$, by the induction hypothesis and the equations above we have $ d_{t} \delta_{t'}(a_{t_1}\dots a_{t_k})=\delta_{t'}	d_{t} (a_{t_1}\dots a_{t_k})$. This proves that $d_{t}\delta_{t'} =\delta_{t'}d_{t}$.
	\end{proof}

	\begin{lemma}\label{lem: delrel}
		The skew derivations $\delta_t, t\in T$ satisfy the following relations:
		\[
		\begin{aligned}
			\delta_t^2=\delta_{t_2}\delta_{t_1}&=0, \quad \forall t\in T, t_1t_2\not\leq \gamma, \\
			\sum_{(t_1,t_2)\in {\rm Rex}_T(w)} \delta_{t_2}\delta_{t_1}&=0, \quad  \forall w\in \mathcal{L}_2. 
		\end{aligned}
		\] 
		Therefore, they describe an action of the opposite algebra $\mathcal{A}^{op}$ on $\mathcal{A}$. 
	\end{lemma}
	\begin{proof}  
		For any reflections $r_1,r_2\in T$, we have 
		\[
		\begin{aligned}
			\delta_{r_2}\delta_{r_1}(a_{t_1}a_{t_2}\dots a_{t_k}) &= \delta_{r_2}( \delta_{r_1}(a_{t_1}a_{t_2}\dots a_{t_{k-1}})a_{t_k}+ (-1)^{k-1} \delta_{r_1,t_k} a_{t_1^{t_k}}a_{t_2^{t_k}}\dots a_{t_{k-1}^{t_k}} ) \\
			&=  \delta_{r_2}( \delta_{r_1}(a_{t_1}a_{t_2}\dots a_{t_{k-1}})) a_{t_k} + (-1)^{k-2} \delta_{r_2,t_k} \delta_{r_1^{t_k}}(a_{t_1^{t_k}}a_{t_2^{t_k}}\dots a_{t_{k-1}^{t_k}} ) \\
			&+ (-1)^{k-1} \delta_{r_1,t_k} \delta_{r_2}( a_{t_1^{t_k}}a_{t_2^{t_k}}\dots a_{t_{k-1}^{t_k}} ),
		\end{aligned}
		\]
		where $(t_1, t_2, \dots, t_k)\in {\rm Rex}_T(w)$ for some $w\in \mathcal{L}$.
		We use induction on $k$ to prove the stated relations among $\delta_t,t\in T$ by evaluating the relevant expressions on $a_{t_1}\dots a_{t_k}$. The base case $k=1$ is obvious. For $k>1$, we prove each relation as follows.

		By the induction hypothesis it is easy to see that $\delta_t^2=0$ if $t=r_1=r_2$. 
		
		If $r_1r_2\not \leq \gamma$, then there are two cases. Case 1: neither of $r_1,r_2$ is $t_k$. By the induction hypothesis, $ \delta_{r_2}\delta_{r_1}(a_{t_1}a_{t_2}\dots a_{t_k})= \delta_{r_2}( \delta_{r_1}(a_{t_1}a_{t_2}\dots a_{t_{k-1}})) a_{t_k} =0$. Case 2: exactly one of $r_1,r_2$ equals $t_k$. If $r_1=t_k$, then we have
		\[
		\delta_{r_2}\delta_{r_1}(a_{t_1}a_{t_2}\dots a_{t_k})= \delta_{r_2}( \delta_{r_1}(a_{t_1}a_{t_2}\dots a_{t_{k-1}})) a_{t_k}+ (-1)^{k-1}  \delta_{r_2}( a_{t_1^{t_k}}a_{t_2^{t_k}}\dots a_{t_{k-1}^{t_k}} ).
		\]
		We claim that  $ \delta_{r_2}( a_{t_1^{t_k}}a_{t_2^{t_k}}\dots a_{t_{k-1}^{t_k}} )=0$. Otherwise, by the definition of $\delta_{r_2}$, we have $r_2= t_i^{t_k}$ for some $1\leq i\leq k-1$. However, this leads to $r_1r_2=t_kt_i^{t_k}= t_it_k$, which precedes $\gamma$ by \lemref{lem: BW} and hence violates our assumption that $r_1r_2\not \leq \gamma$. Therefore,  	$ \delta_{r_2}\delta_{r_1}(a_{t_1}a_{t_2}\dots a_{t_k})= \delta_{r_2}( \delta_{r_1}(a_{t_1}a_{t_2}\dots a_{t_{k-1}})) a_{t_k}=0$ by the induction hypothesis. If $r_2=t_k$, the proof is similar. 
		
		For the last relation, by the induction hypothesis this is equivalent to showing that 
		\[
		\sum_{(r_1,r_2)\in {\rm Rex}_T(w)}  \delta_{r_2,t_k} \delta_{r_1^{t_k}}-\delta_{r_1,t_k} \delta_{r_2}=0.
		\]  
		We may assume $w=s_1s_m=s_2s_1=\dots =s_ms_{m-1}$ has $m$ $T$-reduced factorisations. If none of  the $s_i$ is equal to $t_k$, then the above equation holds true. Otherwise, $t_k=s_i$ for some $i=1,\dots, m$. In that case we have 
		\[  \sum_{(r_1,r_2)\in {\rm Rex}_T(w)}  \delta_{r_2,t_k} \delta_{r_1^{t_k}}-\delta_{r_1,t_k} \delta_{r_2}= \sum_{j=1}^m (\delta_{s_{j-1},s_i} \delta_{s_j^{s_i}}- \delta_{s_j,s_i}\delta_{s_{j-1}})=\delta_{s_{i+1}^{s_i}}- \delta_{s_{i-1}}=0, \]
		where  $s_0:=s_m$ for notational convenience. This completes the proof.
	\end{proof}

	\begin{definition}
		Define the bilinear pairing
		\begin{equation*}\label{eq: pairing}
			\langle-,-\rangle: \mathcal{A}\times \mathcal{A} \longrightarrow \mathbb{Z}
		\end{equation*}
		by $\langle 1,1\rangle=1$ and 
		\begin{enumerate}
			\item $\langle \mathcal{A}_k,\mathcal{A}_{\ell}\rangle =0$ for any $0\leq k\neq \ell\leq n$;
			\item For any $x\in \mathcal{A}_k$ and $t_i\in T, i=1,\dots,k$, 
			\[\langle a_{t_1}a_{t_2}\dots a_{t_k}, x\rangle:= \delta_{t_k}\dots \delta_{t_1}(x). \]
		\end{enumerate}
	\end{definition}
	In view of \lemref{lem: delrel}, this bilinear form is well-defined. 
	
	\begin{lemma}\label{lem: adj}
		For any $t\in T$, and $x,y \in \mathcal{A}$ we have 
		\[
		\langle a_{t}x, y \rangle = \langle x, \delta_t(y) \rangle, \quad  \text{and} \quad 	\langle x a_{t}, y \rangle = \langle x, d_t(y) \rangle.
		\]
		Thus with respect to the bilinear form the skew-derivations $\delta_{t}$ and $d_{t}$ are right adjoint to left and right multiplication by $a_{t}$, respectively.
	\end{lemma}
	\begin{proof}
		We assume $x\in \mathcal{A}_{k-1}$ and $y\in \mathcal{A}_k$ for $1\leq k\leq n$. The first adjunction follows immediately from the definition. For the second one, we use induction on $k$. If $k=1$, then $\langle \lambda a_t, \mu a_{t'} \rangle= \lambda\mu \delta_{t,t'}= \langle \lambda , \mu d_{t}(a_{t'}) \rangle$ for any $\lambda,\mu\in \mathbb{Z}$. For $k>1$, we may assume  $x= a_{t_1}\dots a_{t_{k-1}}$. Then 
		\[
		\begin{aligned}
			\langle  a_{t_1}\dots a_{t_{k-1}} a_{t}, y\rangle & = \langle  a_{t_2}\dots a_{t_{k-1}} a_{t}, \delta_{t_1}(y)\rangle = \langle  a_{t_2}\dots a_{t_{k-1}}, d_{t} \delta_{t_1}(y)\rangle\\
			&= \langle  a_{t_2}\dots a_{t_{k-1}},\delta_{t_1} d_{t} (y)\rangle = \langle  a_{t_1}a_{t_2}\dots a_{t_{k-1}}, d_{t} (y)\rangle,
		\end{aligned}
		\]
		where the first and the last equation follow from the adjunction between $\delta_t$ and left multiplication by $a_{t}$, the second equation follows from the induction hypothesis, and the third equation follows from \lemref{lem: deld}. The proof is complete.
	\end{proof}

	\begin{proposition}\label{prop: bilinear}
		The bilinear form $\langle-,-\rangle$ is unimodular, i.e., it induces an isomorphism $\mathcal{A}\cong \mathcal{A}^*=\Hom_\Z(\CA,\Z)$, given by $x \mapsto \langle x,-\rangle$ for any $x\in \mathcal{A}$.
	\end{proposition}
	\begin{proof}
		We only need to  show that the bilinear form induces an  isomorphism $\mathcal{A}_k\cong \mathcal{A}_k^*$ for $0\leq k\leq n$. Recall that $\mathcal{A}_k= \bigoplus_{w\in \mathcal{L}_k} \mathcal{A}_w$.  It suffices to prove the following claims:
		\begin{enumerate}
			\item $\langle \mathcal{A}_v, \mathcal{A}_w\rangle =0$ for any two elements $v\neq w$ of $\mathcal{L}_k$;
			\item For any $w\in \mathcal{L}$, the bilinear form induces an  isomorphism $\mathcal{A}_w\cong \mathcal{A}_w^*$. 
		\end{enumerate}
		
		To prove claim (1), we use induction on $\ell_{T}(v)=\ell_{T}(w)=k$. It is trivial if $k=1$. For $k>1$, assume that $x= x'a_{t}\in \mathcal{A}_v$. Then for any $y\in \mathcal{A}_w$ with $w\neq v$ we have 
		\[ \langle x,y\rangle = \langle x'a_t,y\rangle =\langle x', d_t(y)\rangle. \]
		Suppose that $y= a_{t_1}a_{t_2}\dots a_{t_k}$ with $w=t_1t_2\dots t_k$ being a $T$-reduced expression. Then by the definition,
		\[ d_t(y)= \sum_{i=1}^k(-1)^{k-i} \delta_{t, r_i } a_{t_1}\dots a_{t_{i-1}}a_{t_{i+1}}\dots a_{t_k}, \]
		where $r_i=t_i^{t_{i+1}\dots t_k}$. If $t\neq r_i$ for any $i$, then $d_t(y)=0$ and hence $\langle x,y\rangle=0$. Otherwise, there exists a unique $i_0$ such that $t=r_{i_0}$. Assume that $x'=a_{t'_1} a_{t'_2}\dots a_{t'_{k-1}}$ with $v= t'_1t'_2\dots t'_{k-1}t$ being a $T$-reduced expression. We have 
		\[ \langle x,y\rangle =\langle x', d_t(y)\rangle = (-1)^{k-i_0}\langle a_{t'_1} a_{t'_2}\dots a_{t'_{k-1}},  a_{t_1}\dots a_{t_{i_0-1}}a_{t_{i_0+1}}\dots a_{t_k}\rangle.  \]
		Assume with a view to obtaining a contradiction that $\langle x,y\rangle \neq 0$. By the induction hypothesis we have $ t'_1t'_2\dots t'_{k-1}=t_1\dots t_{i_0-1}t_{i_0+1}\dots t_k$. It follows that 
		\[ v= (t'_1t'_2\dots t'_{k-1})t=t_1\dots t_{i_0-1}t_{i_0+1}\dots t_k r_{i_0}=t_1t_2\dots t_k=w, \]
		which leads to a contradiction. Therefore, we have $\langle x,y\rangle =0$ for any $x\in \mathcal{A}_v$ and $y \in \mathcal{A}_w$ with $w\neq v$. 
		
		We proceed next to prove claim (2). To this end, we need to introduce a lexicographical order on the  basis of $\mathcal{A}_w$, and then show that the corresponding matrix of the bilinear form is upper triangular with diagonal entries being 1.

		Recall from \propref{prop: Abasis} that $\mathcal{A}_w$ has a basis consisting of elements $a_{t_1}a_{t_2}\dots a_{t_k}$, where $w=t_1t_2\dots t_k$ is a $T$-reduced expression and $t_1\succ t_2\succ \dots \succ t_k$ with respect to the total order of $T$. We define the lexicographical order on the  basis by
		\begin{equation}\label{eq: lexorder}
					a_{t_1}a_{t_2}\dots a_{t_k}< a_{t'_1}a_{t'_2}\dots a_{t'_k}   \iff t_i\prec  t'_i, 
		\end{equation}
		where $i$ is the first place where the two monomials differ. Let $\{m_1<m_2<\dots <m_p\}$ be the totally ordered set of the basis of $\mathcal{A}_{w}$, and let $M$ be the matrix of the bilinear form with the $(i,j)$-th entry being $\langle m_i,m_j\rangle $. Next we show that $M$ is a upper-triangular matrix with all diagonal entires equal to $1$. 
		
		To show that $M$ is upper-triangular, we need to prove that 
		\begin{equation}\label{eq: upptri}
			\langle a_{t'_1}a_{t'_2}\dots a_{t'_k}, a_{t_1}a_{t_2}\dots a_{t_k}\rangle =0 
		\end{equation}
		for any two basis elements $a_{t_1}a_{t_2}\dots a_{t_k}< a_{t'_1}a_{t'_2}\dots a_{t'_k}$ of $\mathcal{A}_w$. Note that $t_1\succ t_2\succ \dots \succ t_k$ and $t'_1\succ t'_2\succ \dots\succ t'_k$ with respect to the total order on $T$. Assuming that $t_1=t'_1, \dots, t_{i-1}=t'_{i-1}$ and $t_{i}\prec t'_{i}$, we have 
		\[ 
		\begin{aligned}
			\langle a_{t'_1}a_{t'_2}\dots a_{t'_k}, a_{t_1}a_{t_2}\dots a_{t_k}\rangle&=\langle  a_{t'_2}\dots a_{t'_k}, \delta_{t'_1}(a_{t_1}a_{t_2}\dots a_{t_k}) \rangle  \\
			&=\langle a_{t'_2}\dots a_{t'_k}, a_{t_2}\dots a_{t_k}\rangle,\\
			&\cdots\\
			&=\langle a_{t'_i}a_{t'_{i+1}}\dots a_{t'_k},  a_{t_i}a_{t_{i+1}}\dots a_{t_k}\rangle,
		\end{aligned}  
		\]
		where in the first line we have used the adjoint property from \lemref{lem: adj}, and in the third line we repeat the process until we obtain the last equation. 
		Using the adjoint property again, we have
		\[ \langle a_{t'_i}a_{t'_{i+1}}\dots a_{t'_k},  a_{t_i}a_{t_{i+1}}\dots a_{t_k}\rangle = \langle a_{t'_{i+1}}\dots a_{t'_k}, \delta_{t'_i}(a_{t_i}a_{t_{i+1}}\dots a_{t_k}) \rangle. \]
		Since $t'_i\succ t_i \succ t_{i+1}\succ \dots \succ t_k$, by the definition of $\delta_{t'_i}$  we have 
		$\delta_{t'_i}(a_{t_i}a_{t_{i+1}}a_{t_2}\dots a_{t_k})=0$.
		Hence we have $\langle a_{t'_1}a_{t'_2}\dots a_{t'_k}, a_{t_1}a_{t_2}\dots a_{t_k}\rangle =0$.
		
		Applying the adjunction between $\delta_t$ and left multiplication by $a_t$ repeatedly,  we have 
		\begin{equation}\label{eq: diagone}
				\begin{aligned}
			\langle a_{t_1}a_{t_2}\dots a_{t_k}, a_{t_1}a_{t_2}\dots a_{t_k}\rangle &= \langle a_{t_2}\dots a_{t_k}, \delta_{t_1}(a_{t_1}a_{t_2}\dots a_{t_k})\rangle \\
			&= \langle a_{t_2}\dots a_{t_k}, a_{t_2}\dots a_{t_k}\rangle \\
			&=\dots =1
		\end{aligned}
		\end{equation}
		for any basis element $a_{t_1}a_{t_2}\dots a_{t_k}$ of $\mathcal{A}_w$. Hence the diagonal entries of $M$ are all equal to $1$. 
		Therefore, the bilinear form is unimodular on $\mathcal{A}_w$. Combing this with claim (1), we complete the proof. 
	\end{proof}
	
	\subsection{Acyclic dual complexes}
	In this subsection, we introduce some acyclic cochain complexes which are dual to the previous acyclic chain complexes with respect to suitable bilinear forms. 

	We recall the universal coefficient theorem for cohomology, which will be used later.
	
	\begin{theorem}\label{thm: uct} \cite[Section 3.1]{Hat02}
		Let $G$ be an abelian group, and $\mathcal{C}$ be the following chain complex of free abelian groups
		\[
		0 \lr C_n \overset{\partial}{\lr} C_{n-1} \lr \cdots \lr C_1 \overset{\partial}{\lr} C_0\lr 0. 
		\]
		Let $C_k^*={\rm Hom}_{\mathbb{Z}}(C_k,G)$ be the dual of the chain group $C_k$ and $\partial^*: C_{k-1}^*\rightarrow C^*_{k}$ be the dual coboundary map for $1\leq k\leq n$.  Then the cohomology groups $H^k(\mathcal{C}; G)$ of the cochain complex 
		\[
		0 \lr C^*_0 \overset{\partial^*}{\lr} C^*_{1} \lr \cdots \lr C^*_{n-1} \overset{\partial^*}{\lr} C^*_n \lr 0.    
		\] 
		are determined by split exact sequences
		\[ 0 \longrightarrow {\rm Ext}(H_{k-1}(\mathcal{C}), G) \longrightarrow  H^k(\mathcal{C}; G) \longrightarrow  {\rm Hom}_{\mathbb{Z}}(H_k(\mathcal{C}), G )\longrightarrow 0.  \]
	\end{theorem}
	
	We are only concerned with the case $G=\mathbb{Z}$. If the homology groups $H_k(\mathcal{C})$ are finitely generated abelian groups with torsion subgroups $T_k\subseteq H_k(\mathcal{C})$, then the homology of the chain complex $(\mathcal{C}, \partial)$ and  the cohomology of the dualised chain complex  $(\mathcal{C}^*, \partial^*)$ are related by
	\[H^k(\mathcal{C}^*; \mathbb{Z})\cong  H_k(\mathcal{C})/T_k \oplus T_{k-1}. \] 
	In particular, if $(\mathcal{C}, \partial)$ is acyclic then so is $(\mathcal{C}^*, \partial^*)$. If $\mathcal{C}$ consists of finitely generated free abelian groups, then $(\mathcal{C}^{**}, \partial^{**})\cong (\mathcal{C}, \partial)$. 
	
	\subsubsection{A pair of acyclic complexes}
	Using the bilinear form on $\mathcal{A}$, we define acyclic cochain complexes which are dual to the acyclic chain complexes $(\mathcal{A}, d)$ and $(\mathcal{A},\delta)$.
	
	Define the element
	$$\omega:= \sum_{t\in T} a_t \in \mathcal{A}_1.$$ 
	It is easily verified that $\omega^2=0$. Hence it gives rise to a cochain complex whose coboundary maps are  right multiplication by $\omega$:  
	\begin{equation*}\label{eq: cochainA}
		0\lr \CA_0\overset{r_\omega}{\lr}\CA_1\lr \cdots \overset{r_\omega}{\lr} \CA_n\lr 0
	\end{equation*} 
	We denote this complex by $(\mathcal{A}, r_\omega)$. Similarly, we define the complex $(\mathcal{A},\ell_{\omega})$, where $\ell_{\omega}$ denotes left multiplication  by $\omega$.

 Recall from \propref{prop: bilinear} that the  bilinear form  induces an  isomorphism of graded free abelian groups:
	\begin{equation} \label{eq: psi}
		\psi: \mathcal{A} \rightarrow \mathcal{A}^{\ast}, \quad x \mapsto \psi(x):=\langle-,x\rangle,
	\end{equation}
	where $\psi(x)(x')= \langle x',x\rangle$  for any $x,x'\in \mathcal{A}_k$. Therefore,  for any linear form $\lambda\in \mathcal{A}^{\ast}$, there exists a unique $x\in \mathcal{A}$ such that $\lambda=\psi(x)$.  In what follows,  a linear map on $\mathcal{A}^*$ will be defined by its action on elements of the form $\psi(x), x\in \mathcal{A}$.
	

	\begin{lemma}\label{lem: isocpl}
		The linear map $\psi$ induces an isomorphism between chain complexes $(\mathcal{A}, d)$ and $(\mathcal{A}^{\ast}, r_{\omega}^{\ast})$, where  $r_{\omega}^{\ast}$ is the adjoint map of $r_{\omega}$ defined by 
		\[
		r_{\omega}^{\ast}(\psi(y))(x): = \psi(y)(r_{\omega}(x))=\langle x\omega,y\rangle, \quad \forall x\in \mathcal{A}_{k-1},y \in \mathcal{A}_{k}. 
		\]
		Similarly, $\psi$ induces an isomorphism between chain complexes $(\mathcal{A}, \delta)$ and $(\mathcal{A}^{\ast}, \ell_{\omega}^{\ast})$
	\end{lemma}
	
	\begin{proof}
		For any integer $k$, it suffices to prove that the following diagram commutes: 
		\begin{center}
			\begin{tikzcd}
				\mathcal{A}_k \arrow[r, "d"] \arrow[d, "\psi"]
				&  \mathcal{A}_{k-1} \arrow[d, "\psi"] \\
				\mathcal{A}^*_k  \arrow[r,  "r_{\omega}^*" ]
				&  \mathcal{A}^*_{k-1} \end{tikzcd}
		\end{center}
		We verify this directly. For any $x\in \mathcal{A}_k, y \in \mathcal{A}_{k-1}$, by the definition of $r_{\omega}^{\ast}$ we have 
		\[  ((r_{\omega}^*\psi)(x))(y)=  r_{\omega}^{\ast}(\psi(x))(y)=\langle y\omega,x\rangle.   \]
		On the other hand, we have 
		\[
		((\psi d)(x))(y)= \psi(dx)(y)=\langle y,dx\rangle.  
		\]
		Using the adjointness property from \lemref{lem: adj}, we have $\langle y\omega,x\rangle=\langle y,dx\rangle$, whence $r_{\omega}^*\psi=\psi d$. Therefore,  $\psi$ is a chain map between  $(\mathcal{A}, d)$ and $(\mathcal{A}^{\ast}, r_{\omega}^{\ast})$. As $\psi$ is a graded isomorphism,  $(\mathcal{A}, d)$ and $(\mathcal{A}^{\ast}, r_{\omega}^{\ast})$ are isomorphic. The isomorphism between  $(\mathcal{A}, \delta)$ and $(\mathcal{A}^{\ast}, \ell_{\omega}^{\ast})$ can be proved similarly.
	\end{proof}

\begin{proposition}\label{prop: acyccocx}
		The cochain complexes $(\mathcal{A}, r_\omega)$ and $(\mathcal{A},\ell_{\omega})$ are both acyclic. 
\end{proposition}
\begin{proof}
		Recall from \propref{prop: propB} that the chain complex $(\mathcal{A},d)$ is acyclic. It follows from \lemref{lem: isocpl} that the chain complex $(\mathcal{A}^*,r_{\omega}^*)$ is acyclic.  Since $(\mathcal{A}^*,r_{\omega}^*)$  is the dual complex of $(\mathcal{A}, r_{\omega})$ induced by the isomorphism $\psi$, the cochain complex $(\mathcal{A}, r_{\omega})$ is acyclic by the universal coefficient theorem for cohomology. Similarly, one can prove that $(\mathcal{A}, \ell_{\omega})$ is acyclic, using the acyclicity of $(\mathcal{A}, \delta)$ from \propref{prop: delta}.
\end{proof}

	\subsubsection{Another pair of acyclic complexes}
	
	Recall from \thmref{thm: homoM} that the boundary maps of the complex $(\Z W\ot\CA_k, \partial_k)$ are given by $\partial_k=\partial'_k+(-1)^k\partial''_k$, where the linear maps $\partial', \partial'': \mathbb{Z}W\otimes \mathcal{A}_k \rightarrow  \mathbb{Z}W\otimes \mathcal{A}_{k-1}$ are defined  by
	\begin{equation}\label{eq: defpartials}
		\begin{aligned}
			\partial'(w\ot a_{t_1}a_{t_2}\dots a_{t_k})=&\sum_{i=1}^k(-1)^{i-1}wt_i\ot a_{t_1^{t_i}}\dots a_{t_{i-1}^{t_i}}\widehat{a}_{t_i} a_{t_{i+1}}\dots a_{t_k}, \\
			\partial''(w\ot a_{t_1}a_{t_2}\dots a_{t_k})=&\sum_{i=1}^k(-1)^{k-i}w\ot a_{t_1}\cdots \widehat{a}_{t_i}\dots a_{t_k}.\\
		\end{aligned}
	\end{equation}
	In what follows,  whenever dealing with (co)boundary maps,  $\mathbb{Z}W$ is viewed as a right $\mathbb{Z}W$-module with any $t\in T$ acts on the right side.

	\begin{lemma}\label{lem: pars}
		Let $ \partial',  \partial''$ be as in \eqref{eq: defpartials}. Then  we have 
		\[\partial'= \sum_{t\in T} t\otimes \delta_t, \quad \partial''=\sum_{t\in T} 1\otimes d_t=1\otimes d. \]
		Therefore, $\partial', \partial''$ satisfy  $(\partial')^2=(\partial'')^2=0 $ and  $ \partial'\partial''=\partial''\partial'.$
	\end{lemma}
	\begin{proof}
		For any nonzero element $a_{t_1}\dots a_{t_k}$ we have 
		\[
		\begin{aligned}
			\sum_{t\in T} t\otimes \delta_t(w\otimes a_{t_1}a_{t_2}\dots a_{t_k})=& \sum_{t\in T}\sum_{i=1}^k (-1)^{i-1} wt\otimes \delta_{t,t_i} a_{t_1^{t_i}}\dots a_{t_{i-1}^{t_i}}\widehat{a}_{t_i} a_{t_{i+1}}\dots a_{t_k}\\
			=& \sum_{i=1}^k (-1)^{i-1} wt_i\otimes a_{t_1^{t_i}}\dots a_{t_{i-1}^{t_i}}\widehat{a}_{t_i} a_{t_{i+1}}\dots a_{t_k}.
		\end{aligned}
		\]
		It follows that $\partial'= \sum_{t\in T} t\otimes \delta_t$, and similarly, $\partial''=\sum_{t\in T} 1\otimes d_t$. Using the fact that $\delta_t^2=d_t^2=0$ for any $t\in T$, we obtain  $(\partial')^2=(\partial'')^2=0 $. By \lemref{lem: deld} $d_{t}$ and $\delta_{t'}$ commute with each other, we have 
		\[
		\partial'\partial''= \sum_{t,t'\in T} t\otimes \delta_t d_{t'}= \sum_{t,t'\in T} t\otimes d_{t'} \delta_t = \partial''\partial'.
		\]  \end{proof}

	By \lemref{lem: pars}, we have a pair of chain complexes $(\mathbb{Z}W\otimes \mathcal{A}, \partial')$ and $(\mathbb{Z}W\otimes \mathcal{A}, \partial'')$. We proceed next to define a bilinear form on $\mathbb{Z}W\otimes \mathcal{A}$, which enables us to dualise the chain complexes.

	We will always denote by $\langle-,-\rangle$ the bilinear form on a linear space whenever no confusion arises. Recall that there is a standard $\mathbb{Z}$-bilinear form on $\mathbb{Z}W$ defined by 
	\[
	\langle-,-\rangle: \mathbb{Z}W\times \mathbb{Z}W\rightarrow \mathbb{Z}, \quad  \langle v,w \rangle:= \delta_{v,w}, \quad \forall v,w\in W,
	\]
	where $\delta$ is the Kronecker delta. This together with the bilinear form on $\mathcal{A}$ which we have already defined, gives rise to  a $\mathbb{Z}$-bilinear form $\langle-,-\rangle$ on $\mathbb{Z}W\otimes \mathcal{A}$, defined by
	\[
	\langle v\otimes x, w\otimes y\rangle:= \langle v,w \rangle\,\langle x,y \rangle, \quad \forall v,w\in W, x,y\in \mathcal{A}.
	\] 
	Note that $\mathbb{Z}W\otimes \mathcal{A}$  is a $\mathbb{Z}$-graded algebra, with the grading inherited from $\mathcal{A}$ and multiplication given by $( v\otimes x) (w\otimes y) =vw\otimes xy$. By definition, we have  $\langle \mathbb{Z}W\otimes  \mathcal{A}_k, \mathbb{Z}W\otimes  \mathcal{A}_{\ell} \rangle=0 $ for any integers $k\ne \ell$.

	It follows from \propref{prop: bilinear} that the bilinear form on $\mathbb{Z}W\otimes \mathcal{A}$ is also unimodular. Hence we have the following graded  isomorphism of free abelian groups:
	\begin{equation}\label{eq: psiW}
		\psi_W: \mathbb{Z}W\otimes \mathcal{A} \rightarrow (\mathbb{Z}W\otimes \mathcal{A})^*, \quad w\otimes x \mapsto  \psi_W(w\otimes x):= \langle -, w\otimes x\rangle,	
	\end{equation}
	where $w\in W$ and $x\in \mathcal{A}_k$ for $0\leq k\leq n$.

	We introduce the following elements:
	\begin{equation*}\label{eq: sigs}
		\sigma=\sum_{t\in T} t\otimes a_{t}, \quad \varsigma= \sum_{t\in T}1\otimes a_t= 1\otimes \omega.
	\end{equation*}
	It is easily verified that $\sigma^2=\varsigma^2=0$ in $\mathbb{Z}W\otimes \mathcal{A}$. Hence $\sigma$ and $\varsigma$ give rise to two cochain complexes $(\mathbb{Z}W\otimes \mathcal{A}, \ell_\sigma)$ and $(\mathbb{Z}W\otimes \mathcal{A}, r_\varsigma)$, where the coboundary maps are given by 
	\begin{equation*}\label{eq: sigtau}
		\ell_\sigma(w\otimes x)= \sum_{t\in T} wt\otimes a_{t}x,  
		\quad 	r_\varsigma(w\otimes x) =  \sum_{t\in T} w\otimes xa_{t}.   
	\end{equation*}
	for any $w\otimes x \in \mathbb{Z}W\otimes \mathcal{A}_k$. 
	
		\begin{lemma}\label{lem: adjW}
		For any $v, w\in W$ and  $x,y\in \mathcal{A}$, we have
		\[
		\begin{aligned}
			\langle \ell_\sigma( v\otimes x), w\otimes y\rangle&=\langle v\otimes x, \partial'(w\otimes y)\rangle,    \\
			\langle r_\varsigma(v\otimes x), w\otimes y\rangle&=\langle v\otimes x, \partial''(w\otimes y)\rangle.    
		\end{aligned}
		\]
		Therefore, with respect to the bilinear form on $\mathbb{Z}W\otimes \mathcal{A}$,  $\partial'$ and $\partial''$ are right adjoint to the linear operators $\ell_{\sigma}$ and $r_{\varsigma}$, respectively. 
	\end{lemma}
	\begin{proof}
		We only prove the first adjunction; the second one can be treated similarly. 
		Assume that $x\in \mathcal{A}_{k-1}$ and $y\in \mathcal{A}_k$ for some $k=0,\dots n$. Then we have
		\[
		\begin{aligned}
			\langle \ell_\sigma(v\otimes x), w \rangle=&\sum_{t\in T}\langle  vt\otimes a_tx, w\otimes y\rangle = \sum_{t\in T} \langle vt, w\rangle \,\langle a_tx,y) \\
			=& \sum_{t\in T} \langle v, wt\rangle\,\langle x, \delta_t(y)) 
			= \sum_{t\in T}\langle v\otimes x, (t\otimes \delta_t)(w\otimes y)\rangle \\
			=&\langle v\otimes x, \partial'(w\otimes y)\rangle,
		\end{aligned}  
		\]
		where the third equation follows from the adjoint property in \lemref{lem: adj}, and last equation is a consequence of \lemref{lem: pars}. 
	\end{proof}

\begin{proposition}\label{prop: aycpar}
	The chain complexes $(\mathbb{Z}W\otimes \mathcal{A}, \partial')$ and $(\mathbb{Z}W\otimes \mathcal{A}, \partial'')$ are both acyclic.
\end{proposition}
\begin{proof}{}
By \lemref{lem: pars}, we have $\partial''=1\otimes d$. As $\mathbb{Z}W$ is a flat $\mathbb{Z}$-module, the acyclicity of $(\mathbb{Z}W\otimes \mathcal{A}, \partial'')$ follows from that of $(\mathcal{A}, d)$, which is given in \propref{prop: dprop}. 

It remains to prove that  $(\mathbb{Z}W\otimes \mathcal{A}, \partial')$ is acyclic. Recall from \eqnref{eq: kap} the Kereweras automorphism $\kappa: \mathcal{A}\rightarrow \mathcal{A}$ whose inverse is given by  $\kappa^{-1}(a_{t_1}a_{t_2}\dots a_{t_k})= a_{t_k} a_{t_{k-1}}^{t_k}\dots a_{t_1}^{t_2\dots t_k}$. Note that this is a graded automorphism.  Applying $1\otimes \kappa^{-1}$ to the chain complex $(\mathbb{Z}W\otimes \mathcal{A}, \partial')$, we obtain a new chain complex  $(\mathbb{Z}W\otimes \mathcal{A}, \widetilde{\partial}')$ whose boundary map satisfies $  \widetilde{\partial}'(1\otimes \kappa^{-1})= (1\otimes \kappa^{-1}) \partial'$, defined explicitly by 
\begin{equation}\label{eq: wtdpar}
	 \widetilde{\partial}'(w \otimes a_{t_1}a_{t_2}\dots a_{t_k})= \sum_{i=1}^k(-1)^{k-i} wt_i^{t_{i-1}\dots t_1}\otimes a_{t_1}\dots \hat{a}_{t_i}\dots a_{t_k} 
\end{equation}
for any $w\in W$ and nonzero element $a_{t_1}a_{t_2}\dots a_{t_k}\in \mathcal{A}_k$. Hence we are reduced to proving the acyclicity of $(\mathbb{Z}W\otimes \mathcal{A}, \widetilde{\partial}')$. 

For each $k=0,1, \dots n$, we have the following split exact sequence
\begin{equation}\label{eq: exseq}
   0\rightarrow {\rm Ker}\, \widetilde{\partial}'_k  \rightarrow \mathbb{Z}W\otimes \mathcal{A}_k 
   \rightarrow {\rm Im}\, \widetilde{\partial}'_k  \rightarrow 0.
\end{equation}
To prove that the homology $H_k= {\rm Ker}\, \widetilde{\partial}'_k /{\rm Im}\, \widetilde{\partial}'_{k+1}$ is trivial, we shall determine each image ${\rm Im}\, \widetilde{\partial}'_k$, and then use the exact sequence above to show that ${\rm Im}\, \widetilde{\partial}'_{k+1}$ and ${\rm Ker}\, \widetilde{\partial}'_k $ have equal ranks for all $k$. Finally, we show that the  homology is torsion free and hence is trivial. 

Let us start with a combinatorial description of the chain group $\mathbb{Z}W\otimes \mathcal{A}_k$. As the chain complex $(\mathcal{A}, d)$ is acyclic, we have  decompositions of free abelian groups $\mathcal{A}_k=d(\mathcal{A}_k)\oplus d(\mathcal{A}_{k+1})$ for $0\leq k\leq n-1$. It is proved in \cite[Theorem 4.8]{Zha22} that $d(\mathcal{A}_k)$ has a $\mathbb{Z}$-basis consisting of elements $d(a_{t_1}\dots a_{t_k})$ for $1\leq k\leq n$, where $(t_1, \dots ,t_k)\in \mathcal{D}_{[k-1]}$ with $\mathcal{D}_{[k-1]}$ defined by 
\[
   \mathcal{D}_{[k-1]}:=\Big\{ (t_1,\dots, t_{k}) \Bigm|
     \begin{matrix}
       &  \text{\, $\gamma=t_1t_2\dots t_n$ is $T$-reduced and \, } \\
       &t_1\succ \dots \succ t_{k-1}\succ t_{k}\prec t_{k+1}\prec \dots\prec t_{n}
     \end{matrix}
      \Big\}. 
\]
Then we have 
\[ {\rm rank}(\mathbb{Z}W\otimes \mathcal{A}_k)= |W|(\mathcal{D}_{[k]}+ \mathcal{D}_{[k-1]}), \quad 0\leq k\leq n. \]
Moreover, recall from \corref{coro: basis} that $\mathcal{A}_k$ has a $\mathbb{Z}$-basis $\{ a_{{\bf t}}= a_{t_1}\dots a_{t_k}| {\bf t}=(t_1, \dots t_k)\in  \bigcup_{u\in \mathcal{L}_k} \mathcal{D}_u\}$, where $\mathcal{D}_{u}$ is defined as in \eqnref{eq: decw}. Therefore,  $\{ w\otimes a_{{\bf t}}| w\in W, {\bf t}\in \bigcup_{u\in \mathcal{L}_k} \mathcal{D}_u\}$ is a $\mathbb{Z}$-basis for $\mathbb{Z}W\otimes \mathcal{A}_k$. 

Now we consider the image ${\rm Im}\, \widetilde{\partial}'_k$. 
Denote by  $B_k\subseteq \mathbb{Z}W\otimes \mathcal{A}_{k-1}$ the free abelian group spanned by the set $S_{[k-1]}:=\{ \widetilde{\partial}'(w\otimes a_{{\bf t}})\mid w\in W, {\bf t}\in \mathcal{D}_{[k-1]} \}$. Then we have $B_k \subseteq {\rm Im}\, \widetilde{\partial}'_k$.  We shall prove that the set $S_{[k-1]}$ is $\mathbb{Z}$-linearly independent. To this end, we choose a total order on the basis of $\mathbb{Z}W\otimes \mathcal{A}_k$ for each $k$ such that
\begin{equation}\label{eq: totalord}
	  w\otimes a_{{\bf t}} <  v \otimes a_{{\bf t}'} \text{\quad whenever \quad} a_{{\bf t}}< a_{{\bf t}'},
\end{equation}
for any $v,w\in W$ and ${\bf t}, {\bf t}'\in \bigcup_{u\in \mathcal{L}_k} \mathcal{D}_u$, where $a_{{\bf t}}< a_{{\bf t}'}$ is the total order defined in \eqref{eq: lexorder}. Then it follows from \eqnref{eq: upptri} that for any  pair of basis elements $w\otimes a_{{\bf t}} <  v \otimes a_{{\bf t}'}$
\begin{equation}\label{eq: lower}
	\langle v \otimes a_{{\bf t}'}, w\otimes a_{{\bf t}}\rangle =0. 
\end{equation}
Now assume that there exist nonzero $
  \lambda_{w_0,{\bf t}_0} , \lambda_{w, {\bf t}}\in \mathbb{Z}$ such that 
\[
  \lambda_{w_0,{\bf t}_0}  \widetilde{\partial}'(w_0\otimes a_{{\bf t}_0})= \sum_{ w_0\otimes a_{{\bf t}_0} > w\otimes a_{{\bf t}}} \lambda_{w,{\bf t}} \widetilde{\partial}'(w\otimes a_{{\bf t}}),
\]
for some elements $\widetilde{\partial}'(w\otimes a_{{\bf t}})$ and $\widetilde{\partial}'(w_0\otimes a_{{\bf t}_0})$ of the set $S_{[k-1]}$. In view of the definition \eqref{eq: wtdpar}, for any element $w\otimes a_{{\bf t}}=w\otimes a_{t_1}\dots a_{t_k}$ with $t_1\succ \dots \succ t_k$, the element 
\[
w\otimes a_{{\bf t}}(\hat{k}):= wt_k^{t_{k-1}\dots t_1}\otimes a_{t_1}\dots a_{t_{k-1}}
\]
 is the unique maximal term under the total order \eqnref{eq: totalord} in the expression of $\widetilde{\partial}'(w\otimes a_{t_1}\dots a_{t_k})$. Hence by \eqnref{eq: lower} and \eqnref{eq: diagone} we have $\langle  w\otimes a_{{\bf t}}(\hat{k}), \widetilde{\partial}'(w\otimes a_{{\bf t}})\rangle=1$.  Moreover, observe that if $w_0\otimes a_{{\bf t}_0} > w\otimes a_{{\bf t}}$ for ${\bf t}, {\bf t}_0\in \mathcal{D}_{[k-1]}$,  then $w_0\otimes a_{{\bf t}_0}(\hat{k})> w\otimes a_{{\bf t}}(\hat{k})$ (cf. \cite[Proposition 4.7]{Zha22}). This further implies that $\langle w_0\otimes a_{{\bf t}_0}(\hat{k}),  \widetilde{\partial}'(w\otimes a_{{\bf t}})\rangle=0$ by \eqnref{eq: lower}. Therefore, we obtain 
 \[ 
 \begin{aligned}
   \lambda_{w_0,{\bf t}_0}&=\langle w_0\otimes a_{{\bf t}_0}(\hat{k}),  \lambda_{w_0,{\bf t}_0}\widetilde{\partial}'(w_0\otimes a_{{\bf t}_0})\rangle \\
   &= \sum_{ w_0\otimes a_{{\bf t}_0} > w\otimes a_{{\bf t}}} \langle w_0\otimes a_{{\bf t}_0}(\hat{k}), \lambda_{w,{\bf t}}  \widetilde{\partial}'(w\otimes a_{{\bf t}})\rangle\\
   &=0,
 \end{aligned}
  \]
  which contradicts our assumption that $\lambda_{w, {\bf t}_0}\neq 0$. Therefore, $S_{[k-1]}$ is a $\mathbb{Z}$-linearly independent set, which spans the free abelian subgroup $B_k\subseteq {\rm Im}\, \widetilde{\partial}'_k$. 

Now by the exact sequence \eqnref{eq: exseq} we obtain 
\[
\begin{aligned}
  {\rm rank}\, {\rm Ker}\, \widetilde{\partial}'_k &= {\rm rank}\, \mathbb{Z}W\otimes \mathcal{A}_k- {\rm rank}\, {\rm Im}\, \widetilde{\partial}'_k \leq {\rm rank}\, \mathbb{Z}W\otimes \mathcal{A}_k- {\rm rank}\,B_k\\
  &=|W|(\mathcal{D}_{[k]}+ \mathcal{D}_{[k-1]})- |W||\mathcal{D}_{[k-1]}|\\
  &=  |W||\mathcal{D}_{[k]}|. 
\end{aligned}
\]
On the other hand, since $B_{k+1}\subseteq {\rm Im}\, \widetilde{\partial}'_{k+1} \subseteq {\rm Ker}\, \widetilde{\partial}'_k$, we have
\[
  {\rm rank}\, {\rm Ker}\, \widetilde{\partial}'_k \geq {\rm rank}\,{\rm Im}\, \widetilde{\partial}'_{k+1} \geq {\rm rank}\,B_{k+1}= |W||\mathcal{D}_{[k]}|.
\]
Combing the above two inequalities, for each $k$ we have $ {\rm rank}\, {\rm Ker}\, \widetilde{\partial}'_k=|W||\mathcal{D}_{[k]}|$, and 
\[{\rm rank}\, {\rm Im}\, \widetilde{\partial}'_k={\rm rank}\, \mathbb{Z}W\otimes \mathcal{A}_k-{\rm rank}\, {\rm Ker}\, \widetilde{\partial}'_k= |W||\mathcal{D}_{[k-1]}|.\]
 It follows that ${\rm Im}\, \widetilde{\partial}'_{k+1}$ and ${\rm Ker}\, \widetilde{\partial}'_k $ have equal rank. 


It remains to show that the homology $H_k={\rm Ker}\, \widetilde{\partial}'_k/{\rm Im}\, \widetilde{\partial}'_{k+1}$ is trivial.
We aim to prove that $B_{k+1}={\rm Ker}\, \widetilde{\partial}'_k={\rm Im}\, \widetilde{\partial}'_{k+1}$. 
By the above arguments on ranks,  $B_{k+1}\subseteq {\rm Ker}\, \widetilde{\partial}'_{k}$ is  a free abelian subgroup of maximal rank. Therefore, every element of the quotient group ${\rm Ker}\, \widetilde{\partial}'_{k}/B_{k+1}$ has finite order. For any $\beta\in {\rm Ker}\, \widetilde{\partial}'_{k}$, there exists a nonzero integer $m$ such that 
\begin{equation}\label{eq: mbeta}
	 m\beta = \sum_{w\in W, {\bf t}\in \mathcal{D}_{[k]}} \lambda_{w, {\bf t}} \widetilde{\partial}'_{k+1}(w\otimes a_{{\bf t}})\in B_{k+1}, \quad \lambda_{w, {\bf t}}\in \mathbb{Z}.
\end{equation}
Suppose that $w_0\otimes {\bf t}_0$ is the biggest element under the total order \eqref{eq: totalord} such that $m \nmid \lambda_{w_0, {\bf t}_0}$. Then 
\[
 m\beta- \sum_{w\otimes {\bf t}> w_0\otimes {\bf t}_0 }\lambda_{w, {\bf t}} \widetilde{\partial}'_{k+1}(w\otimes a_{{\bf t}})  = \lambda_{w_0, {\bf t}_0} \widetilde{\partial}'_{k+1}(w_0\otimes a_{{\bf t}_0})+ \sum_{w\otimes {\bf t}< w_0\otimes {\bf t}_0 } \lambda_{w, {\bf t}} \widetilde{\partial}'_{k+1}(w\otimes a_{{\bf t}}). 
\]
Recalling that $\langle  w_0\otimes a_{{\bf t}_0}(\hat{k}), \widetilde{\partial}'_{k+1}(w\otimes a_{{\bf t}})\rangle=0 $ for any $w\otimes {\bf t}< w_0\otimes {\bf t}_0 $ with ${\bf t}, {\bf t}_0\in \mathcal{D}_{[k]}$ and $\langle  w_0\otimes a_{{\bf t}_0}(\hat{k}), \widetilde{\partial}'_{k+1}(w_0\otimes a_{{\bf t}_0})\rangle=1$, we have 
\[
 \langle w_0\otimes a_{{\bf t}_0}(\hat{k}), m\beta- \sum_{w\otimes {\bf t}> w_0\otimes {\bf t}_0 }\lambda_{w, {\bf t}} \widetilde{\partial}'_{k+1}(w\otimes a_{{\bf t}}) \rangle =\lambda_{w_0, {\bf t}_0}.
\]
By the choice of $w_0\otimes {\bf t}_0 $, we have $m| \lambda_{w, {\bf t}}$ for any $w\otimes {\bf t}> w_0\otimes {\bf t}_0$. Thus the left hand side is divisible by $m$, so is $\lambda_{w_0, {\bf t}_0}$ on the right hand side. This contradicts our assumption for $\lambda_{w_0, {\bf t}_0}$. Therefore, all coefficients $\lambda_{w, {\bf t}}$ in \eqnref{eq: mbeta} are divisible by $m$, and hence $\beta\in B_{k+1}$. This proves that $B_{k+1}= {\rm Ker}\, \widetilde{\partial}'_{k}$. Similarly, one can prove that $B_{k+1}= {\rm Im}\, \widetilde{\partial}'_{k+1}$. Thus the homology group $H_k={\rm Ker}\, \widetilde{\partial}'_k/{\rm Im}\, \widetilde{\partial}'_{k+1}$ is trivial. 
\end{proof}
	
	\begin{proposition}\label{prop: ayccomp}
		The cochain complexes $(\mathbb{Z}W\otimes \mathcal{A}, \ell_\sigma)$ and $(\mathbb{Z}W\otimes \mathcal{A}, r_\varsigma)$ are both acyclic. 
	\end{proposition}
	\begin{proof}
		By \lemref{lem: adjW} $\partial'$ is right adjoint to $\ell_{\sigma}$. Using the same method as in \lemref{lem: isocpl}, one can show that the isomorphism $\psi_W$ given in \eqnref{eq: psiW} induces an isomorphism between chain complexes $(\mathbb{Z}W\otimes \mathcal{A}, \partial')$ and   $((\mathbb{Z}W\otimes \mathcal{A})^*, \ell_\sigma^*)$. The former is acyclic by \propref{prop: aycpar}, so is the latter. Hence $(\mathbb{Z}W\otimes \mathcal{A}, \ell_\sigma)$  is acyclic by the universal coefficient theorem for cohomology. Similarly, one can prove the acyclicity of $(\mathbb{Z}W\otimes \mathcal{A}, r_\varsigma)$.
	\end{proof}
	
	\subsection{Dual complexes for Milnor fibres and hyperplane complements}
	In terms of the bilinear forms on $\mathcal{A}$ and $\mathbb{Z}W\otimes \mathcal{A}$, we shall give complexes which are dual to the chain complexes introduced in \secref{sec: comp}. These dual complexes have the same integral cohomology as that of the Milnor fibres and hyperplane complements.

	\subsubsection{Dual complexes for $M$ and $M/W$}
	Recall from \thmref{thm: homoM}  the chain complex which computes the integral homology of the hyperplane complement $M$. The following gives a cochain complex dual to the chain complex.

	\begin{theorem}\label{thm: dualM}
		The integral cohomology of the hyperplane complement $M$ is isomorphic to the cohomology of the following cochain  complex of free abelian groups:
		\[
		0\lr \mathbb{Z}W\ot\CA_0\overset{\partial^0}{\lr} \mathbb{Z}W\ot\CA_{1}\overset{\partial^1}{\lr}  \dots\lr \mathbb{Z}W\ot \CA_n\lr 0,
		\]
		where the coboundary maps are given by $\partial^k:=\ell_\sigma-(-1)^k r_\varsigma$ for $0\leq k\leq n$, i.e.
		\[
		\partial^{k}(w \otimes x) =\sum_{t\in T} wt \otimes a_{t}x- (-1)^k \sum_{t\in T}w\otimes  x a_t, \quad \forall x\in \mathcal{A}_k, w\in W.
		\] 
	\end{theorem}
	
	\begin{proof}
		Recall that the chain complex given in \thmref{thm: homoM} computes the integral homology of $M$. We dualise this chain complex by replacing the $k^{\rm th}$ chain group with $(\mathbb{Z}W \otimes \mathcal{A}_k)^*\cong{\rm Hom}_{\mathbb{Z}}( \mathbb{Z}W\otimes \mathcal{A}_k, \mathbb{Z})$ and the  $k^{\rm th}$ boundary map by its dual coboundary map $\partial^{*}: (\mathbb{Z}W \otimes \mathcal{A}_{k-1})^*\rightarrow  (\mathbb{Z}W \otimes \mathcal{A}_{k})^*$. Then we obtain the following cochain complex:
		\[
		\mathcal{C}(W)^*:  0\lr( \mathbb{Z}W\ot\CA_0)^*\overset{\partial^{*0}}{\lr} (\mathbb{Z}W\ot\CA_{1})^*\overset{\partial^{*1}}{\lr}  \dots\lr (\mathbb{Z}W\ot \CA_n)^*\lr 0.
		\]
		By the universal coefficient theorem for cohomology, we have $H^k(\mathcal{C}(W)^*)\cong H^k(M;\mathbb{Z})$ for $0\leq k\leq n$. It remains to show that $\mathcal{C}(W)^*$ and the complex given in the theorem have the same cohomology. 
		
		Recall the the bilinear form $\langle-,-\rangle$ on $\mathbb{Z}W\otimes \mathcal{A}$ is unimodular. This gives rises to the isomorphism:
		\begin{equation}\label{eq: psiWp}
			\psi'_W: \mathbb{Z}W\otimes \mathcal{A} \rightarrow (\mathbb{Z}W\otimes \mathcal{A})^*, \quad v\otimes x \mapsto  \psi'_W(v\otimes x):= \langle  v\otimes x,-\rangle
		\end{equation}
		for any $v\in W$ and $x\in \mathcal{A}$.  Now consider the following diagram:
		\begin{center}
			\begin{tikzcd}
				\mathbb{Z}W\otimes \mathcal{A}_k \arrow[r, "\partial^k"] \arrow[d, "\psi'_W"]
				& \mathbb{Z}W\otimes \mathcal{A}_{k+1} \arrow[d, "\psi'_W"] \\
				(\mathbb{Z}W\otimes \mathcal{A}_k)^*  \arrow[r,  "\partial^{*k}" ]
				&  (\mathbb{Z}W\otimes \mathcal{A}_{k+1})^*. \end{tikzcd}
		\end{center}
		For any $v, w\in W$, $x\in \mathcal{A}_k$ and $y\in \mathcal{A}_{k+1}$, we have 
		\[
		\begin{aligned}
			\psi'_W(\partial^k(v\otimes x))(w\otimes y) &= \langle \partial^k(v\otimes x), w\otimes y\rangle  = \langle \ell_{\sigma}(v\otimes x)-(-1)^{k} r_{\varsigma}(v\otimes x) , w\otimes y\rangle \\
			&=\langle v\otimes x, (\partial_k'+(-1)^{k+1} \partial_k'')(w\otimes y)\rangle,\\
			&= \partial^{*k}(\psi'_W(v\otimes x))(w\otimes y).
		\end{aligned}
		\]
		where the third equation follows from \lemref{lem: adjW} and the last equation follows from the fact that $\partial_k= \partial_k'+(-1)^{k+1} \partial_k''$. Therefore, $\psi'_W$ is a chain isomorphism between $(\mathbb{Z}W\otimes \mathcal{A}, \partial^k)$ and $\mathcal{C}(W)^*$, and hence these two cochain complexes have the same cohomology.
	\end{proof}

	\begin{theorem}\label{thm: dualMW}
		The integral cohomology of  $M/W$ or the Artin group $A(W)$ is isomorphic to the cohomology of the following cochain  complex of free abelian groups:
		\[
		0\lr \CA_0\overset{\partial^0}{\lr} \CA_{1}\overset{\partial^1}{\lr}  \dots \overset{}{\lr}  \CA_n\lr 0,
		\]
		where  $\partial^k=\ell_{\omega}-(-1)^k r_{\omega}$  for $0\leq k\leq n$ with $\omega=\sum_{t\in T}a_t$, i.e.
		\[
		\partial^{k}( x) =\omega x- (-1)^k  x \omega, \quad \forall x\in \mathcal{A}_k.
		\] 
	\end{theorem}
	\begin{proof}
		Recall from \thmref{thm:  homoMquo} the chain complex which computes the integral homology of $M/W$ or $A(W)$. The theorem can be proved using the same method as in the proof of \thmref{thm: dualM}.
	\end{proof}

	\subsubsection{Dual complexes for $F$ and $F/W$}
	Recall from \thmref{thm: fullcom} and \thmref{thm: disfibre} the chain complexes which realise the integral homology of the Milnor fibres $F$ and $F/W$, respectively. Using the bilinear form on $\mathbb{Z}W\otimes \mathcal{A}$, we will construct the cochain complexes which are dual to these chain complexes. 
	
	We begin with the following proposition, which identifies $d(\mathcal{A}_{k+1})$ with $\mathcal{A}_{k}\omega$. The latter is the $k^{{\rm th}}$ homogeneous component of the left ideal $\mathcal{A}\omega$ of $\mathcal{A}$ generated by $\omega$. 
	
	\begin{proposition}\label{prop: Adec}
		For each $k=0,\dots, n$, the free abelian group $\mathcal{A}_k$ decomposes as
		\[ \mathcal{A}_k = d(\mathcal{A}_{k+1}) \oplus \mathcal{A}_{k-1}\omega.  \]
		Moreover, we have an isomorphism $\mathcal{A}_{k-1}\omega \cong d(\mathcal{A}_k)$,  given by the linear map $d$. 
	\end{proposition}
	
	\begin{proof}
		First, we prove the isomorphism $\mathcal{A}_{k-1}\omega \cong d(\mathcal{A}_k)$.  
		Since the cochain complex $(\mathcal{A}, r_{\omega})$ is  acyclic,
		we have the following short exact sequences
		\[
		0\rightarrow \mathcal{A}_{k-1}\omega \rightarrow \mathcal{A}_{k} \rightarrow  \mathcal{A}_{k}\omega \rightarrow 0, \quad 0\leq k\leq n
		\]
		Then each short exact sequence splits since $\mathcal{A}_{k}\omega$ is free, being a subgroup of the free abelian group $\mathcal{A}_{k+1}$ .  Therefore, we  obtain that 
		\[ \mathcal{A}_k \cong \mathcal{A}_{k-1}\omega \oplus \mathcal{A}_k\omega, \quad   0\leq k\leq n.\]
		Similarly, using the acyclicity of $(\mathcal{A},d)$ we have 
		\[ 
		\mathcal{A}_k\cong d(\mathcal{A}_k)\oplus  d(\mathcal{A}_{k+1}), \quad  0\leq k\leq n.
		\] 
		Note that $\mathcal{A}_n\cong d(\mathcal{A}_n)\cong \mathcal{A}_{n-1}\omega$.  By comparing the isomoprhisms above  we have  $\mathcal{A}_{k-1}\omega \cong d(\mathcal{A}_k)$ for $0\leq k\leq n$. 
		
		Using the adjoint property in \lemref{lem: adj}, we have $\langle x\omega, d(y)\rangle= \langle x, d^2(y)\rangle=0$  for any $x\in \mathcal{A}_{k-1}$ and $y\in \mathcal{A}_{k+1}$. 
		Hence $\mathcal{A}_{k-1}\omega$ is orthogonal to $d(\mathcal{A}_{k+1})$ as subgroups of $\mathcal{A}_k$. Moreover, using the above isomorphisms we obtain
		\[{\rm rank}\, \mathcal{A}_k = {\rm rank}\, d(\mathcal{A}_k) + {\rm rank}\, d(\mathcal{A}_{k+1})={\rm rank}\, \mathcal{A}_{k-1}\omega + {\rm rank}\, d(\mathcal{A}_{k+1}). \]
		Therefore,	 we have the decomposition $\mathcal{A}_k=\mathcal{A}_{k-1}\omega \oplus  d(\mathcal{A}_{k+1})$ for all $k$.  
	\end{proof}

	\begin{theorem}\label{thm: dualF}
		The integral cohomology of the Milnor fibre $F$ is isomorphic to the cohomology of the following cochain complex of free abelian groups:
		\[
		0\lr \mathbb{Z}W\ot \CA_{0}\omega\overset{\ell_{\sigma}}{\lr} \mathbb{Z}W\ot \CA_1\omega\overset{\ell_{\sigma}}{\lr} \dots\lr \mathbb{Z}W\ot \CA_{n-1}\omega\lr 0,
		\]
		where $\omega=\sum_{t\in T}a_t$ and the coboundary maps are given by $\ell_{\sigma}$, i.e.
		\[
		\ell_{\sigma}(w \otimes x\omega)= \sum_{t\in T} wt \otimes a_{t}x\omega, \quad \forall x\in \mathcal{A}_k, w\in W.
		\] 
	\end{theorem}
	\begin{proof}
		For each $k=0,1,\dots, n-1$, let $(\mathbb{Z}W\otimes \mathcal{A}_k\omega)^*= {\rm Hom}_{\mathbb{Z}}(\mathbb{Z}W\otimes \mathcal{A}_k\omega, \mathbb{Z} )$. Consider the following chain complex $((\mathbb{Z}W\ot \CA\omega)^*, \ell^*_{\sigma}) $:
		\[
		0\lr (\mathbb{Z}W\ot \CA_{n-1}\omega)^* \overset{\ell_{\sigma}^*}{\lr} (\mathbb{Z}W\ot \CA_{n-2}\omega)^*\overset{\ell^*_{\sigma}}{\lr} \dots\lr (\mathbb{Z}W\ot \CA_{0}\omega)^*\lr 0.
		\]
		We will show that this chain complex is isomorphic to the chain complex $(\mathbb{Z}W \otimes d(\mathcal{A}_k), \partial_{k-1})$ given in \thmref{thm: fullcom}. Thus, these two chain complexes both compute the integral homology of $F$. By the universal coefficient theorem for cohomology (cf. \thmref{thm: uct}),  the cochain complex $(\mathbb{Z}W\ot \CA\omega, \ell_{\sigma})$ given in the theorem, which is the dual complex of $((\mathbb{Z}W\ot \CA\omega)^*, \ell^*_{\sigma}) $,   has the integral cohomology of $F$. 
		
		Now we only need  to prove that $(\mathbb{Z}W \otimes d(\mathcal{A}_k), \partial_{k-1})$ and $((\mathbb{Z}W\ot \CA\omega)^*, \ell^*_{\sigma}) $ are isomorphic. 
		We start by constructing an isomorphism between $\mathbb{Z}W \otimes d(\mathcal{A}_k)$ and $(\mathbb{Z}W\ot \CA_{k-1}\omega)^*$ for each $k=1, \dots ,n$. By \propref{prop: Adec}, we have 
		\[ \mathbb{Z}W\otimes \mathcal{A}_k\, \omega \cong  \mathbb{Z}W\otimes d(\mathcal{A}_{k+1}),\quad  w\otimes x \mapsto \partial''(w\otimes x)=u\otimes d(x) 
		\]
		for any $w\in W$ and $x\in \mathcal{A}_k\omega$. Moreover, the linear map $\psi_W$ defined in \eqref{eq: psiW} restricts to the following isomorphism:
		\[
		\mathbb{Z}W\otimes \mathcal{A}_k\, \omega  \cong (\mathbb{Z}W \otimes \mathcal{A}_k\, \omega)^*, \quad w\otimes x \mapsto \psi_W(w\otimes x)=\langle-, w\otimes x\rangle.
		\]
		Combining the above two isomorphisms, we obtain that
		\[ \xi_k:  \mathbb{Z}W\otimes d(\mathcal{A}_{k+1}) \rightarrow  (\mathbb{Z}W\otimes \mathcal{A}_k\, \omega))^*\]
		is an isomorphism given by 
		\[
		\xi_k(w\otimes d(x)) :=\psi_W(w\otimes x)=  \langle-, w\otimes x\rangle
		\]
		for any $w\in W$ and $x\in \mathcal{A}_k\, \omega$. This is well-defined, since if $d(x)=d(y)$ for some $x,y\in \mathcal{A}_k\, \omega$, then we have $w\otimes (x-y) \in {\rm Ker}(1\otimes d)={\rm Ker}\, \partial''$ and 
		\[
		\xi_k(w\otimes d(x-y) ) (v\otimes z\omega)= \langle v\otimes z\omega,  w\otimes (x-y) \rangle= \langle v\otimes z, \partial''(w\otimes (x-y))\rangle =0,
		\]
		for any $w,v\in W$ and $z\in \mathcal{A}_k$, where the last equation follows from \lemref{lem: adjW}. 
		
		We now prove that $\xi_k$ is an isomorphism between chain complexes. It suffices to show that  the following diagram commutes for each $k$:
		\begin{center}
			\begin{tikzcd}
				\mathbb{Z}W\otimes d(\mathcal{A}_{k+1}) \arrow[r, "\partial_k"] \arrow[d, "\xi_k"]
				& \mathbb{Z}W\otimes d(\mathcal{A}_{k}) \arrow[d, "\xi_{k-1}"] \\
				(\mathbb{Z}W\otimes \mathcal{A}_k \omega)^* \arrow[r,  "\ell_\sigma^*" ]
				& (\mathbb{Z}W\otimes \mathcal{A}_{k-1}\omega)^*. \end{tikzcd}
		\end{center}
		Recalling that in terms of notation \eqref{eq: defpartials}, we have $\partial_k=\partial'_k$ and $1\otimes d_k= \partial''_k$. 
		For any $v,w\in W$ and $y\in \mathcal{A}_{k-1}\, \omega$, we have
		\[
		\begin{aligned}
			\xi_{k-1}\partial_k( v \otimes d(x))(w\otimes y)&= \xi_{k-1}(\partial' \partial'' (v \otimes x))(w\otimes y) = \xi_{k-1}( \partial''\partial'  (v \otimes x))(w\otimes y) \\
			&= \langle w\otimes y, \partial'(v\otimes x)\rangle, 
		\end{aligned} 
		\]
		where the second equation follows from \lemref{lem: pars}. On the other hand, 
		\[
		\ell_\sigma^*\xi_k(v\otimes d(x)) (w\otimes y)=\xi_k(v\otimes d(x))(\ell_\sigma(w\otimes y))  =\langle \ell_\sigma(w\otimes y), v\otimes x\rangle.  
		\]
		By the adjoint property in \lemref{lem: adjW}, we have $\xi_{k-1}\partial_k= \ell_\sigma^*\xi_k$. Therefore, $\xi_k$ is an isomorphism between the chain complexes $(\mathbb{Z}W \otimes d(\mathcal{A}_k), \partial_k)$ and $((\mathbb{Z}W\ot \CA\omega)^*, \ell^*_{\sigma}) $. This completes the proof.
	\end{proof}

	\begin{corollary}
		Let $\mathscr{A}:= \mathbb{Z}W \otimes \mathcal{A}$ be the tensor product equipped with the usual multiplicative structure. Then we have the following $\mathbb{Z}$-graded isomorphism of  abelian groups:
		\[
		H^*(F;\mathbb{Z})[-1] = (\mathscr{A}\varsigma \cap \sigma \mathscr{A})/ \sigma \mathscr{A}\varsigma,  
		\]
		where  $\sigma=\sum_{t\in T} t\otimes a_{t}$ and $\varsigma=\sum_{t\in T}1\otimes a_{t}$, and $A[-1]_n:=A_{n-1}$ for the $\mathbb{Z}$-graded abelian group $A$.
	\end{corollary}
	\begin{proof}
     Recall from \propref{prop:  ayccomp} that the complex $(\mathbb{Z}W \otimes \mathcal{A}, \ell_{\sigma})$ is acyclic. Denote $\mathscr{A}_{k}:=\mathbb{Z}W \otimes \mathcal{A}_k$ for $0\leq k\leq n$. Then the coboundary map $\ell_{\sigma}: \mathscr{A}_{k+1} \rightarrow  \mathscr{A}_{k+2}$ has the kernel $\sigma \mathscr{A}_{k}$. Using \thmref{thm: dualF}, we obtain  
     	\[
		H^{k}(F; \mathbb{Z}) \cong {\rm Ker}(\mathscr{A}_{k}\varsigma \overset{\ell_{\sigma}}{\lr} \mathscr{A}_{k+1}\varsigma)/ \sigma \mathscr{A}_{k-1}\varsigma = (\sigma \mathscr{A}_{k}\cap \mathscr{A}_{k}\varsigma)/\sigma \mathscr{A}_{k-1}\varsigma. 
		\]
		This completes the proof.
	\end{proof}
	
	\begin{theorem}\label{thm: dualFW}
		The integral cohomology of the Milnor fibre $F/W$ is isomorphic to the cohomology of the following cochain complex of free abelian groups:
		\[
		0\lr \CA_{0}\omega\overset{\ell_{\omega}}{\lr} \CA_1\omega\overset{\ell_{\omega}}{\lr} \dots \overset{\ell_{\omega}}{\lr} \CA_{n-1}\omega\lr 0,
		\]
		where the coboundary maps are given by left multiplication by $\omega=\sum_{t\in T}a_t$. 
	\end{theorem}
	\begin{proof}
		The proof is similar to that of \thmref{thm: dualF}. Recall from \thmref{thm:  disfibre} the chain complex which computes the integral homology of $F/W$. Using the bilinear form \eqnref{eq: psi}, one can identify $d(\mathcal{A}_{k+1})$ and $(\mathcal{A}_k\omega)^*$ and prove that the complexes $(d(\mathcal{A}_{k+1}), \partial_k)$ and $((\mathcal{A}\omega)^*, \ell_{\omega}^*)$ are isomorphic. Thus the complex $(\mathcal{A}\omega, \ell_{\omega})$ computes the integral cohomology of $F/W$.
	\end{proof}

	\begin{corollary}\label{coro: cohomFW}
		Let $\mathcal{A}\omega$ (resp. $\omega\mathcal{A}$) be the left (resp. right) ideal of $\mathcal{A}$ generated by $\omega$. Then we have the following $\mathbb{Z}$-graded isomorphism of  abelian groups:
		\[H^*(F/W; \mathbb{Z})[-1]\cong (\mathcal{A}\omega \cap \omega \mathcal{A})/ \omega \mathcal{A}\omega, \]
		where $A[-1]_n:=A_{n-1}$ for the $\mathbb{Z}$-graded abelian group $A$. 
	\end{corollary}
	\begin{proof}
		By \propref{prop: acyccocx}, the cochain complex $(\mathcal{A}, \ell_{\omega})$ is acyclic. It follows that the coboundary map $\ell_{\omega}: \mathcal{A}_{k+1} \rightarrow \mathcal{A}_{k+2}$ has the kernel $\omega\mathcal{A}_{k}$. Using \thmref{thm: dualFW}, we have 
		\[
		H^{k}(F/W; \mathbb{Z}) \cong {\rm Ker}(\mathcal{A}_{k}\omega \overset{\ell_{\omega}}{\lr} \mathcal{A}_{k+1}\omega )/ \omega \mathcal{A}_{k-1}\omega = (\omega\mathcal{A}_{k}\cap \mathcal{A}_{k}\omega)/\omega \mathcal{A}_{k-1}\omega. 
		\]
		This completes the proof.
	\end{proof}
	
	\subsection{A pair of dual complexes with complex coefficients}
	We shall introduce a pair of cochain complexes which are dual to the complexes $\mathcal{C}(U)$ and $\mathcal{K}(U)$. In this subsection, we work over $\mathbb{C}$.
	
	Let $U$ be any finite dimensional right $\mathbb{C}W$-module. We define the dual complex of  $\mathcal{C}(U)$ by  
	\[
	\CC^*(U):= 0\lr U\ot\CA_0\overset{\partial^*}{\lr}\dots\lr U\ot\CA_{n-1}\overset{\partial^*}{\lr}U\ot \CA_n\lr 0,
    \]
	where the coboundary maps are given by 
	\begin{equation*}\label{eq: cobdycu}
		\partial^{\ast}(u \otimes x) =\sum_{t\in T} ut \otimes a_{t}x- (-1)^ku\otimes  x\omega, \quad \forall x\in \mathcal{A}_k, u\in U.
	\end{equation*}
	It is straightforward to verify that $(\partial^*)^2=0$. Similarly, define  the dual complex of $\mathcal{K}(U)$ by 
	\[
	\CK^*(U):= 0\lr U\ot\CA_0\,\omega \overset{\partial^*_{0}}{\lr}\dots\lr U\ot\CA_{n-2}\,\omega \overset{\partial^*_{n-2}}{\lr}U\ot \CA_{n-1}\,\omega \lr 0,
	\]
	where $\mathcal{A}_k\,\omega$ is the subspace of $\mathcal{A}_{k+1}$ linearly spanned by elements of the form $a_{t_1}\dots a_{t_k}\omega$, and the coboundary maps are defined by
	\begin{equation*}\label{eq: cobdyku}
		\partial^{\ast}(u \otimes x\omega)= \sum_{t\in T} ut \otimes a_{t}x\omega, \quad \forall x\in \mathcal{A}_k, u\in U.
	\end{equation*}

	The following is a cohomology version of \thmref{thm: mult}.
	
	\begin{theorem}\label{thm: multcohom}
		For any right $W$-module $U$ and for each integer $k\geq 0$, we have:
		\[
		\dim H^k(\CC^{\ast}(U))= 	 \langle U_L^*, H^k(M)\rangle
	    \]
		and 
		\[
		\dim H^k(\CK^*(U))= \langle U_L^*, H^k(F)\rangle.
		\]
	\end{theorem}
	\begin{proof}
		Use \thmref{thm: dualM} and \thmref{thm: dualF}.  The proof is similar to that of \thmref{thm: mult}. 
	\end{proof}

\section{Complements on a covering algebra $\wt\CA$ of $\CA$.}


In this section we define an algebra $\widetilde{\mathcal{A}}$, whose presentation is simpler than that of $\CA$, and which in fact has $\CA$ as a homomorphic image. The algebra  $\widetilde{\mathcal{A}}$ has some remarkable similarities with the Fomin-Kirillov algebra $\CE_n$ (in type $A_{n-1}$) (cf. \cite{FK99}, and we conjecture that the two algebras have the same Hilbert-Poincar\'e series (always in type $A_n$). They are in some sense ``dual'' to each other, but are not Koszul.

In \secref{sec: newalg} we  define a braided Hopf algebra and show that there exists a  surjective algebra homomorphism from the braided Hopf algebra to the noncrossing algebra. We show also that $\widetilde{\mathcal{A}}$ has the structure of a $W$-graded Hopf algebra, that it has a braiding, and more generally belongs to a category of Yetter-Drinfeld modules over $\C W$. In \secref{sec: diff} we define some differential operators on the  braided Hopf algebra, and determine some of their adjoint properties, which are similar to those of $\CA$. 
We also  prove that there is a bilinear form on $\widetilde{\mathcal{A}}$, and that this form descends to the one we already have on $\CA$. Throughout this section, we work over the complex field $\mathbb{C}$.


\subsection{A new braided Hopf algebra} \label{sec: newalg}
We introduce a cover of the noncrossing algebra, which is analogous to the Fomin-Kirillov algebra \cite{FK99}. This new algebra is a Yetter-Drinfeld module over the group algebra $\mathbb{C}W$, and has a braided Hopf algebra structure. We refer to \cite{AS02} for background concerning Yetter-Drinfeld modules. 
\subsubsection{Definition and Examples}
\begin{definition}\label{def: newalg}
	Let $W$ be any finite Coxeter group and $T$ be the set of reflections of $W$. Define  $\widetilde{\mathcal{A}}=\widetilde{\mathcal{A}}(W)$ to be the associative algebra over $\mathbb{C}$ generated by $\alpha_t, t\in T$, subject to the following quadratic relations:
	\begin{align}
		\alpha_{t}^2=0, \quad & \text{for any $t\in T$ }, \label{eq: newquad1}\\
		\sum_{(t_1,t_2)\in {\rm Rex}_T(w)} \alpha_{t_1}\alpha_{t_2}=0, \quad & \text{for any  $ w\in W$ with $\ell_{T}(w)=2$}    \label{eq: newquad2}.
	\end{align}
\end{definition}


The algebra $\widetilde{\mathcal{A}}$ is a $\mathbb{Z}$-graded algebra with the $\mathbb{Z}$-grading ${\rm deg}(\alpha_t)=1$ for all $t\in T$. In this section, we denote by $\mathcal{A}$ the noncrossing algebra over the complex field $\mathbb{C}$. Theses two algebras are related by the following lemma. 

\begin{lemma}\label{lem: surj}
	We have  a surjective algebra homomorphism $\pi: \widetilde{\mathcal{A}} \rightarrow \mathcal{A}$,  given by $\alpha_t \mapsto a_{t}$ for all $t\in T$. 
\end{lemma}
\begin{proof}
	We just need to check that the relations  \eqref{eq: newquad2} are preserved in $\mathcal{A}$. If $w\leq \gamma$, then  relation \eqref{eq: newquad2} is sent to the defining relation of $\mathcal{A}$ under the map $\pi$. Otherwise, for any two reflections $t_1, t_2$ such that $t_1t_2\not\leq \gamma$, we have $\pi(\alpha_{t_1} \alpha_{t_2})= a_{t_1}a_{t_2}=0$,  and hence for any  $w\not \leq \gamma$ we have 
	$\sum_{(t_1,t_2)\in {\rm Rex}_T(w)} \pi(\alpha_{t_1}\alpha_{t_2})=0$.
\end{proof}


\begin{example}
	The new algebra $\widetilde{\mathcal{A}}({\rm Sym}_{n})$ of type $A_{n-1}$ is generated by $\alpha_{ij}= \alpha_{ji}$ for $1\leq i<j\leq n$ with the following relations: 
	\begin{equation}\label{eq: typeA}
		\begin{aligned}
			&\alpha_{ij}^2=0, \\
			\alpha_{ij}\alpha_{kl}&+\alpha_{kl}\alpha_{ij}=0, \quad \text{for distinct $i,j,k,l$}\\
			\alpha_{ij}\alpha_{jk}+ \alpha_{jk}&\alpha_{ki}+\alpha_{ki}\alpha_{ij}=0, \quad \text{for distinct $i,j,k$}
		\end{aligned}	
	\end{equation}
	In type $A_2$, these  relations read:
	\[ 
	\begin{aligned}
		&\alpha_{12}^2=\alpha_{13}^2=\alpha_{23}^2=0,\\
		\alpha_{12}\alpha_{23}&+ \alpha_{23}\alpha_{13}+ \alpha_{13}\alpha_{12}=0,\\
		\alpha_{23} \alpha_{12}&+ \alpha_{12}\alpha_{13}+ \alpha_{13} \alpha_{23}=0.
	\end{aligned}
	\]
	This algebra looks similar to the Fomin-Kirillov algebra $\mathcal{E}_n$ \cite{FK99}, which is generated by $x_{ij}=-x_{ji}$ for $1\leq i<j\leq n$ with relations:
	\[
	\begin{aligned}
		&x_{ij}^2=0, \\
		x_{ij}x_{kl}&-x_{kl}x_{ij}=0, \quad \text{for distinct $i,j,k,l$}\\
		x_{ij}x_{jk}+ x_{jk}&x_{ki}+x_{ki}x_{ij}=0, \quad \text{for distinct $i,j,k$}.
	\end{aligned}	
	\]
	
	For any $\mathbb{Z}$-graded algebra $A=\bigoplus_{k\in \mathbb{Z}}A_k$, we denote by $H_{A}(t)= \sum_{k\in \mathbb{Z}} {\rm dim}A_k \,t^k$ the Hilbert-Poincar\'e series of $A$. Using computational software, we obtain that   $\widetilde{\mathcal{A}}({\rm Sym}_{n})$ and $\mathcal{E}_n$ have the same Hilbert-Poincar\'e series  for $n\leq 5$:  
	\[
	\begin{aligned}
		&n=1: H_{\widetilde{\mathcal{A}}({\rm Sym}_{n})}(t)=H_{\mathcal{E}_n}(t)= 1,\\
		&n=2: H_{\widetilde{\mathcal{A}}({\rm Sym}_{n})}(t)=H_{\mathcal{E}_n}(t)=[2]=1+t,\\
		&n=3: H_{\widetilde{\mathcal{A}}({\rm Sym}_{n})}(t)=H_{\mathcal{E}_n}(t)=[2]^2[3]=1+3t+4t^2+3t^3+t^4,\\
		&n=4: H_{\widetilde{\mathcal{A}}({\rm Sym}_{n})}(t)=H_{\mathcal{E}_n}(t)=[2]^2[3]^2[4]^2,\\
		&n=5: H_{\widetilde{\mathcal{A}}({\rm Sym}_{n})}(t)=H_{\mathcal{E}_n}(t)=[4]^4[5]^2[6]^4,
	\end{aligned}
	\]
	where we have used the notation $[k]:=1+t+\dots +t^{k-1}$.  These Hilbert-Poincar\'e series  have symmetric coefficients. In particular, the top homogeneous component has dimension 1.  It is unknown whether $\mathcal{E}_n$ is finite-dimensional for $n\geq 6$. 
\end{example}

\begin{conjecture}
	The algebra $\widetilde{\mathcal{A}}({\rm Sym}_n)$ and the Fomin-Kirillov algebra $\mathcal{E}_n$  have the same Hilbert-Poincar\'e series.  
\end{conjecture}

\begin{remark}\label{rmk: quo}
	Recall that the Orlik-Solomon algebra associated to the reflection arrangement of ${\rm Sym}_n$ is generated by elements $e_{ij}=e_{ji}$ for $1\leq i<j\leq n$, subject to the following relations:
	\[
	\begin{aligned}
		e_{ij}e_{kl}&= -e_{kl}e_{ij}, &\quad& 1\leq i<j\leq n, 1\leq k\leq l\leq n,\\
		e_{ij}e_{jk}+e_{jk}e_{ki}&+ e_{ki}e_{ij}=0, &\quad& 1\leq i,j,k\leq n.
	\end{aligned}
	\]
	This appears as a quotient of $\widetilde{\mathcal{A}}({\rm Sym}_n)$ by imposing the anti-commutative relations $\alpha_{ij} \alpha_{kl}= -\alpha_{kl} \alpha_{ij}$ for $i<j$ and $k<l$, that is, we allow $\{i,j\} \cap \{k,l\} \neq \emptyset$ in the second relation of \eqref{eq: typeA}.
\end{remark}

\subsubsection{Braided Hopf algebra structure}
We now take a Hopf-theoretic point of view to the covering algebra $\widetilde{\mathcal{A}}$ \cite{AG99,MS00,AS02}. 

Let us recall relevant definitions. The group algebra $\mathbb{C}W$ has a  Hopf algebra structure with the comultiplication $\Delta(w)= w\otimes w$, counit $\epsilon(w)=1$ and 
antipode $S(w)=w^{-1}$ for any $w\in W$. A \emph{Yetter-Drinfeld module} $A$ over  $\mathbb{C}W$ is a $W$-graded vector space $A=\bigoplus_{w\in W} A_w$, which is a $W$-module such that $w. A_{u}\subseteq A_{wuw^{-1}} $ for all $u,w\in W$.

The algebra $\widetilde{\mathcal{A}}$ is a Yetter-Drinfeld module over $\mathbb{C}W$, as we now describe. In addition to the natural $\mathbb{Z}$-grading, $\widetilde{\mathcal{A}}$ has a grading with respect to $W$ such that the $W$-degree of the generator $\alpha_{t}$ is $t\in T$ and this is extended to  all monomials by multiplication. As the defining relations of $\widetilde{\mathcal{A}}$ are homogeneous with respect to the $W$-degree, this gives a  $W$-grading of $\widetilde{\mathcal{A}}= \bigoplus_{w\in W} \widetilde{\mathcal{A}}_w$, where $\widetilde{\mathcal{A}}_w$ is spanned by monomials $\alpha_{t_1}\dots \alpha_{t_k}$ such that $t_1t_2\dots t_k=w$. Note that $w=t_1t_2\dots t_k$ is not necessarily a reduced expression with respect to reflections of $T$ or simple reflections of $S$. 

The  $W$-module structure  on $\widetilde{\mathcal{A}}$ is defined by 
\begin{equation}\label{eq: Wact}
	w.\alpha_{t}: =(-1)^{\ell(w)}  \alpha_{wtw^{-1}}, \quad , \forall w\in W,
\end{equation}
where $\ell(w)$ is the usual length of $w$ with respect to the generating set $S$ of $W$. Clearly, this action preserves the defining relations \eqref{eq: newquad1} and \eqref{eq: newquad2} of $\widetilde{\mathcal{A}}$, and is compatible with the $W$-grading, i.e. $w. \widetilde{\mathcal{A}}_{u}\subseteq \widetilde{\mathcal{A}}_{wuw^{-1}} $ for all $u,w\in W$. Therefore, $\widetilde{\mathcal{A}}$ is a Yetter-Drinfeld module over $\mathbb{C}W$.

We denote by $^{W}_{W}\mathcal{YD}$ the category of Yetter-Drinfeld modules over $\mathbb{C}W$. A morphism  $f: A\rightarrow B$ of the category $^{W}_{W}\mathcal{YD}$ is a homomorphism of $W$-modules which preserves the $W$-grading. 

An important ingredient of the Yetter-Drinfeld category is the canonical braiding. For any $A, B\in\,  ^{W}_{W}\mathcal{YD}$, the  canonical braiding  $c: A\otimes B \rightarrow B \otimes A$ is defined by  
\begin{equation}\label{eq: braiding}
	c(a \otimes b)= b\otimes (w^{-1}. a), \quad \forall a \in A, b\in B_w.
\end{equation}
The tensor product $A\otimes B$ is an object of $^{W}_{W}\mathcal{YD}$, with the $W$-grading $(A\otimes B)_w= \bigoplus_{ab=w} A_a\otimes B_b$ and the $W$-action $w. (a\otimes b)= w.a\otimes w.b$ for any $w\in W$, $a\in A$ and $b\in B$. 
In particular, we have  $\widetilde{\mathcal{A}}\otimes \widetilde{\mathcal{A}} \in {}^{W}_{W}\mathcal{YD}$. Moreover, the tensor product $\widetilde{\mathcal{A}}\otimes \widetilde{\mathcal{A}}$ is still an algebra, with multiplication  defined via the canonical braiding:
\begin{equation}\label{eq: mult}
	(x_1 \otimes y_1)(x_2\otimes y_2)=  x_1x_2 \otimes (w^{-1}.y_1)y_2, \quad \forall  x_2\in \widetilde{\mathcal{A}}_w, \, x_1,y_1, y_2\in \widetilde{\mathcal{A}}.
\end{equation}
More concisely, $\mu_{ \widetilde{\mathcal{A}}\otimes \widetilde{\mathcal{A}}}= (\mu_{\widetilde{\mathcal{A}}} \otimes \mu_{\widetilde{\mathcal{A}}})(1\otimes c \otimes 1)$, where $\mu_A: A\otimes A \rightarrow A $ denotes the multiplication map of the algebra $A$.

Recall that a \emph{braided bialgebra} $A$ in $^{W}_{W}\mathcal{YD}$ is a collection $(A, \mu, \eta, \Delta, \epsilon)$ such that $(A, \mu, \eta)$ is an algebra in $^{W}_{W}\mathcal{YD}$, $(A, \Delta, \epsilon)$ is a coalgebra in $^{W}_{W}\mathcal{YD}$ and $\Delta: A\rightarrow A\otimes A$ 
and $\epsilon: A\rightarrow \mathbb{C}$ are morphisms of algebras (here $A\otimes A$ 
is an algebra in $^{W}_{W}\mathcal{YD}$ with multiplication defined via the  braiding $c$). 
We call $A$  a \emph{braided Hopf algebra} if in addition there is an antipode $S: A\rightarrow A$ in $^{W}_{W}\mathcal{YD}$ such that $(1\otimes S)\Delta=(S\otimes 1)\Delta=\eta \epsilon$.

\begin{proposition}\label{prop: Hopfstr}
	The algebra $\widetilde{\mathcal{A}}$ is a braided Hopf algebra in $^{W}_{W}\mathcal{YD}$ with the coproduct $\Delta$, the counit $\epsilon$ and the antipode $S$ defined on the generators $\alpha_t, t\in T$ by 
	\begin{equation}\label{eq: coalg}
		\begin{aligned}
			&\Delta(\alpha_{t})= \alpha_t\otimes 1 + 1\otimes \alpha_t,\\
			&\epsilon(\alpha_{t})=0, \quad S(\alpha_{t})=-\alpha_t.
		\end{aligned}
	\end{equation}
	
\end{proposition}

\begin{proof}
	We need to check that $\Delta$, $\epsilon$ and $S$ are well-defined, and then check that they satisfy Hopf algebra axioms. Straightforward calculations show that:
	\[
	\begin{aligned}
		&\Delta(\alpha_t^2)= \alpha_t^2\otimes 1+ 1\otimes \alpha_{t}^2, \quad 
		S(\alpha_t^2)= \alpha_t^2,\quad 
		\epsilon(\alpha_{t}^2)=0,\\
		&\Delta(R_w)= R_w\otimes 1+ 1\otimes R_w,\\
		&S(R_w)= R_w,\quad 
		\epsilon(R_w)=0,
	\end{aligned}
	\]
	where we have used the notation $R_w:= \sum_{(t_1,t_2)\in {\rm Rex}_T(w)} \alpha_{t_1}\alpha_{t_2}$ for any  $w\in W$ with $\ell_{T}(w)=2$. Therefore, $\Delta$, $\epsilon$ and $S$ are all well-defined.
	
	Next we check the Hopf algebra axioms on the generators of $\widetilde{\mathcal{A}}$: (1) Coassociativity: 
	\[
	\begin{aligned}
		(\Delta\otimes 1)(\Delta(\alpha_t))&=(\Delta\otimes 1)(\alpha_t\otimes 1+ 1\otimes \alpha_t )\\
		&= \alpha_t\otimes 1\otimes 1+  1\otimes  \alpha_t \otimes 1+ 1\otimes 1 \otimes \alpha_t \\
		&= (1\otimes \Delta)(\Delta(\alpha_t)).  
	\end{aligned}
	\]
	(2) The counit axiom:
	\[
	\begin{aligned}
		(\epsilon \otimes 1)(\Delta (\alpha_t))= (\epsilon \otimes 1)(\alpha_t\otimes 1+ 1\otimes \alpha_t )= \alpha_t,\\
		(1 \otimes \epsilon)(\Delta (\alpha_t))= (1 \otimes \epsilon)(\alpha_t\otimes 1+ 1\otimes \alpha_t )= \alpha_t.
	\end{aligned}
	\]
	(3) The antipode axiom:
	\[
	\begin{aligned}
		\mu(1\otimes S)(\Delta(\alpha_t))&=\mu(1\otimes S)(\alpha_t\otimes 1+ 1\otimes \alpha_t)
		=\alpha_t- \alpha_t=0= \epsilon(\alpha_t),\\
		\mu(S\otimes 1)(\Delta(\alpha_t))&=\mu(S\otimes 1)(\alpha_t\otimes 1+ 1\otimes \alpha_t)
		=-\alpha_t+ \alpha_t=0= \epsilon(\alpha_t).
	\end{aligned}
	\]
\end{proof}

We can extend \eqref{eq: coalg} to any monomials of $\widetilde{\mathcal{A}}$. Clearly, $\epsilon(\alpha_{t_1}\dots \alpha_{t_k})=0$ for $k\geq 1$. In the following we give explicit formulae for $S$ and $\Delta$.

\begin{proposition}\label{lem: antipode}
	The antipode $S$ is given  explicitly by 
	\begin{equation}\label{eq: antipode}
		S(\alpha_{t_1}\dots \alpha_{t_k})= \varepsilon(t_1, \dots ,t_k)\, \alpha_{t_k} \alpha_{t_{k-1}^{t_k}} \dots \alpha_{t_1^{t_{2}\dots t_k}}  ,
	\end{equation}
	where $\varepsilon(t_1, \dots ,t_k)= (-1)^k\prod_{i=2}^k (-1)^{\ell(t_i\dots t_{k})}$.
\end{proposition}
\begin{proof}
	Recall that in $^{W}_{W}\mathcal{YD}$ we have $S\mu= \mu(S\otimes S)c$. This follows from the fact that both $S\mu$ and $\mu(S\otimes S)c$ are the inverse of $\mu$ under the convolution product in the algebra 
	${\rm Hom}(\widetilde{\mathcal{A}} \otimes \widetilde{\mathcal{A}}, \widetilde{\mathcal{A}})$;  refer to \cite[Lemma 1.2.2]{AG99}. Using this equation and induction on $k$, we have 
	\[ 
	\begin{aligned}
		S(\alpha_{t_1} \alpha_{t_{2}}\dots  \alpha_{t_k})
		&= S(\alpha_{t_2}\dots \alpha_{t_{k}}) S((t_k\dots t_{2}).\alpha_{t_1})\\
		&=- (-1)^{\ell(t_2\dots t_{k})}  S(\alpha_{t_2}\dots \alpha_{t_{k}}) \alpha_{t_1^{t_{2}\dots t_k}}.
	\end{aligned}
	\]
	The formula follows by induction hypothesis.
\end{proof}

\begin{proposition}\label{prop: comult}
	The comultiplication of $\widetilde{\mathcal{A}}$ is given explicitly by 
	\[ \Delta(\alpha_{t_1} \dots \alpha_{t_k} )=\sum_{j=0}^k
	\sum_{1\leq i_1<i_2<\dots <i_j\leq k}  \alpha_{t_{i_1}}\dots \alpha_{t_{i_j}}\otimes E_{t_{i_j}}\dots E_{t_{i_1}}(\alpha_{t_1} \dots \alpha_{t_k}),\]
	where 
	$E_{t_{r_i}}(\alpha_{t_{r_1}} \dots \alpha_{t_{r_s}}):= t_{r_i}.(\alpha_{t_{r_1}} \dots \alpha_{t_{r_{i-1}}} )\alpha_{t_{r_{i+1}}} \dots \alpha_{t_{r_s}}$,
	for any monomial $\alpha_{t_{r_1}} \dots \alpha_{t_{r_s}}\in \widetilde{\mathcal{A}}$ and $1\leq i\leq s$.
\end{proposition} 
\begin{proof}
	Use induction on $k$. The formula is trivial if $k=1$. For $k>1$, by induction hypothesis we have 
	\[\begin{aligned}
		&\Delta(\alpha_{t_1} \dots \alpha_{t_k})= \Delta(\alpha_{t_1} \dots \alpha_{t_{k-1}})\Delta(\alpha_{t_k})\\
		=&\sum_{j=0}^{k-1}
		\sum_{1\leq i_1<i_2<\dots <i_j\leq k-1}  (\alpha_{t_{i_1}}\dots \alpha_{t_{i_j}}\otimes E_{t_{i_j}}\dots E_{t_{i_1}}(\alpha_{t_1} \dots \alpha_{t_{k-1}}))( 1\otimes \alpha_{t_k}+ \alpha_{t_k}\otimes 1)\\ 
		=& \sum_{j=0}^{k-1}
		\sum_{1\leq i_1<i_2<\dots <i_j\leq k-1}  \alpha_{t_{i_1}}\dots \alpha_{t_{i_j}}\otimes (E_{t_{i_j}}\dots E_{t_{i_1}}(\alpha_{t_1} \dots \alpha_{t_{k-1}})) \alpha_{t_k}\\
		&+ \sum_{j=0}^{k-1}
		\sum_{1\leq i_1<i_2<\dots <i_j\leq k-1}  \alpha_{t_{i_1}}\dots \alpha_{t_{i_j}}\alpha_{t_k}\otimes t_k.(E_{t_{i_j}}\dots E_{t_{i_1}}(\alpha_{t_1} \dots \alpha_{t_{k-1}})). 
	\end{aligned}
	\]
	Note that in the above equation, we have
	\[
	\begin{aligned}
		(E_{t_{i_j}}\dots E_{t_{i_1}}(\alpha_{t_1} \dots \alpha_{t_{k-1}})) \alpha_{t_k}&=E_{t_{i_j}}\dots E_{t_{i_1}}(\alpha_{t_1} \dots \alpha_{t_{k-1}}\alpha_{t_k}) \\
		t_k.(E_{t_{i_j}}\dots E_{t_{i_1}}(\alpha_{t_1} \dots \alpha_{t_{k-1}}))&= E_{t_{k}}E_{t_{i_j}}\dots E_{t_{i_1}}(\alpha_{t_1} \dots \alpha_{t_{k}}).
	\end{aligned}   
	\]
	Then we obtain the formula of $\Delta(\alpha_{t_1} \dots \alpha_{t_k})$ as desired.
\end{proof}

\begin{example}
	Using \propref{prop: comult},  we have 
	\[ 
	\begin{aligned}
		\Delta(\alpha_{t_1} \alpha_{t_2})&= 1\otimes \alpha_{t_1}\alpha_{t_2}+ \alpha_{t_1}\otimes E_{t_1}(\alpha_{t_1}\alpha_{t_2}) + \alpha_{t_2}\otimes E_{t_2}(\alpha_{t_1}\alpha_{t_2}) + \alpha_{t_1}\alpha_{t_2} \otimes E_{t_{2}}E_{t_1}(\alpha_{t_1}\alpha_{t_2})\\
		&= 1\otimes \alpha_{t_1}\alpha_{t_2}+ \alpha_{t_1}\otimes \alpha_{t_2} - \alpha_{t_2}\otimes \alpha_{t_2t_1t_2} + \alpha_{t_1}\alpha_{t_2} \otimes 1.
	\end{aligned}
	\]
	It follows that $\widetilde{\mathcal{A}}$ is not  cocommutative. 
\end{example}

\begin{proposition}
	The noncrossing algebra $\mathcal{A}$  is a subcoalgebra of $\widetilde{\mathcal{A}}$.  
\end{proposition}
\begin{proof}
	By \lemref{lem: surj} $\mathcal{A}$ can be lifted as a subspace of $\widetilde{\mathcal{A}}$.  Note that if $a_{t_1}\dots a_{t_k}$ is a nonzero element of $\mathcal{A}$, that is, $w=t_1t_2\dots t_k\in \mathcal{L}$ is a $T$-reduced expression, then  $E_{t_i}(a_{t_1}\dots a_{t_k})= (-1)^{i-1} a_{t_1^{t_i}} \dots. a_{t_{i-1}^{t_i}} a_{t_{i+1} } \dots a_{t_k}$ is still a nonzero element of $\mathcal{A}$. 
	In view of \propref{prop: comult}, $\mathcal{A}$ is closed under  the  comultiplication $\Delta$ of $\widetilde{\mathcal{A}}$. In addition, $\mathcal{A}$ is clearly closed under the  counit $\epsilon$ of $\widetilde{\mathcal{A}}$. Therefore,  $\mathcal{A}$  is a subcoalgebra of $\widetilde{\mathcal{A}}$.
\end{proof}

\subsection{Skew-derivations on $\widetilde{\mathcal{A}}$}\label{sec: diff}
Recall that the noncrossing algebra $\mathcal{A}$ has skew-derivations $\delta_t$ and $d_t$ for any $t\in T$. We shall show that these skew-derivations can be lifted to the algebra $\widetilde{\mathcal{A}}$ with similar properties. The difference is that these  skew-derivations are defined using the braided Hopf algebra structure of $\widetilde{\mathcal{A}}$.

For any integer $k\geq 0$ let $\pi_k: \widetilde{\mathcal{A}}\rightarrow \widetilde{\mathcal{A}}_k$ be the projection of $\widetilde{\mathcal{A}}$ onto its $k$-th homogeneous component $\widetilde{\mathcal{A}}_k$. We denote by 
\[ \Delta_{i,j}: \widetilde{\mathcal{A}}_{i+j} \overset{\Delta}{\longrightarrow} \widetilde{\mathcal{A}} \otimes \widetilde{\mathcal{A}} \overset{\pi_i\otimes \pi_j}{\longrightarrow}  \widetilde{\mathcal{A}}_i \otimes \widetilde{\mathcal{A}}_j   \]
the $(i,j)$-th component of the comultiplication $\Delta$. 

For any $t\in T$, we define the linear map $\nabla_t: \widetilde{\mathcal{A}} \rightarrow \widetilde{\mathcal{A}}$ as follows: let $\nabla_{t}(1)=0$, and for any $x\in \widetilde{\mathcal{A}}_k$ define $\nabla_t(x) \in \widetilde{\mathcal{A}}_{k-1}$ by 
\begin{equation}\label{def: nabla}
	\Delta_{1,n-1}(x)= \sum_{t\in T} \alpha_t\otimes \nabla_t(x).
\end{equation}
Similarly, we define $D_t: \widetilde{\mathcal{A}} \rightarrow \widetilde{\mathcal{A}}$ by $D_t(1)=0$ and 
$\Delta_{n-1,1}(x)= \sum_{t\in T}  D_t(x)\otimes \alpha_t$. It is clear that $\nabla_{t_1}(\alpha_{t_2})= D_{t_1}(\alpha_{t_2})= \delta_{t_1,t_2}$ (Kronecker delta). 


\begin{proposition}\label{prop: D}
	For any $t\in T$, let $\nabla_t, D_t$ be as above. 
	\begin{enumerate}
		\item  For any $x,y\in \widetilde{\mathcal{A}}$, we have
		\[
		\begin{aligned}
			\nabla_t(xy)&= \nabla_t(x)y+ (t.x) \nabla_t(y),\\
			D_t(xy)&= (|y|.D_t)(x)y+ x D_t(y),
		\end{aligned}
		\] 
		where $|y|\in W$ denotes the $W$-grading of $y$, i.e. $y\in \widetilde{\mathcal{A}}_{|y|}$, and  $|y|.D_t:=(-1)^{\ell(|y|)} D_{|y|t|y|^{-1} }$. 
		
		\item For any $\alpha_{t_1}\alpha_{t_2}\dots \alpha_{t_k}\in \widetilde{\mathcal{A}}$, we have 
		\[
		\begin{aligned}
			\nabla_{t}(\alpha_{t_1}\alpha_{t_2}\dots \alpha_{t_k})&= \sum_{i=1}^k (-1)^{i-1} \delta_{t,t_i} \alpha_{t_1^{t_i}} \dots \alpha_{t_{i-1}^{t_i}} \alpha_{t_{i+1} } \dots \alpha_{t_k},\\
			D_{t}(\alpha_{t_1}\alpha_{t_2}\dots \alpha_{t_k})&= \sum_{i=1}^k (-1)^{k-i} 
			\delta_{t,r_i} 
			\alpha_{t_1} \dots. \alpha_{t_{i-1}} \alpha_{t_{i+1} } \dots \alpha_{t_k},
		\end{aligned}    
		\] 
		where $r_i= t_i^{ t_{i+1}\dots t_{k} }$, and $\delta$ is the Kronecker delta. 
		\item The linear operators $\nabla_t, D_t$ preserve the defining relations of $\widetilde{\mathcal{A}}$.
	\end{enumerate}
\end{proposition}
\begin{proof}
	For part (1), note that 
	\[
	\begin{aligned}
		\Delta(xy)&=\Delta(x)\Delta(y)=(1\otimes x+ \sum_{t\in T}\alpha_t\otimes \nabla_t(x)+\cdots )(1\otimes y+ \sum_{t\in T}\alpha_t\otimes \nabla_t(y)+\cdots )\\
		&=1\otimes xy+ \sum_{t\in T} \alpha_t\otimes (t.x) \nabla_t(y)+ \sum_{t\in T} \alpha_t\otimes  \nabla_t(x) y+\cdots.
	\end{aligned}
	\]
	It follows that $\nabla_t(xy)= \nabla_t(x)y+ (t.x) \nabla_t(y)$. Similarly, using the expression $\Delta(x)= x\otimes 1+ \sum_{t\in T}D_t(x) \otimes \alpha_t+ \cdots$, we have 
	\[ \begin{aligned}
		\Delta(xy)&=\Delta(x)\Delta(y)=(x\otimes 1 + \sum_{t\in T} D_t(x)\otimes \alpha_t +\cdots )(y\otimes 1+ \sum_{t\in T} D_t(y)\otimes \alpha_t +\cdots )\\
		&=xy \otimes 1+ \sum_{t\in T} x D_t(y) \otimes \alpha_t + \sum_{t\in T} D_t(x) y\otimes |y|^{-1}.\alpha_t +\cdots.
	\end{aligned}
	\]
	Note that   \[
	\begin{aligned}
		\sum_{t\in T} D_t(x) y\otimes |y|^{-1}.\alpha_t  &= \sum_{t\in T} (-1)^{\ell(|y|)} D_{t}(x)y \otimes \alpha_{|y|^{-1}t |y|} \\
		&  = \sum_{t\in T} (-1)^{\ell(|y|)} D_{|y|t |y|^{-1}}(x)y \otimes \alpha_{t}
	\end{aligned}  
	\]
	Therefore,  we have $D_t(xy)= (|y|.D_t)(x)y+ x D_t(y)$. Part (2) is a consequence of part (1), and part (3) follows immediately from the formulae in part (2). 
\end{proof}

\begin{remark}
	The linear operators $\nabla_{t},D_t$ are called skew-derivations of the braided Hopf algebra $\widetilde{\mathcal{A}}$ \cite{AG99, AS02}. 
\end{remark}

\begin{proposition}\label{prop: AtildeDel}
The skew-derivations $\nabla_{t}, t\in T$ satisfy the following relations: 
	\[
	\begin{aligned}
		\nabla_{t}^2=0, \quad &\text{$\forall t\in T$}, \\
		\sum_{(t_1,t_2)\in {\rm Rex}_T(w)} \nabla_{t_1}\nabla_{t_2}=0, \quad &   \text{$\forall w\in W$ with $\ell_{T}(w)=2$},
	\end{aligned}
	\]
	Therefore, they describe an action of $\widetilde{\mathcal{A}}$ on itself. 
\end{proposition} 
\begin{proof}
	We evaluate these relations on $x\in \widetilde{\mathcal{A}}$ and  use induction on the $\mathbb{Z}$-degree of $x$. It is trivial if ${\rm deg}(x)=1$. In general, assume that $x=\alpha_{t_1}\dots \alpha_{t_k}$. For any $r_1,r_2\in T$,  using the formula from part (1) of \propref{prop: D} we have 
  \[
	\begin{aligned}
	\nabla_{r_1}\nabla_{r_2}(\alpha_{t_1}\dots \alpha_{t_k})&= \nabla_{r_1}( \nabla_{r_2}(\alpha_{t_1}\alpha_{t_2}\dots \alpha_{t_{k-1}})\alpha_{t_k}+ (-1)^{k-1} \delta_{r_2,t_k} \alpha_{t_1^{t_k}}\alpha_{t_2^{t_k}}\dots \alpha_{t_{k-1}^{t_k}} ) \\
			&=  \nabla_{r_1}( \nabla_{r_2}(\alpha_{t_1}\dots \alpha_{t_{k-1}})) \alpha_{t_k} + (-1)^{k-2} \delta_{r_1,t_k} \nabla_{r_2^{t_k}}(\alpha_{t_1^{t_k}}\dots \alpha_{t_{k-1}^{t_k}} ) \\
			&+ (-1)^{k-1} \delta_{r_2,t_k} \nabla_{r_1}( \alpha_{t_1^{t_k}}\alpha_{t_2^{t_k}}\dots \alpha_{t_{k-1}^{t_k}} ),
		\end{aligned}
  \]
  If $r_1=r_2=t$, then by the induction hypothesis we have $\nabla_{r_1}\nabla_{r_2}(\alpha_{t_1}\dots \alpha_{t_k})=0$, proving the first relations. 

  For the second relation, by the induction hypothesis   it is equivalent to proving that 
	\[\sum_{(r_1,r_2)\in {\rm Rex}_T(w)} \delta_{r_1,t_k} \nabla_{t_kr_2t_k}- \delta_{r_2,t_k}\nabla_{r_1}=0\]
	for any  $w\in W$ with $\ell_{T}(w)=2$. 
If $t_{k}\not\prec w$, then the above equation holds trivially. Otherwise, we have two $T$-reduced expressions $w=t_kt=(t_ktt_k)t_k$ for $t=t_{k}^{-1}w \in T$, which leads to the above equation. 
\end{proof}

We do not know whether this action of $\wt\CA$ on itself faithful. Compare \cite[\S 9]{FK99}. Next we  define a bilinear form on $\widetilde{\mathcal{A}}$ in terms of the skew-derivations. 
\begin{definition}\label{def: bilinear}
	Define the bilinear pairing
	\begin{equation*}
		\langle-,-\rangle: \widetilde{\mathcal{A}}\times \widetilde{\mathcal{A}} \longrightarrow \mathbb{C}
	\end{equation*}
	by $\langle 1,1\rangle =1$ and 
	\begin{enumerate}
		\item $\langle \widetilde{\mathcal{A}}_k,\widetilde{\mathcal{A}}_{\ell}\rangle =0$ for any $0\leq k\neq \ell$;
		\item For any $x\in \widetilde{\mathcal{A}}_k$ and $t_i\in T, i=1,\dots,k$, 
		\[\langle \alpha_{t_1}\alpha_{t_2}\dots \alpha_{t_k}, x\rangle:= \nabla_{t_1} \nabla_{t_2}\dots \nabla_{t_k} (x). \]
	\end{enumerate}
\end{definition}

The bilinear form on $\widetilde{\mathcal{A}}$ is well-defined in view of \propref{prop: AtildeDel}. Note that we do not reverse the order of $\nabla_{t_i}$ in the above definition. This is different from that in \defref{def: bilinear}. However,  one can  prove the following the properties which are similar to those given for $\mathcal{A}$.

\begin{proposition}\label{prop: nabd}
	We have the following properties.
	\begin{enumerate}
		\item For any $t\in T$ and $x\in \widetilde{\mathcal{A}}$, we have 
		\[
		\begin{aligned}
			\nabla_t(x)= & \sum_{(x)} \langle \alpha_{t}, x_{(1)} \rangle\,  x_{(2)}, \\
			D_t(x) =& \sum_{(x)}  x_{(1)}\, \langle x_{(2)}, \alpha_t \rangle,
		\end{aligned}
		\]
		where we have used the Sweedler's notation $\Delta(x)=\sum_{(x)} x_{(1)}\otimes x_{(2)}$ for the comultiplication of $\widetilde{\mathcal{A}}$.
		\item For any $t,t'\in T$, we have $\nabla_{t}D_{t'}=D_{t'}\nabla_{t}$. 
		\item   For any $x,y\in \widetilde{\mathcal{A}}$, we have 
		\[
		\langle x \alpha_t, y\rangle =  \langle x, \nabla_t(y) \rangle, \quad \text{and} \quad 
		\langle  \alpha_t x , y\rangle =   \langle x, D_t(y) \rangle. 
		\]
		Therefore, with respect to the bilinear form the skew-derivations $\nabla_{t}$ and $D_{t}$ are right adjoint to right and left multiplication by $\alpha_{t}$, respectively. 
	\end{enumerate}
\end{proposition} 
\begin{proof}
	Part (1) follows from the definitions of  $\nabla_t$,  $D_t$ (see \eqref{def: nabla}) and the bilinear form. For part (2), for any $x\in \widetilde{\mathcal{A}}$ we have 
	\[
	\nabla_{t}D_{t'}(x)= \nabla_{t}( \sum_{(x)}  x_{(1)}\, \langle x_{(2)}, \alpha_{t'} \rangle)= \sum_{(x)} \langle \alpha_{t}, x_{(1)} \rangle\,  x_{(2)} \langle x_{(3)}, \alpha_{t'} \rangle.
	\]
	Similarly, we can express $D_{t'}\nabla_{t}(x)$ and obtain that $\nabla_{t}D_{t'}=D_{t'}\nabla_{t}$.  The proof of part (3) is similar to that of \lemref{lem: adj}.
\end{proof} 

\begin{remark}
	We do not know whether the bilinear form on $\widetilde{\mathcal{A}}$ is non-degenerate. Note that by \lemref{lem: surj} $\mathcal{A}$ and its opposite $\mathcal{A}^{op}$ can be lifted to $\widetilde{\mathcal{A}}$ as  vector spaces. It follows from \propref{prop: bilinear} that the restriction $\langle -, -\rangle: \mathcal{A}^{op} \times \mathcal{A} \rightarrow \mathbb{C}$ is non-degenerate.
	
\end{remark}

We define
\[\widetilde{\omega}:=\sum_{t\in T} \alpha_t. \]
By the defining relations of $\widetilde{\mathcal{A}}$ we have $\widetilde{\omega}^2=0$.
Hence  $(\widetilde{\mathcal{A}},r_{\widetilde{\omega}})$ (resp. $(\widetilde{\mathcal{A}},\ell_{\widetilde{\omega}})$) is a cochain complex, where $r_{\widetilde{\omega}}$ (resp. $\ell_{\widetilde{\omega}}$) is given by right (resp. left) multiplication by $\widetilde{\omega}$.


\begin{proposition} We have the following:
	\begin{enumerate}
		\item Let $\nabla= \sum_{t\in T}\nabla_t$ and $D=\sum_{t\in T}D_t$. Then we have 
		\[
		\langle x \widetilde{\omega}, y\rangle =  \langle x, \nabla(y) \rangle, \quad \text{and} \quad 
		\langle  \widetilde{\omega}x , y\rangle =   \langle x, D(y) \rangle. 
		\]
		\item The  complexes $(\widetilde{\mathcal{A}}, D)$ and $(\widetilde{\mathcal{A}}, r_{\widetilde{\omega}} )$  are acyclic.
		\item The complexes $(\widetilde{\mathcal{A}}, \nabla)$and $(\widetilde{\mathcal{A}}, \ell_{\widetilde{\omega}})$ are acyclic.
	\end{enumerate}
\end{proposition}
\begin{proof}
	Part (1) is a consequence of \propref{prop: nabd}.
	For part (2), we have 
	\[ Dr_{\widetilde{\omega}}(x)= D(x\widetilde{\omega})=\sum_{t\in T} D(x\alpha_t)=\sum_{t\in T}(-D(x)\alpha_t+x)= -r_{\widetilde{\omega}}D(x)+Nx.\]
	Therefore, we have $Dr_{\widetilde{\omega}}+ r_{\widetilde{\omega}} D= N {\rm id}$, which implies that $(\widetilde{\mathcal{A}}, D)$ and $(\widetilde{\mathcal{A}}, r_{\widetilde{\omega}} )$ are acyclic and part (2) follows. Part (3) can be proved similarly. 
\end{proof}

\appendix 
\section{Computational results on the multiplicity}

In this appendix, we tabulate some computational results on the cohomology $H^k(\mathcal{K}^*(U))$ for the simple $\mathbb{C}W$-module $U$, which by \thmref{thm:  multcohom} counts the multiplicity of the contragredient $U_L^*$  in  the cohomology $H^k(F; \mathbb{C})$ of the Milnor fibre. The homology $H_k(\mathcal{K}(U))$ returns the same result; see \thmref{thm: mult}.  All calculations are done with the computational algebra system Magma.

We only focus on the case of the symmetric group $W={\rm Sym}_{n+1}$. Then the Milnor fibre $F$ is an algebraic variety defined by 
\[
  F:=\{ (x_1, \dots x_{n+1})\in \mathbb{C}^{n+1}\mid \prod_{1\leq i<j\leq n+1}(x_i-x_j)^2 =1\}. 
\]
The reduced Milnor fibre $F_0$ is defined by 
\[  F_0:=\{ (x_1, \dots x_{n+1})\in \mathbb{C}^{n+1}\mid \prod_{1\leq i<j\leq n+1}(x_i-x_j) =1\}. \]
The symmetric group ${\rm Sym}_{n+1}$ acts on $F$ by permuting coordinates. Hence it induces a linear (left) action on the cohomology $H^k(F; \mathbb{C})$ for  $0\leq k\leq n-1$. As vector spaces, $H^{k}(F; \mathbb{C})\cong H^{k}(F_0; \mathbb{C})\oplus H^{k}(F_0; \mathbb{C})$; see \cite{DL16}.

The simple right modules  $S_{\lambda}$ of ${\rm Sym}_{n+1}$ are indexed by partitions $\lambda=(\lambda_1^{m_1}, \dots ,\lambda_{p}^{m_p})$ of $n+1$, where $\lambda_i^{m_i}$ means that $\lambda_i$ repeats $m_i$ times and $\sum_{i=1}^p m_i\lambda_i=n+1$. Let $\lambda'$ be the conjugate partition of $\lambda$. Then we have $S_{\lambda'}\cong S_{\lambda}\otimes \epsilon$, where $\epsilon$ is the alternating representation associated to $(1^{n+1})$. It follows from \eqnref{eq: homK} that 
\[ H^k(\mathcal{K}^*(S_{\lambda}))\cong H^k(\mathcal{K}^*(S_{\lambda'})), \quad 0\leq k\leq n-1  \] 
for any conjugate pair $\lambda, \lambda'$ of partitions.

Let $(S_{\lambda})_L$ denote the simple left module  of ${\rm Sym}_{n+1}$ associated to the partition $\lambda$. It is well known that the contragredient $(S_{\lambda})^*_L$ is isomorphic to $(S_{\lambda})_L$ as left ${\rm Sym}_{n+1}$-module. Therefore, using \thmref{thm:  multcohom} we have 
\begin{equation}\label{eq: multy}
	\langle (S_\lambda)_L, H^k(F, \mathbb{C})\rangle = {\rm dim}\,  H^k(\mathcal{K}^*(S_{\lambda})).
\end{equation}
The Poincar\'e polynomial $P(t)$ of the Milnor fibre $F$ can be computed by 
\begin{equation}\label{eq: poin}
	 P(t)= \sum_{k=0}^{n-1} {\rm dim} H^k(F; \mathbb{C}) t^k= \sum_{k=0}^{n-1} \sum_{\lambda\vdash n+1}{\rm dim} S_{\lambda}\langle (S_\lambda)_L, H^k(F, \mathbb{C})\rangle t^k. 
\end{equation}
Note that the Poincar\'e polynomial $P_0(t)$ of  $F_0$  is $P_0(t)=P(t)/2$.

We tabulate computational results on \eqnref{eq: multy} and \eqnref{eq: poin} for all simple modules of ${\rm Sym}_{n+1}$ for $2\leq n \leq 7 $ in the following tables. Each conjugate pair of partitions is listed in the same row as they produce the same cohomology. 

\begin{remark}
Note that the results tabulated here are consistent with those appearing in \cite{DL16}, up to type $A_4$, where $W=\Sym_5$.
However the cases computed in {\it loc. cit.} include the action of the monodromy group on the cohomology, in the sense that the
 structure of $H^*(F,\C)$is described as a $\Gamma$-module, where $\Gamma=\Sym_{n+1}\times \mu_{n(n+1)}$. In the present work, although the original
Brady-Falk-Watt model of $F$ does come with an action of the monodromy on the CW complex describing $F$ (see \cite{BFW18} or \cite[Section 3.4]{Zha20}), our analysis of the model
has not been able to preserve the monodromy action, except to the following extent. In general, the group $\langle\gamma\rangle$
acts (like any subgroup of $W$) on $F$, and hence on $H^*(F)$. But it is known that $\langle\gamma\rangle$ may be identified with a quotient
(or subgroup) of the monodromy $\mu$, and this action is easily identifiable in our model \cite[Remark 6.4]{Zha22}.
\end {remark}

\begin{remark}
 Settepanella computed the cohomology $H^{k}(PB_{n+1}, \mathbb{Q}[q,q^{-1}])$ of the pure braid group $PB_{n+1}$ with coefficients in the Laurent polynomial ring $\mathbb{Q}[q,q^{-1}]$ for $n\leq 7$ \cite[Table 2]{Set09}. This is related to the Milnor fibre $F_0$ by
\[ H^{k+1}(PB_{n+1}, \mathbb{Q}[q,q^{-1}])\cong H^k(F_0, \mathbb{Q}).\]
Hence  Settepanella's results give rise to the Poincar\'e polynomials $P_0(t)$ in type $A_n$ for $n\leq 7$. 
The Poincar\'e polynomials $P_0(t)=P(t)/2$ given in the tables below coincide with those of  Settepanella up to type $A_6$. However, for type $A_7$ our cohomology groups $H^k(F_0; \mathbb{Q})$ agree with Settepanella's except for $k=5,6$,  but the Euler characteristic remains the same.
\end{remark}

\begin{remark}
Apart from the monodromy action, the major issue untouched by our work is the mixed Hodge structure on each cohomology group
$H^i(F,\C)$. We hope to return to this theme later.
\end{remark}

\begin{table}[h]
\begin{tabular}{c|cc} \hline
         & $H^0$         & $H^1$       \\ \hline
$(3), (1^3) $   & 1       & 2       \\ \hline
(2,1)    & 0        & 2 \\ \hline
\end{tabular}
\vspace*{2mm}
\caption{${\rm Sym}_3$, $P(t)= 2+8t$}
\end{table}

\begin{table}[h]
\begin{tabular}{c|ccc} \hline
         & $H^0$         & $H^1$    & $H^2$    \\ \hline
$(4),(1^4)$     & 1       & 2   & 2    \\ \hline
$(3,1),(2,1,1)$    & 0        & 1 & 4\\ \hline
$(2,2)$  & 0        & 2   &4     \\ \hline
\end{tabular}
\vspace*{2mm}
\caption{${\rm Sym}_4$, $P(t)=2+ 14t+ 36t^2$}
\end{table}

\begin{table}[h]
\begin{tabular}{c|cccc} \hline
         & $H^0$         & $H^1$    & $H^2$ & $H^3$    \\ \hline
$(5), (1^5)$     & 1       & 0   & 2 & 4   \\ \hline
$(4,1), (2,1^3)$    & 0        & 1 & 1 & 4\\ \hline
$(3,2), (2^2,1) $  & 0        & 1   &2 & 6     \\ \hline
$(3,1,1)$ & 0        & 0   &4 & 10     \\ \hline
\end{tabular}
\vspace*{2mm}
\caption{${\rm Sym}_5$, $P(t)=2+ 18t+ 56t^2+160t^3$}
\end{table}

\begin{table}[h]
\begin{tabular}{c|ccccc} \hline
         & $H^0$         & $H^1$    & $H^2$ & $H^3$ & $H^4$    \\ \hline
$(6), (1^6)  $   & 1       & 0   & 2 & 4 & 2  \\ \hline
$(5,1), (2,1^4)$    & 0        & 1 & 1 & 3 & 8 \\ \hline
$(4,2), (2^2,1^2)$   & 0        & 1   &1 & 6  &15   \\ \hline
$(3,3), (2,2,2) $  & 0        & 0   &1 & 7 &11     \\ \hline
$(4,1,1), (3,1^3)$  & 0        & 0   &2 & 5 &13     \\ \hline
$(3,2,1) $ & 0        & 0   &4 & 6 &18     \\ \hline
\end{tabular}
\vspace*{2mm}
\caption{${\rm Sym}_6$, $P(t)= 2+ 28t+ 146t^2+ 412 t^3+  1012t^4$}
\end{table}

\begin{table}[hbt!]
\begin{tabular}{c|cccccc} \hline
         & $H^0$         & $H^1$    & $H^2$ & $H^3$ & $H^4$ & $H^5$    \\ \hline
$(7), (1^7) $    & 1       & 0   & 2 & 0 & 2 & 6  \\ \hline
$(6,1), (2,1^5)$    & 0        & 1 & 1 & 3 & 3 &6 \\ \hline
$(5,2), (2^2,1^3) $  & 0        & 1   &1 & 4  &8 &18   \\ \hline
$(5,1,1), (3,1^4)$   & 0        & 0   &2 & 3 &10 &24     \\ \hline
$(4,3), (2^3,1)$  & 0        & 0   &1 & 4 &7 & 18     \\ \hline
$(4,2,1), (3,2,1^2)$  & 0        & 0   &2 & 8 &17 &46     \\ \hline
$(3,3,1), (3,2^2)$ & 0        & 0   &1 & 5 &11 &28     \\ \hline
$(4,1^3)$  & 0        & 0   &0 & 6 &8 &22     \\ \hline
\end{tabular}
\vspace*{2mm}
\caption{${\rm Sym}_7$, $P(t)= 2+ 40t+314t^2+ 1240 t^3+ 2572 t^4+ 6648 t^5$}
\end{table}

\begin{table}[hhbt!]
\begin{tabular}{c|ccccccc} \hline
         & $H^0$         & $H^1$    & $H^2$ & $H^3$ & $H^4$ & $H^5$ & $H^6$    \\ \hline
$(8), (1^8)$   & 1       & 0   & 0 & 0 & 2 & 6 &4 \\ \hline
$(7,1), (2,1^6) $   & 0        & 1 & 1 & 1 &1 & 3 &10 \\ \hline
$(6,2),(2^2,1^4)$   & 0        & 1   &1 & 2  &4 &10 &28   \\ \hline
$(6,1^2), (3,1^5)$  & 0        & 0   &2 & 3 &3 &7 &26     \\ \hline
$(5,3), (2^3,1^2)$  & 0        & 0   &1 & 3 &6 & 10 &34    \\ \hline
$(5,2,1), (3,2,1^3)$  & 0        & 0   &2 & 5 &16 &25 &76     \\ \hline
$(5,1^3), (4,1^4)$  & 0      &0  & 0    & 3 &10 &22 &50     \\ \hline
$(4,4), (2^4)$  & 0        & 0   &0 & 1 &4 &19& 30     \\ \hline
$(4,3,1), (3,2^2,1)$  & 0        & 0   &1 & 6 &18 &27& 84     \\ \hline
$(4,2^2), (3^2,1^2)$  & 0        & 0   &0 & 3 &16 &31&74     \\ \hline
$(4,2,1^2)$  & 0        & 0   &0 & 8 &22 &36 &112     \\ \hline
$(3,3,2)$  & 0        & 0   &0 & 4 &10 &16& 52     \\ \hline
\end{tabular}
\vspace*{2mm}
\caption{${\rm Sym}_8$, $P(t)= 2+54t+590t^2+ 3330t^3+10212t^4+17744t^5+50644t^6$}
\end{table}


\end{document}